\documentclass{article}[9pt]

\usepackage{lipsum}
\usepackage{cite}
\usepackage{url}
\usepackage{amsfonts}
\usepackage{graphicx}
\usepackage{epstopdf}
\usepackage{algorithmic}
\usepackage{amsopn}
\usepackage{enumitem}
\usepackage{tikz}
\usepackage{pgfplots}
\pgfplotsset{compat=newest}
\usetikzlibrary{plotmarks,patterns}
\usetikzlibrary{arrows.meta}
\usepgfplotslibrary{patchplots}
\usepackage{grffile}
\usepackage{amsmath}
\usepackage{mathtools}
\usepackage{amssymb}
\usepackage{amsthm}
\usepackage{subcaption}
\usepackage{esint}
\usepackage{enumitem}

\RequirePackage[capitalize,nameinlink]{cleveref}[0.19]
\crefname{section}{section}{sections}
\crefname{subsection}{subsection}{subsections}
\Crefname{figure}{Figure}{Figures}
\crefformat{equation}{\textup{#2(#1)#3}}
\crefrangeformat{equation}{\textup{#3(#1)#4--#5(#2)#6}}
\crefmultiformat{equation}{\textup{#2(#1)#3}}{ and \textup{#2(#1)#3}}
{, \textup{#2(#1)#3}}{, and \textup{#2(#1)#3}}
\crefrangemultiformat{equation}{\textup{#3(#1)#4--#5(#2)#6}}%
{ and \textup{#3(#1)#4--#5(#2)#6}}{, \textup{#3(#1)#4--#5(#2)#6}}{, and \textup{#3(#1)#4--#5(#2)#6}}
\Crefformat{equation}{#2Equation~\textup{(#1)}#3}
\Crefrangeformat{equation}{Equations~\textup{#3(#1)#4--#5(#2)#6}}
\Crefmultiformat{equation}{Equations~\textup{#2(#1)#3}}{ and \textup{#2(#1)#3}}
{, \textup{#2(#1)#3}}{, and \textup{#2(#1)#3}}
\Crefrangemultiformat{equation}{Equations~\textup{#3(#1)#4--#5(#2)#6}}%
{ and \textup{#3(#1)#4--#5(#2)#6}}{, \textup{#3(#1)#4--#5(#2)#6}}{, and \textup{#3(#1)#4--#5(#2)#6}}
\crefdefaultlabelformat{#2\textup{#1}#3}

\allowdisplaybreaks

\DeclareMathOperator{\R}{\mathbb{R}}

\newcommand{\norm}[1]{\ensuremath{\left\lVert#1\right\rVert}}
\newcommand{\abs}[1]{\ensuremath{\left\lvert#1\right\rvert}}
\newcommand{\set}[1]{\ensuremath{\left\lbrace#1\right\rbrace}}

\newcommand{\bx}{\ensuremath{\mathbf{x}}}
\newcommand{\by}{\ensuremath{\mathbf{y}}}

\newtheorem{theorem}{Theorem}[section]
\newtheorem{corollary}{Corollary}[section]
\newtheorem{lemma}{Lemma}[section]
\newtheorem{proposition}{Proposition}[section]
\newtheorem{definition}{Definition}[section]

\newtheorem{remark}{Remark}[section]

\graphicspath{{img/}}	
\newcommand{\corr}[1]{#1}
\newcommand{\add}[1]{#1}
\newcommand{\newadd}[1]{#1}
\let\refcite\citen

\def\keywords#1{\par
	\vspace*{8pt}
	{{\footnotesize\leftskip18pt\rightskip\leftskip
			\noindent{\it Keywords}\/:\ #1\par}}\par}
\def\ccode#1{\par
	\vspace*{8pt}
	{{\footnotesize\leftskip18pt\rightskip\leftskip
			\noindent #1\par}}\par}
\makeatletter 
\renewcommand\@biblabel[1]{#1.} 
\makeatother

\title{A parabolic local problem with exponential decay of the resonance error for numerical homogenization}
\author{%
	Assyr Abdulle\footnote{Institute of Mathematics, \'Ecole Polytechnique F\'ed\'erale de Lausanne, 1015, Station 8, Lausanne, CH-1015, Switzerland, {assyr.abdulle@epfl.ch}} %
	\and Doghonay Arjmand\footnote{Division of Applied Mathematics, School of Education, Culture and Communication, M{\"a}lardalen University, Box 883, 721 23, V{\"a}ster{\aa}s, Sweden, doghonay.arjmand@mdh.se} %
	\and Edoardo Paganoni\footnote{Institute of Mathematics, \'Ecole Polytechnique F\'ed\'erale de Lausanne, Station 8, Lausanne, CH-1015, Switzerland, {edoardo.paga@gmail.com}} }
\date{}
		
\begin{document}
	\maketitle
	
%
\begin{abstract}
	This paper aims at an accurate and efficient computation of effective quantities, e.g., the homogenized coefficients for approximating the solutions to partial differential equations with oscillatory coefficients. Typical multiscale methods are based on a micro-macro coupling, where the macro model describes the coarse scale behaviour, and the micro model is solved only locally to upscale the effective quantities, which are missing in the macro model. The fact that the micro problems are solved over small domains within the entire macroscopic domain, implies imposing artificial boundary conditions on the boundary of the microscopic domains. A naive treatment of these artificial boundary conditions leads to a first order error in $\varepsilon/\delta$, where $\varepsilon < \delta$ represents the characteristic length of the small scale oscillations and $\delta^d$ is the size of micro domain. This error dominates all other errors originating from the discretization of the macro and the micro problems, and its reduction is a main issue in today's engineering multiscale computations. The objective of the present work is to analyse a parabolic approach, \corr{first announced in [A.~Abdulle, D.~Arjmand, E.~Paganoni, C. R. Acad. Sci. Paris, Ser. I, 2019]}, for computing the homogenized coefficients with arbitrarily high convergence rates in $\varepsilon/\delta$. The analysis covers the setting of periodic microstructure, and numerical simulations are provided to verify the theoretical findings for more general settings, e.g. \add{non-periodic} micro structures.
%
\end{abstract}




\section{Introduction}\label{sec:intro}

Multiscale problems involving several spatial and temporal scales are ubiquitous in physics and engineering. We mention for example, stiff stochastic differential equations (SDEs) in biological and chemical systems, oscillatory physical systems, partial differential equations (PDEs) with multiscale data resonance, e.g., mechanics of composite materials, fracture dynamics of solids, PDEs with oscillating parameters, see Ref. \refcite{Abd09a,E11,HeM14,HFM98,MaP14} and the references therein. 
A common computational challenge in relation with such multiscale problems is the presence of small scales in the model which should be represented over a much larger macroscopic scale of interest. One rather classical way of overcoming this issue is to analytically derive macroscopic equations from a given microscopic model, and then solve the resulting macroscale equation at a cheaper computational cost. However, such derivations often come together with some simplifying assumptions, making the accuracy of the macroscopic model questionable once the restrictive assumptions are relaxed. \corr{In contrast}, multiscale numerical methods result in models with improved accuracy and efficiency as they rely on a coupling between microscopic and macroscopic models, combining the efficiency of macroscopic models with the accuracy of microscopic ones. Inexact couplings may afflict such methods by the so-called \emph{resonance error}, Ref. \refcite{Abd09a,HoW97}. Reducing such an error is a common problem of modern multiscale methods designed over the last two decades.
	

This paper concerns the numerical homogenization of elliptic partial differential equations with multiscale coefficients, whose oscillation length scale (denoted by $\varepsilon$) is much smaller than the size of the domain $\Omega \subset \R^d$, which is bounded and convex. Our model problem is the following $\varepsilon$-indexed family of elliptic equations on $\Omega$

\begin{equation}
\label{eq:model problem}
\left\lbrace
\begin{aligned}
-\nabla \cdot \left( a^{\varepsilon}(\bx) \nabla u^{\varepsilon} \right) &= f & \quad  & \text{in }\Omega \\
u^{\varepsilon} &= 0 & \quad & \text{on } \partial \Omega .
\end{aligned}
\right.
\end{equation}
Here $a^{\varepsilon}(\bx) \in \left[ L^{\infty} (\Omega) \right]^{d \times d}$ is symmetric, uniformly elliptic and bounded, i.e., $\exists \, \alpha,\, \beta > 0$ such that
\begin{equation}\label{eq: continuity coercivity assumption}
\alpha \abs{ \zeta }^2 \le \zeta \cdot a^{\varepsilon}(\bx) \zeta \le \beta \abs{ \zeta }^2 ,\, \forall \zeta \in \R^d ,\, \text{ a.e. } \bx \in \Omega ,\, \forall \varepsilon>0.
\end{equation}
The well-posedness of the original problem \cref{eq:model problem} is then well-known for any $f  \in H^{-1}(\Omega)$. As ${\varepsilon\to 0}$, the solution of \cref{eq:model problem} can be approximated, by the solution of the so-called homogenized equation:
\begin{equation} \label{eq:homogenized elliptic}
\left\lbrace
\begin{aligned}
-\nabla \cdot \left( a^{0}(\bx) \nabla u^{0} \right) &= f & \quad & \text{in }\Omega \\
u^{0} &= 0 & \quad & \text{on } \partial \Omega ,
\end{aligned}
\right.
\end{equation}
where the coefficients $a^0_{ij}$ (and hence the solution $u^0$)  no longer oscillate at the $\varepsilon$-scale. By using the concepts of $G$-convergence for the symmetric case, Ref. \refcite{Spa68}, 
one can show that the homogenized problem \cref{eq:homogenized elliptic} is the limit for $\varepsilon \to 0$ of a subsequence of problems \cref{eq:model problem}. In general, we do not have explicit formulae for evaluating the homogenized tensor, unless certain structural assumptions on $a^{\varepsilon}(\bx)$ are made. For example, if $a^{\varepsilon}(\bx) = a(\bx/\varepsilon)$  is \add{$K$-}periodic, then the homogenized tensor \corr{$a^0$} is given by
\begin{equation} \label{eq:a0}
 \mathbf{e}_i \cdot a^0  \mathbf{e}_j = \fint_K \mathbf{e}_i \cdot a(\bx) \left( \mathbf{e}_j + \nabla \chi^j (\bx) \right) \,d\bx ,\quad i,j = 1,\dots,d,
\end{equation}
where $K:=(-1/2,1/2)^d$ is the unit cube in $\mathbb{R}^{d}$, and the functions $\lbrace \chi^i \rbrace_{i=1}^d$ are the solutions of the so-called \emph{cell problems}:
\begin{equation} \label{eq:periodic micro}
\left\lbrace
\begin{aligned}
- &\nabla\cdot\left(a(\bx)\nabla \chi^i \right) = \nabla\cdot\left(a(\bx) \mathbf{e}_i\right) \qquad \text{in } K, \\
& \chi^i \quad K\text{-periodic}.
\end{aligned}
\right. 
\end{equation}
In \eqref{eq:a0} and \eqref{eq:periodic micro} we have used the substitution $\mathbf{y}=\frac{\bx}{\varepsilon}$, mapping a sampling domain  of size $\varepsilon^d$ to the unit cube $K$. For simplicity of notation, we will again denote by $\bx$ (instead of $\mathbf{y}$) the variable on the unit cube.
We refer to Ref. \refcite{BLP78,CiD99,JKO94} for further technical details. 

When the period of the microstructure is not known exactly or the periodicity assumption is relaxed (e.g., if $a$ is random stationary ergodic or quasi-periodic \corr{tensor}), the formula \eqref{eq:a0} breaks down. In this case, \eqref{eq:a0} may be replaced by 
 \begin{equation} \label{eq:bar_a0_Intro}
 \mathbf{e}_i\cdot  a^{0}_{R} \mathbf{e}_j = \fint_{K_R} \mathbf{e}_i \cdot a(\bx) \left( \mathbf{e}_j + \nabla \chi_R^j(\bx) \right) \,d\bx, \quad i,j = 1,\dots,d,
 \end{equation}  
where $K_R :=(-R/2,R/2)^d$, and\footnote{In \eqref{eq: microproblem_Intro}, the choice of the boundary condition (BC) is not unique, and the homogeneous Dirichlet BCs can be safely replaced e.g., by periodic BCs without any change in \eqref{eq:ModellingError}.} 
 \begin{equation}
\label{eq: microproblem_Intro}
 \left\lbrace
 \begin{aligned}
 -&\nabla\cdot\left(a(\bx)\nabla \chi_R^i \right) = \nabla\cdot\left(a(\bx) \mathbf{e}_i\right) & \quad  & \text{in } K_R \\
 &\chi_R^i(\bx) = 0  &\quad & \text{on } \partial K_{R},
 \end{aligned}
 \right.
 \end{equation}
and the homogenized coefficient $a^{0}$ is given by 
\begin{equation*}
a^{0}  = \lim_{R \to \infty } a^{0}_{R}.
\end{equation*}
Assume for a moment that \corr{the tensor} $a$ is $K$-periodic and that periodic BCs are imposed in \eqref{eq: microproblem_Intro}, then the homogenized \corr{tensor} $a^{0}$ will be equal to $a^{0}_{R}$ only when $R$ is an integer. When $R$ is not an integer, there will be a difference between $\chi_R$ and $\chi$ on $\partial K_R$, which results in a so-called resonance error, Ref. \refcite{AEE12,EE03,EMZ05}\footnote{We denote the resonance error by $e_{MOD}$, as in Ref. \refcite{AEE12}.}, 
 \begin{equation} \label{eq:ModellingError}
 e_{MOD} :=  \norm{  a^{0}_{R} - a^0 }_{F} \leq C \dfrac{1}{R},
 \end{equation} 
 where $\norm{  \cdot }_{F}$ denotes the Frobenius norm for a tensor. 
 Note that the first order rate is valid also when the problem \eqref{eq: microproblem_Intro} is equipped with periodic BCs. From a computational point of view, this first order decay rate of the error is the efficiency and accuracy bottleneck of numerical upscaling schemes, i.e., in order to reduce the resonance error down to practically reasonable accuracies, one needs to solve the problem \eqref{eq: microproblem_Intro} over large computational domains $K_R$, possibly on each quadrature point of a macro computational domain (see Ref. \refcite{AEE12,EE03}), which becomes prohibitively expensive. Our central goal is then to design micro models which reduce the resonance error down to desired accuracies without requiring a substantial enlargement of the computational domain $K_R$. 
 \add{ In \cref{subsec: existing approaches}}, we provide a review of existing strategies, some of which improve the decay rate of the resonance error.
  
 \add{In this work, we restricted ourselves to the analysis of the resonance error for the parabolic local problems under the assumption that the coefficients are periodic and oscillate only at the $\varepsilon$-scale, $a^{\varepsilon}(\bx)= a(\bx/\varepsilon)$. 
 The analysis can be extended to the case of locally periodic media, $a^{\varepsilon}(\bx)=a(\bx,\bx/\varepsilon)$, where the coefficients are $\varepsilon$-periodic in the second argument. In such a case, the parabolic micro problems can still be employed to approximate the homogenized coefficients in the quadrature points of the macro computational domain at the cost of introducing a \emph{collocation} error. 
 Having to fix the first variable of $a(\bx,\by)$ in the micro model computations produces the new error term.
 The collocation error has been previously analysed in the context of elliptic cell problems, e.g. in Ref. \refcite{Abd11b}, and similar ideas can be extrapolated to the setting of the present work.   
}

\add{The present work addresses the problem of computing the homogenized limit for a single equation with oscillating coefficients. The techniques used here, in particular the use of the fundamental solution of parabolic equations, are not directly extendable to linear systems, as problems in linear elasticity. Nevertheless, we believe that the ideas presented here can provide a fairly good road-map for generalizing the parabolic approach to such more complex cases.}
\subsection{Existing approaches for reducing the resonance error}\label{subsec: existing approaches} 
Over the last two decades, several interesting approaches have been proposed to reduce the resonance error. These strategies can be classified in two classes: a) Methods which reduce the prefactor (but not the convergence rate) in \eqref{eq:ModellingError}, b) Methods which improve the convergence rate. 

\emph{a) Methods reducing the prefactor only:}

One of the very first approaches to reduce the prefactor is based on the idea of oversampling, see Ref. \refcite{HoW97}. In oversampling, the cell problem \eqref{eq: microproblem_Intro} is solved over $K_{R}$, while the computation of the homogenized coefficient takes place in an interior domain $K_{L} \subset K_{R}$. Another attempt is based on exploring the combined effect of oversampling and imposing different BCs (Dirichlet, Neumann and periodic) for \eqref{eq: microproblem_Intro}, see Ref. \refcite{YuE07}. It has been found that the periodic BCs perform better than the other two. Moreover, the Dirichlet BCs tend to overestimate the effective coefficients, while Neumann BCs underestimate them. Clearly, the use of these strategies becomes questionable if one is interested in practically relevant error tolerances, since there is still a need for substantially enlarging the computational domain $K_R$ before reaching a satisfactory accuracy. 

\emph{b) Methods improving the convergence rate:}

Several methods which rely on modifying the cell problem \eqref{eq: microproblem_Intro}, while still retaining a good approximation (with higher order convergence rates in $1/R$) of the homogenized coefficient have been developed in the last few years. In Ref. \refcite{BlL10}, an approach with weight (filtering) functions in the very definition of the cell problem, as well as in the averaging formula, is proposed. While the method has arbitrarily high convergence rates in a one-dimensional setting, the convergence rate in dimensions $d>1$ has been proved to be $2$. Numerical simulations demonstrate the optimality of the \corr{second order rate in dimension} $d>1$.

Another promising strategy
is to add a small zero-th order regularization term to the cell problem \eqref{eq: microproblem_Intro} so as to make the associated Green's function exponentially decaying.
\add{This idea was originally introduced in the context of stochastic homogenization Ref. \refcite{GlO11a,GlO12a}, and an analysis of the boundary error in the periodic case can be found in Ref. \refcite{Glo11}.}
The effect of the boundary mismatch will then decay exponentially fast in the interior of $K_{L} \subset K_{R}$. However, the method will suffer from a bias (or systematic error) due to added  regularization term, which limits the convergence rate to fourth order. Moreover, \corr{numerical simulations in Ref. \refcite{Glo11}} show that the method requires very large values of $R$ to achieve the optimal fourth order asymptotic rate.
\add{ In Ref. \refcite{GlH16,GNO15}}, Richardson extrapolation is used to increase the convergence rate to higher orders at the expense of solving the cell problem several times with different regularization terms.
	
An interesting idea, proposed in Ref. \refcite{ArR16b,ArS16}, is to solve a second order wave equation on $K_{R} \times (0,T)$, instead of the elliptic cell problem \eqref{eq: microproblem_Intro}, see also Ref. \refcite{ArR17} for an analysis in locally-periodic media. Thanks to the finite speed of propagation of waves, this approach leads to an ultimate removal of the error due to inaccurate BCs if $K_R$ is sufficiently large; i.e., the boundary values will not be seen in an interior region $K_{L}$ (where the averaging takes place) if $R > L +\sqrt{\norm{a}_{\infty}} T$. Hence, size of the computational domain should increase linearly with respect to the wave speed $\sqrt{\norm{a}_{\infty}}$, which increases also the computational cost. Moreover, solving a wave equation is computationally more challenging than solving an elliptic PDE since an accurate discretization requires a more refined resolution per wave-length, implying a more refined stepsize for temporal discretization due to the presence of CFL condition in typical time-stepping methods for the wave equation such as the leap frog scheme.  

The goal of this paper is to provide a rigorous analysis of yet another approach, announced in Ref. \refcite{AAP19a}, based on parabolic cell problems which results in arbitrarily high convergence rates in $1/R$. The parabolic approach adopted here is inspired by Ref. \refcite{Mou19} \add{(although the idea of describing corrector functions by means of semigroup was also discussed in Ref. \refcite{GlO15})} and can be classified under category b), but with significant advantages from a computational point of view in comparison to the above mentioned strategies (see the discussions in the numerical results section). Moreover, this approach can be directly used in typical upscaling based multiscale formalisms such as the Heterogeneous Multiscale Methods (HMM) Ref. \refcite{Abd05b,AEE12,AbV11b,EE03}, and the equation free approaches Ref. \refcite{KGH03}, as well as Multiscale Finite Elements Methods (MsFEM) Ref. \refcite{HoW97,HWC99}, which are used to approximate either the homogenized solutions to \eqref{eq:model problem} or directly approximating the oscillatory response $u^{\varepsilon}$ in \eqref{eq:model problem}. 

The paper in structured as follows: in \Cref{sec:notation definition} we collect our notations and provide some definitions that will be used to present a new approximation scheme for the homogenized tensor. \corr{The main results of the present work are reported in \Cref{sec:main result}.} \Cref{sec:analysis} is devoted to the analysis of the modelling error, where \corr{arbitrary high} order convergence rates are proved. In \Cref{sec:numerical experiments}, numerical examples are given to verify our \corr{theoretical findings}. Finally, in \Cref{sec:computational cost} the computational cost of the parabolic method is analysed theoretically and compared to \corr{the classical elliptic scheme}.


\section{Notations and definitions}\label{sec:notation definition}
We will use the following notations throughout the exposition:
\begin{itemize}
\item The Sobolev space $W^{k,p}(\Omega)$ is defined as 
$$W^{k,p}(\Omega):= \{ f: D^{\gamma} f \in L^{p}(\Omega) \text{ for all multi-index } \gamma \text{ with } |\gamma| \leq k\}.$$
The norm of a function $f \in W^{k,p}(\Omega)$ is given by 
\begin{align*}
\| f \|_{W^{k,p}(\Omega)} := \begin{cases} \left( \sum_{|\gamma| \leq k} \int_{\Omega} |D^{\gamma} f(\bx)|^{p} \; d\bx \right)^{1/p} & (1 \leq p < \infty) \\
\sum_{|\gamma| \leq k} \text{ess sup}_{\Omega} |D^{\gamma} f| & (p=\infty).
\end{cases}
\end{align*}
\item \add{ The space $W^{1,p}_0(\Omega)$ is the closure in the $W^{1,p}$-norm of $C_c^{\infty}(\Omega)$, the space of infinitely differentiable functions with compact support in $\Omega$. The norm associated with $W^{1,p}_0(\Omega)$ is 
	\begin{equation*}
		\| f \|^p_{W^{1,p}_0(\Omega)} := \| f \|^p_{L^p(\Omega)} + \| \nabla  f \|^p_{L^p(\Omega)} .
\end{equation*}}
\item \add{The space $W^{1,2}_0(\Omega)$ is also denoted as $H^1_0(\Omega)$. 
By making use of the Poincar\'e inequality, the norm
\begin{equation*}
	\| f \|_{H_{0}^1(\Omega)} := \| \nabla  f \|_{L^2(\Omega)}
\end{equation*}
is equivalent to $\| f \|^p_{W^{1,2}_0(\Omega)}$.
We will sometimes use this form for the $H_{0}^1$-norm.}
\item We use the notation $\langle f, g \rangle_{L^2(\Omega)} := \int_{\Omega} f g \; d\bx$ to denote the $L^2$ inner product over $\Omega$.  
\item The space $H_{{div}}$ is 
$$H_{{div}}(\Omega) := \{f: f \in [L^2(\Omega)]^{d} \text{ and } \nabla \cdot f \in L^{2}(\Omega)\}.$$
The norm associated with $H_{{div}}$ is 
\begin{align*}
\| f \|_{H_{div}(\Omega)}^{2} := \| f \|_{L^2(\Omega)}^{2} +  \| \nabla \cdot f \|_{L^2(\Omega)}^{2}. 
\end{align*}
\item The space $W^1_{per}(K)$ is defined as the closure of
	\begin{equation*}
		\left\lbrace f\in C^{\infty}_{per}(K) : \int_{K} f \,d\bx=0\right\rbrace
	\end{equation*}
	for the $H^1$-norm. Thanks to the Poincar\'e-Wirtinger inequality, an equivalent norm in $W^1_{per}(K)$ is
	\begin{equation*}
		\norm{f}_{W^1_{per(K)}} = \norm{\nabla f}_{L^2(K)}.
	\end{equation*}

\item The space $L^2_0(K)$ is defined as 
\begin{equation*}
L^2_0(K) = \left\lbrace f\in L^2(K) : \int_{K} f \,d\bx=0\right\rbrace.
\end{equation*}
It is an Hilbert space with respect to the $L^2$-inner product.
\item Let $f$ belong to the Bochner space $ L^{p}(0,T; X)$, where $X$ is a Banach space. Then the norm associated with this space is defined as 
$$
\| f \|_{L^{p}(0,T; X)}:= \left( \int_{0}^{T} \|  f \|_{X}^{p} \; dt \right)^{\frac{1}{p}}. 
$$
\item Cubes in $\mathbb{R}^{d}$ are denoted by $K_{L}:=(-L/2,L/2)^{d}$. In particular, $K$ is the unit cube $(-1/2,1/2)^{d}$. 
\item By writing $C$, we mean a generic constant independent of $R,L,T$ which may change in every subsequent occurrence.  
\item Boldface letters are to distinguish functions in multi-dimensions, e.g., $f(\bx)$ is to mean a function of several variable ($\bx \in \mathbb{R}^d, d \geq 2$), while $f(x)$ will be a function of one variable ($ x \in \mathbb{R}$).  

\item We will use the notation $\fint_{ D } f(\bx) \; d\bx$ to denote the average $\frac{1}{|D|} \int_{D} f(\bx) \; d\bx$ over a domain $D$.
\end{itemize}

\add{An important ingredient of the proposed numerical homogenization method are $q$-th order filter functions introduced in Ref. \refcite{Glo11} and defined below.}
\begin{definition}[Definition 3.1 in Ref. \refcite{Glo11}]\label{Def_Filter} We say that a function $\mu: [-1/2,1/2] \to \mathbb{R}^{+}$ belongs to the space $\mathbb{K}^{q}$ with $q > 0$ if
\begin{itemize}
\item[i)] $\mu \in C^{q}([-1/2,1/2]) \cap W^{q+1,\infty}((-1/2,1/2))$
\item[ii)] $\int_{-1/2}^{1/2}  \mu(x) \; dx  =1, $
\item[iii)] $\mu^{k}(-1/2) = \mu^{k}(1/2) =  0$ for all $k \in \{ 0,\ldots,q-1 \}$.
\end{itemize}
In multi-dimensions a $q$-th order filter $\mu_{L}: K_{L} \to \mathbb{R}^{+}$ with $L>0$ is defined by  
\begin{align*}
\mu_L(\bx) := L^{-d} \prod_{i=1}^{d} \mu\left(\frac{x_i}{L}\right),
\end{align*}
where $\mu$ is a one dimensional $q$-th order filter and $\bx = (x_1,x_2,\ldots,x_d) \in \mathbb{R}^d$. In this case, we will say that $\mu_L \in \mathbb{K}^{q}(K_L)$. Note that filters $\mu_L$ are considered extended to 0 outside of $K_L$.
\end{definition}

Filters have the property of approximating the average of periodic functions with arbitrary rate of accuracy, as stated in the following lemma (see Ref. \refcite{Glo11} for a proof).
	\begin{lemma}[Lemma 3.1 in Ref. \refcite{Glo11}]\label{lemma: kernel}
		Let $\mu_L \in \mathbb{K}^{q}(K_L)$. Then, for any $K$-periodic function $f\in L^p(K)$ with $1 < p\le 2$, we have 
		\begin{equation*} 
		\abs{ \int_{K_L}  f(\bx) \mu_L(\bx) \, d\bx - \fint_K f(\bx) \,d\bx } \le C \norm{f}_{L^{p}(K)} L^{-(q+1)},
		\end{equation*}
		where $C$ is a constant independent of $L$.
	\end{lemma}
	\begin{remark}
		\add{The constant $C$ in the above \Cref{lemma: kernel} depends on the Lipschitz constant of $\mu^{(q)}$.}
		The result of \cref{lemma: kernel} was proved in Ref. \refcite{Glo11} for $K$-periodic $f\in L^2(K)$ and, then, extended to the case $f\in L^p(K)$, $1<p<2$.
	\end{remark}


\begin{definition}\label{Def_MSpace} We say that $a \in \mathcal{M}(\alpha,\beta,\Omega)$ if $a_{ij} = a_{ji}$, $a \in [L^{\infty}(\Omega)]^{d\times d}$ and there are constants $0<\alpha \leq \beta$ such that
\begin{align*}
\alpha |\zeta|^2 \leq \zeta\cdot a(\bx) \zeta  \leq \beta |\zeta|^2, \quad \text{ for a.e.} \quad \bx \in \Omega, \forall   \zeta \in \mathbb{R}^d.
\end{align*} 
We write $a \in \mathcal{M}_{\text{per}}(\alpha,\beta,\Omega)$ if in addition $a$ is a $\Omega$-periodic function. 
\end{definition}

Throughout the exposition, we assume that $u^i$ and $v^i$, $i=1,\dots,d$, are the solutions of the following problems:
\begin{equation} \label{eq:parabolic dirichlet problem}
\left\lbrace
\begin{aligned}
& \frac{\partial u^i}{\partial t} - \nabla\cdot(a(\bx)\nabla u^i) = 0 & \quad &\text{in } K_{R}\times(0,+\infty)\\
& u^i = 0 &  \quad &\text{on }  \partial K_{R} \times (0,+\infty)\\
& u^i(\bx,0) = \nabla \cdot (a(\bx)\mathbf{e}_i) & \quad &\text{in } K_{R},
\end{aligned}
\right. 
\end{equation}
and
\begin{equation} \label{eq:parabolic periodic problem}
\left\lbrace
\begin{aligned}
& \frac{\partial v^i}{\partial t} - \nabla\cdot(a(\bx)\nabla v^i) = 0 & \quad & \text{in } K\times(0,+\infty)\\
& v^i(\cdot,t)  \quad K\text{-periodic},\forall t\ge 0\\
& v^i(\bx,0) = \nabla\cdot (a(\bx)\mathbf{e}_i) &  \quad &\text{in } K.
\end{aligned}
\right. 
\end{equation}	

The well-posedness of \cref{eq:parabolic dirichlet problem,eq:parabolic periodic problem} are well-known (see, e.g., Ref. \refcite{LiM68}), and are summarized below.
\begin{proposition}\label{prop: well-posedness}
	Let $a\in \mathcal{M}(\alpha,\beta,K_R)$ and $a(\bx)\mathbf{e}_i\in H_{\text{div}}(K_R)$. Then, \cref{eq:parabolic dirichlet problem} has a unique weak solution $u^i$ such that 
	\begin{equation*}
	u^i \in L^2([0,+\infty), H^1_{0}(K_R)), \partial_t u^i \in L^2([0,+\infty), H^{-1}(K_R)).
	\end{equation*}
	It follows that $u^i \in C([0,+\infty), L^2(K_R))$, and there exists a constant $C>0$ such that the following bound holds true:
	\begin{equation*}
	\norm{u^i}_{L^{\infty}\left([0,+\infty),L^2(K_R)\right)} + \norm{u^i}_{L^2\left( [0,+\infty), H^1_{0}(K_R)\right)} \le C \norm{\nabla\cdot (a(\bx)\mathbf{e}_i)}_{L^2(K_R)}.
	\end{equation*}
\end{proposition}
\begin{proposition}\label{prop: well-posedness per}
	Let $a\in \mathcal{M}_{per}(\alpha,\beta,K)$ and $a(\bx)\mathbf{e}_i\in H_{\text{div}}(K_R)$. Then, \cref{eq:parabolic periodic problem} has a unique weak solution $v^i$ such that 
	\begin{equation*}
	v^i \in L^2([0,+\infty), W^1_{per}(K)), \partial_t v^i \in L^2([0,+\infty), W^1_{per}(K)').
	\end{equation*}
	It follows that $v^i\in  C([0,+\infty), L_0^2(K))$, and there exist constants $C>0$ such that the following bounds hold true:
	\begin{equation*}
	\norm{v^i}_{L^{\infty}\left([0,+\infty),L^2(K)\right)} + \norm{v^i}_{L^2\left( [0,+\infty), W^1_{per}(K)\right)} \le C \norm{\nabla\cdot (a(\bx)\mathbf{e}_i)}_{L^2(K)}. 
	\end{equation*}
\end{proposition}
Here, the space $W^1_{per} (K)'$ is the dual space of $W^1_{per} (K)$ (a characterization of this space can be found in Ref. \refcite{CiD99}). 
With a slight abuse of notation, in the coming sections the functions $v^i$ will indicate both the solution of \cref{eq:parabolic periodic problem} on the cell $K$ and its periodic extension to the whole $\R^d$.
Finally, we define the bilinear form $B:  W^1_{per}(K) \times W^1_{per}(K) \mapsto \R $ through the formula
\begin{equation}\label{eq: bilinear form}
B[u, v ] = \int_{K}\nabla u \cdot a(\bx) \nabla v \,d\bx.
\end{equation}

If $a\in \mathcal{M}_{\text{per}}(\alpha,\beta,K)$, the bilinear form $B[\cdot,\cdot]$ is continuous and coercive and there exists a non-decreasing sequence of strictly positive eigenvalues $\left\{\lambda_j \right\}_{j=0}^{\infty}$ and a $L^2$-orthonormal set of eigenfunctions $\left\{\varphi_j \right\}_{j=0}^{\infty}\subset W^1_{per}(K)$ such that 
\begin{equation}\label{eq: eigenproblem}
B[\varphi_j,w] = \lambda_j \langle \varphi_j,w\rangle_{L^2(K)}, \quad \forall w\in W^1_{per}(K).
\end{equation}


\section{Main results}\label{sec:main result}
The starting point of the analysis is the following new formula for the approximation of the homogenized coefficient $a^{0}$ in \eqref{eq:homogenized elliptic}
\begin{equation}
	\label{eq: a0 parabolic dirichlet}
	\mathbf{e}_i  a^{0}_{R,L,T} \mathbf{e}_j = \int_{K_{L}} \mathbf{e}_i \cdot a(\bx) \mathbf{e}_j \mu_L(\bx)\,d\bx - 2\int_{0}^{T} \int_{ K_{L}} u^i(\bx,t) u^j(\bx,t) \mu_L(\bx) \,d\bx  \,dt,
\end{equation}
where $\{ u^{i} \}_{i=1}^{d}$ are the solutions of the parabolic problems \eqref{eq:parabolic dirichlet problem}. Note that the parabolic solutions $\{ u^{i} \}_{i=1}^{d}$ are solved over $K_R$, from which it follows the dependency of $a^{0}_{R,L,T}$ on the parameter $R$, while the averaging is taking place over the domain $K_L \subset K_R$. The aim of this section is two-fold: first, in \Cref{sec: equivalence}, we will recall a result which is proved in Ref. \refcite{AAP19a}, where the equivalence between the approximate homogenized coefficient \eqref{eq: a0 parabolic dirichlet} (when $T = \infty$\add{, $L=R$ and the constant filter $\mu_L=L^{-d}$ is adopted}), and the approximation \eqref{eq:bar_a0_Intro} based on elliptic cell problems  (when $\chi_R^{i}$ is supplied with homogeneous Dirichlet BCs) is shown. Next, in \Cref{sec: filter intro}, we will present our main statement in \cref{thm: modelling error parabolic correctors filter}, which states that if $T$ \add{and $L$} are chosen optimally, then we obtain arbitrarily high convergence rates for the difference between $a^{0}_{R,L,T}$ in \eqref{eq: a0 parabolic dirichlet} and the exact homogenized coefficient $a^{0}$ in \eqref{eq:a0}, when $a \in \mathcal{M}_{\text{per}}(\alpha,\beta,K)$.

\add{The periodicity of the coefficient $a$ is assumed only to simplify the theory. The method itself can be easily generalized to cases when the homogenized coefficient $a^{0}$ is not a constant by shifting the local spatial domains in \eqref{eq: a0 parabolic dirichlet}, as well as the parabolic problem \eqref{eq:parabolic dirichlet problem} to $\bx  + K_R$, where $\bx$ is a point at which the homogenized coefficient $a^{0}_R$ needs to be computed. One typical example is when the coefficient $a^{\varepsilon}(\bx)$ is locally periodic and has both fast and slow variations. In this case, it is even possible to generalize the analysis by considering a local Taylor's expansion of the parabolic solutions \eqref{eq: a0 parabolic dirichlet} around the slow variable following the ideas in Ref. \refcite{ArR16}.
}
\subsection{Equivalence between the standard and parabolic homogenized coefficients}
\label{sec: equivalence}
Assume that the elliptic solutions $\chi^{i}_{R}$ in \eqref{eq: microproblem_Intro} are supplied either with periodic or homogeneous Dirichlet BCs. By symmetry of $a(\bx)$, we can rewrite \cref{eq:bar_a0_Intro} as:
	\begin{equation}\label{eq: a0 ell symmetric KR}
	\mathbf{e}_i \cdot a^{0}_{R}  \mathbf{e}_j =  \fint_{  K_{R } } \mathbf{e}_i \cdot a (\bx) \mathbf{e}_j \,d\bx -  \fint_{  K_{R } } \nabla \chi_{R}^i ( \bx ) \cdot a (\bx)  \nabla \chi_{R}^j ( \bx ) \,d\bx.
	\end{equation}
\Cref{thm:weak form convergence} provides an alternative expression for the second integral, which will be referred to as the \emph{correction part} of the homogenized tensor, based on the use of parabolic problems over infinite time domain. We refer to Ref. \refcite{AAP19a} for a rigorous proof.
	
	\begin{theorem}\label{thm:weak form convergence}
		Let $a(\bx)\in \mathcal{M}(\alpha,\beta,K_R)$, $u^i \in C([0,+\infty), L^2(K_R))$ be the weak solution of \cref{eq:parabolic dirichlet problem} and $\chi^i_R \in H^1_0( K_{R} )$ be the weak solution of \cref{eq: microproblem_Intro}. Then, for $1\le i,j \le d$, the following identities hold

		\begin{gather}
		\label{eq:result1 proposition}
		\chi_{R}^i = \int_{0}^{+\infty} u^i(\cdot,t) \, dt \quad \text{ in } H^1_0 ( K_{R} ) ,\\
		\label{eq:result2 proposition}
		\frac{1}{2} \int_{K_{R}} \nabla \chi^i_R(\bx) \cdot a(\bx) \nabla \chi^j_R(\bx)  d\bx = \int_{0}^{+\infty} \int_{ K_{R}}u^i(\bx,t) u^j(\bx,t) d\bx\,dt. 
		\end{gather}
	\end{theorem}	

\Cref{thm:weak form convergence} implies that if $T = \infty$ (and $\mu_L = L^{-d} \text{ in } K_L$ with $R=L$) in \eqref{eq: a0 parabolic dirichlet}, then the parabolic formulation does not lead to any gain in the first order convergence rate in \eqref{eq:ModellingError} due to the equivalence relation above. It is important to notice that we do not need the periodicity assumption on the tensor $a$ for deriving the equivalence. Moreover, the same result holds true if we substitute the homogeneous Dirichlet condition with the periodic boundary conditions, under the periodicity assumption for the tensor $a$. Then we have the following corollary:
	\begin{corollary}
		Let $a(\bx)\in \mathcal{M}_{per}(\alpha,\beta,K)$. Let $v^i \in C([0,+\infty), L_{per}^2(K))$ solve \cref{eq:parabolic periodic problem}. Then  
		\begin{equation}
		\label{eq:new a0 periodic}
		 \mathbf{e}_i a^0 \mathbf{e}_j = \fint_{K} \mathbf{e}_i \cdot a(\bx) \mathbf{e}_j \,d\bx - 2 \int_{0}^{+\infty} \fint_{K} v^i(\bx,t) v^j(\bx,t) \,d\bx\,dt.
		\end{equation}
	\end{corollary}
	
\add{%
	\begin{remark}
		Identities \cref{eq: a0 ell symmetric KR,eq:result2 proposition,eq:new a0 periodic} are based on the assumption of symmetry for the matrix $a(\cdot)$. By consequence, \cref{eq: a0 parabolic dirichlet} should be modified for problems with non-symmetric coefficients.
	\end{remark}
}
\subsection{Main result: exponential decay of the resonance error} \label{sec: filter intro}
\newcommand{\nacon}{\ensuremath{\nu}}
As stated in the \Cref{sec: equivalence}, the consequence of the equivalence between the parabolic model and the standard elliptic model is that the first order convergence rate of the resonance error in \eqref{eq:ModellingError} remains unchanged. In this subsection, we summarize our main result which states that we are able to achieve arbitrarily high convergence rates for the resonance error
	 \begin{equation*}
	 e_{MOD} :=  \norm{  a^{0}_{R,L,T} - a^0 }_{F},
	 \end{equation*} 	
upon choosing the parameters $T$ and $L$ optimally. \add{Our basic assumptions, for our main theorem, are the following:
\begin{subequations}
	%
	\label{eq: hp on a} 
	\begin{equation}
		\label{eq: hp on a standard}
		a(\cdot) \in \mathcal{M}_{\text{per}}(\alpha,\beta,K), 
	\end{equation}	
	\begin{equation}
		\label{eq: hp on a regularity}
		a(\cdot)\mathbf{e}_i \in W^{1,\infty}(K_R) ,\quad\text{for }i = 1,\dots,d, 
	\end{equation}	
	\begin{equation}
		\label{eq: hp on v regularity}
		v^i\in L^p\left((0,T), W^{1,p}_{per}(K)\right),\quad\text{for }i = 1,\dots,d,
	\end{equation}	
\end{subequations}
with $p>p_0=\max\left\lbrace\frac{d+2}{2},2\right\rbrace$. 
\begin{remark}
Let us briefly comment on assumption \eqref{eq: hp on v regularity}. 
This assumption is needed to have sufficient regularity of a boundary layer function defined in \eqref{eq:theta2}. One can notice that, for $d=1$ the classical estimate of $v^i\in L^2\left((0,T), W^{1,2}_{per}(K)\right)$ is sufficient, for $d=2$ we need $p>2$ and for $d=3$ we need $p>5/2$. The last two conditions are only slightly more restrictive than what we already know on $v^i$, namely $v^i\in L^2\left((0,T), W^{1,2}_{per}(K)\right)$. Finally, such $L^p$ estimates have been proved (for parabolic problems with homogeneous Dirichlet BCs, as in \eqref{eq:parabolic dirichlet problem}) in Chapter 2, Theorem 2.5 Ref. \refcite{BLP78} under appropriate assumptions on $a(\cdot)$, but without an explicit characterization of $p>2$.\footnote{The assumption $a(\cdot)\in W^{2,p}(K_R)$ is sufficient but might not be necessary, see \refcite{BLP78}.}
\end{remark}
}

\begin{theorem}\label{thm: modelling error parabolic correctors filter}
	 Under assumptions \cref{eq: hp on a}, let $K_R \subset \R^d$ for $R\ge 1 $. 
	 Let $a^{0}_{R,L,T}$ and $a^0$ be defined, respectively, as in \cref{eq: a0 parabolic dirichlet} and \cref{eq:a0}, with $u^i$ satisfying \cref{eq:parabolic dirichlet problem} for any $i=1,\dots,d$ and $\mu_{L} \in \mathbb{K}^{q}(K_{L})$, with $0<L<R-2$.
	 \add{There exist constants $\nu(\beta,d)>0$, $0<c \le\nu$, $\lambda_0(\alpha,d)>0$, $C(\alpha,\beta,d)>0$ and $\tilde{C}_{2,c}(\alpha,\beta,d)>0$ such that, if $T<\frac{2\nacon}{d}\abs{\tilde{R}-L}^2$ and $T<\tilde{C}_{2,c}\abs{R-L}$, then}
	 \add{ %
	  \begin{equation}\label{eq:final result}
	 	\norm{ a^{0}_{R,L,T} - a^0 }_{F} \le C \Bigg[
	 		L^{-(q + 1)} + e^{-2\lambda_0 T} + 
	 		\frac{R^{d-1}}{T^{d/2}} e^{-c\frac{\abs{R-L}^2}{T}} 
	 	+\frac{R^{2(d-1)}}{T^{d}} e^{-2c\frac{\abs{R-L}^2}{T}}
	 		\Bigg].
	 \end{equation}}
 %
%
	 The choice 
	 	$$L = k_o R, \quad \add{T = k_T (R-L)},$$	 
	 with $0<k_o<1$ and 
	 $\add{k_T = \sqrt{\frac{c}{2\lambda_0}} }$ results in the following convergence rate in terms of $R$
	 \begin{equation}\label{eq:final result optim}
	 	\norm{ a^{0}_{R,L,T} - a^0 }_{F} \le C \left[R^{-(q + 1)} + e^{-\sqrt{2\lambda_0c}(1-k_o)R} \right],
	 \end{equation}
	 for a constant $C>0$ independent of $R$, $L$ or $T$.
\end{theorem}

	
\add{\subsubsection{Choice of the parameters}\label{subsubsec: choice parameters}
The choice of $T$ has a direct influence on the decay of the exponential part of the error. If $k_T$ is too small then the first exponential term in \cref{eq:final result} will be dominant, while if $k_T$ is too large then the last two exponentials will dominate the error bound (provided that $R$ and $L$ are not extremely large).
The optimal value of $k_T$ depends on $\lambda_0$ and $c$, but their values are not available in most of the cases.	
A practical choice of $k_T$ can be made by using approximate values for $c$ and $\lambda_0$. 
The constant $c$, which appears in the last two exponential terms in \eqref{eq:final result}, can be taken arbitrarily close to the constant $\nacon$ in the exponent of the Nash--Aronson estimate \eqref{eq:NashAronson}, where the constant $\nu$ depends only on the dimension $d$ and the constant $\beta$ in \Cref{Def_MSpace} (see \Cref{remark: bound for nu}).
For $\lambda_0$, which is the smallest eigenvalue of the bilinear form \eqref{eq: bilinear form}, the following approximation hold: ${\lambda_0 \approx \frac{\pi^2 \alpha}{diam(K)^2} \approx \frac{\pi^2}{d} \inf_{x\in\R^d, i,j}\abs{a_{ij}(x)}}$ (see \Cref{remark: value lambda0}).
As a consequence, the exponent in the exponential term $\sqrt{2\lambda_0c}\approx \sqrt{\alpha/\beta}$ will depend on the contrast ratio. 
}
Later in the analysis, it will be clear that the main idea of limiting $T = \mathcal{O}(R-L)$ is to exploit the mild dependence of the parabolic solutions $u^{i}$ on the boundary conditions, which is the case if parabolic solutions are evolved over a sufficiently short time. 
Second, the use of filtering functions $\mu_L$ is to achieve high order convergence rates for the averages of oscillatory functions, which is another essential component in achieving high order rates for the resonance error. 


\section{Error analysis} \label{sec:analysis}
In this section we prove the bound stated in \cref{thm: modelling error parabolic correctors filter}. The proof can be outlined as follows:
\begin{description}
	\item[Step 1:] We exploit the fact that the exact homogenized coefficient $a^{0}$ in \eqref{eq:a0} is equal to \eqref{eq:new a0 periodic}, and we decompose the error into four terms
	\begin{multline}\label{eq:decomposition}
\mathbf{e}_i	(a^{0}_{R,L,T} - a^0) \mathbf{e}_j= 
	\underbrace{
		\int_{K_{L}} \mathbf{e}_i \cdot a(\bx) \mathbf{e}_j \mu_{L}(\bx) \,d\bx
		- \fint_{K} \mathbf{e}_i \cdot a(\bx) \mathbf{e}_j \,d\bx 
	}_{I^1_{ij}}  \\
	\underbrace{ 
		+2 \int_{0}^{T} \int_{ K_{L}} v^i(\bx,t) v^j(\bx,t) \mu_{L}(\bx) d\bx\,dt
		- 2 \int_{0}^{T} \int_{ K_{L}} u^i(\bx,t) u^j(\bx,t) \mu_{L}(\bx) d\bx\,dt 
	}_{I^2_{ij}} \\
	\underbrace{ 
		+2 \int_{0}^{T} \fint_{K} v^i(\bx,t) v^j(\bx,t) d\bx\,dt
		- 2 \int_{0}^{T} \int_{ K_{L}} v^i(\bx,t) v^j(\bx,t) \mu_{L}(\bx) d\bx\,dt 
	}_{I^3_{ij}} \\
	\underbrace{ 
		+2 \int_{0}^{+\infty} \fint_{K} v^i(\bx,t) v^j(\bx,t) d\bx\,dt
		- 2\int_{0}^{T} \fint_{K} v^i(\bx,t) v^j(\bx,t) d\bx\,dt
	}_{I^4_{ij}} .
	\end{multline}
	\item[Step 2:] Estimation of the \emph{averaging} errors $I^1_{ij}$ and $I^3_{ij}$ by means of \cref{lemma: kernel}.
	\item[Step 3:] Estimation of the \emph{truncation} error $I^4_{ij}$ by means of the exponential decrease in time of $\norm{v^i(\cdot,t)}_{L^2(Y)}$.
	\item[Step 4:] Estimation of the \emph{boundary} error $I^2_{ij}$ by means of upper bounds for the fundamental solution of the parabolic problem and integration over finite time intervals $[0,T]$.
\end{description}
The coming subsections will be devoted to the derivation of upper bounds for $I^1_{ij}$, $I^2_{ij}$, $I^3_{ij}$ and $I^4_{ij}$. 
\subsection{Bounds for $I^1_{ij}$ and $I^3_{ij}$}
The two error terms studied in this subsection originate from the fact that we are approximating the averages of periodic functions by a weighted average over a bounded domain. For such a reason, these errors will be referred to as \emph{averaging} error for $a$ ($I^1_{ij}$) and for $v^iv^j$ ($I^3_{ij}$). The \cref{lemma: I_1} is a direct consequence of \cref{lemma: kernel}, and therefore the proof is omitted.

\begin{corollary}\label{lemma: I_1}
	Let $a \in \mathcal{M}_{per}(\alpha,\beta,K)$ be periodically extended over $K_L$. Then, there exists $C_1>0$, independent of $L$, such that
	\begin{equation*}
	\abs{I^1_{ij}} \le C_1 L^{-(q+1)}, \quad i,j = 1,\dots,d.
	\end{equation*}
\end{corollary}

\begin{lemma}\label{lemma: I_3}
	Let $a(\cdot)$ satisfy assumption \eqref{eq: hp on a standard}, let $v^i\in L^2([0,+\infty),W^{1}_{per}(K))$ be the $K$-periodic solution of \eqref{eq:parabolic periodic problem} and $\mu_L\in\mathbb{K}^q(K_L)$. Then, there exists $C_3>0$, independent of $L$, such that
	\begin{equation*}
	\abs{I^3_{ij}} \le C_3 L^{-(q+1)}.
	\end{equation*}
\end{lemma}
\begin{proof}
	By applying \cref{lemma: kernel} to function $2 v^iv^j$ we get:
	\begin{equation}\label{eq: Lp norm product}
	\abs{I^3_{ij}} \le \int_{0}^{T} C \norm{v^i(\cdot,t)v^j(\cdot,t)}_{L^{p}(K)} L^{-(q+1)} \,dt ,
	\end{equation}
	with $1 < p\le 2$. Following the proof of \cref{lemma: kernel} (see Appendix A, Ref. \refcite{Glo11}), we deduce that, for any $q\ge 2$ one can also choose $p=1$ in the inequality above. Therefore, by the use of Cauchy--Schwarz and H\"older inequalities, $I^3_{ij}$ can be estimated as
	\begin{equation*}
	\begin{split}
	\abs{I^3_{ij}} &\le \int_{0}^{T} C \norm{v^i(\cdot,t)v^j(\cdot,t)}_{L^{1}(K)} L^{-(q+1)} \,dt \\
	&\le C  L^{-(q+1)} \int_{0}^{T} \norm{v^i(\cdot,t)}_{L^{2}(K)} \norm{v^j(\cdot,t)}_{L^{2}(K)}\,dt \\
	&\le C  L^{-(q+1)} \norm{v^i}_{L^2([0,+\infty),L^{2}(K))}  \norm{v^j}_{L^2([0,+\infty),L^{2}(K))}. 
	\end{split}
	\end{equation*}
	The result follows by choosing 
	\[
	C_3 := C \norm{v^i}_{L^2([0,+\infty),L^{2}(K))}  \norm{v^j}_{L^2([0,+\infty),L^{2}(K))}.
	\]
	
	In the case $q\in\set{0,1}$ we cannot utilize any more the $L^1$-norm of the product. In view of \cref{eq: Lp norm product}, with the choice $p=3/2$, it follows that
	\begin{equation*}
	\begin{split}
	\abs{I^3_{ij}} &\le \int_{0}^{T} C \norm{v^i(\cdot,t)v^j(\cdot,t)}_{L^{3/2}(K)} L^{-(q+1)} \,dt \\
	&\le \int_{0}^{T} C \norm{v^i(\cdot,t)v^j(\cdot,t)}_{W^{1,1}(K)} L^{-(q+1)} \,dt \\
	&\le \int_{0}^{T} C \norm{v^i(\cdot,t)}_{W^{1}_{per}(K)} \norm{ v^j(\cdot,t)}_{W^{1}_{per}(K)} L^{-(q+1)} \,dt \\
	&\le C  L^{-(q+1)} \norm{v^i}_{L^2([0,+\infty),W^{1}_{per}(K))}  \norm{v^j}_{L^2([0,+\infty),W^1_{per}(K))},
	\end{split}
	\end{equation*}
	where the first inequality is a direct application of \cref{lemma: kernel}, the second inequality follows from the continuous inclusion of $W^{1,1}(K)$ in $L^{3/2}(K)$, the third inequality comes from the embedding $W^{1}_{per}(K)\subset W^{1,1}(K)$ and the validity of the \add{derivative product rule for functions $v^i\in W^{1,p}(\Omega)\cap L^{\infty}(\Omega)$ for $1\le p\le\infty$ (see, e.g. Ref. \refcite{Bre10} for a proof)} which implies: 
	\begin{equation*}
	\norm{v^i(\cdot,t) v^j(\cdot,t)}_{W^{1,1}(K)} \le C \norm{v^i(\cdot,t)}_{W^{1}_{per}(K)} \norm{ v^j(\cdot,t)}_{W^{1}_{per}(K)}.
	\end{equation*}
	Finally, the last inequality is the Chauchy-Schwarz inequality. The result follows by choosing 
	\[
	C_3 := C \norm{v^i}_{L^2([0,+\infty),W^{1}_{per}(K))} \norm{v^j}_{L^2([0,+\infty),W^{1}_{per}(K))}.
	\]
\end{proof}
\subsection{Bound for $I^4_{ij}$}
In this subsection we derive an \textit{a-priori} estimate for the truncation error, which originates from the restriction of the time integral in \cref{eq: a0 parabolic dirichlet} on the finite interval $[0,T]$. As it will be clearer from the coming analysis, the time truncation is essential for improving the convergence rate of the resonance error, as large values of $T$ result in a pollution of the correctors.
First of all, we recall the following lemma on the exponential decay in time of $\norm{v^i(\cdot,t)}_{L^2(K)}$.
\begin{lemma}\label{lemma:exponential decay}
	Let $v^i \in C([0,\infty), L^2(K))$ be the solution of \cref{eq:parabolic periodic problem} and let $\lambda_0>0$ be the smallest eigenvalue of the bilinear form $B$ introduced in \cref{eq: bilinear form}.	
	Then
	\begin{equation*}
	\norm{v^i(\cdot,t)}_{L^2(K)} \le e^{-\lambda_0 t} \norm{v^i(\cdot,0)}_{L^2(K)}, \quad \operatorname{a.e.} t \in [0,+\infty).
	\end{equation*}
\end{lemma}
\begin{proof}
	The weak formulation of \cref{eq:parabolic periodic problem} reads: Find $v^i \in L^2([0,+\infty), W^1_{per}(K))$, $\partial_t v^i \in L^2([0,+\infty), W^1_{per}(K)')$ such that
	\begin{equation*}
	\begin{aligned}
	&\left( \partial_t v^i, w \right) + B[v^i,w] = 0, \qquad \forall w \in W^1_{per}(K), \\
	&v^i(\cdot,0) = \nabla \cdot \left( a\mathbf{e}_i \right)\in L_0^2(K).
	\end{aligned}
	\end{equation*}
	By using $w = v^i(\cdot,t)$, the second line becomes
	\begin{equation*}
	\frac{1}{2} \frac{d}{dt} \norm{v^i}^2_{L^2(K)} = - B[v^i,v^i].
	\end{equation*}
	Let $\left\lbrace\lambda_j\right\rbrace_{j=0}$ and $\left\lbrace\varphi_j\right\rbrace_{j=0}$ be, respectively, the eigenvalues and eigenfunctions of $B$ and let us denote $\hat{v}^i_j:= \langle v^i,\varphi_j\rangle_{L^2(K)}$. By orthogonality of the eigenfunctions and Parseval's identity, it holds
	\begin{equation*}
	B[v^i,v^i] = \sum_{j=0}^{\infty} \lambda_j \abs{\hat{v}^i_j}^2 \ge \lambda_0 \sum_{j=0}^{\infty} \abs{\hat{v}^i_j}^2 = \lambda_0 \norm{v^i}^2_{L^2(K)}.
	\end{equation*}
	
	Then, by coercivity of the bilinear form $B$ and use of the above inequality, we get
	 \begin{equation*}
	 \norm{v^i}_{L^2(K)} \frac{d}{dt} \norm{v^i}_{L^2(K)} = \frac{1}{2} \frac{d}{dt} \norm{v^i}^2_{L^2(K)} = - B[v^i,v^i] \le - \lambda_0 \norm{v^i}_{L^2(K)}^2.
	 \end{equation*}
	 So, the following differential inequality is derived:
	 \begin{equation*}
	 \frac{d}{dt} \norm{v^i}_{L^2(K)} \le - \lambda_0 \norm{v^i}_{L^2(K)}.
	 \end{equation*}
	 As proved in Ref. \refcite{Eva10}, $\norm{v^i(\cdot,t)}_{L^2(K)}$ is absolutely continuous in time, and the result is obtained by Gronwall's inequality.	 
\end{proof}

\begin{remark} \label{remark: value lambda0}
	It is easy to prove that $\lambda_0 \ge \frac{\alpha}{C_P^2}$, where the Poincar\'e constant for a convex domain $K$ is $C_P=\frac{diam(K)}{\pi}$, see Ref. \refcite{PaW60}. 
	\add{Here, we assume to know exactly $diam(K)$. In practice, it can often be estimated.}
\end{remark}

\begin{lemma}[Truncation error]\label{lemma: I_4}
	Let $v^i \in C([0,+\infty),L^2(K))$ solve \cref{eq:parabolic periodic problem}, and let 
	\begin{equation*}
		I_{ij}^4 :=  2\int_{T}^{+\infty} \fint_{ K} v^i(\bx,t) v^j(\bx,t) d\bx\,dt.
	\end{equation*}
	Then, there exist $C_4>0$, independent of $T$, such that
	\begin{equation}
	\label{eq:truncation error}
		\abs{I^4_{ij}}  \le C_4 e^{- 2 \lambda_0 T},
	\end{equation}
	where $\lambda_0$ is the smallest eigenvalue of $B$.
\end{lemma}
\begin{proof}
	We start by applying the Cauchy-Schwarz inequality on $L^2(K)$:
	\begin{equation}\label{truncation error estimate 1}
	\abs{I^4_{ij}}  \le \frac{2}{\abs{K}} \int_{T}^{\infty} \norm{ v^i(\cdot,t) }_{L^2(K)} \norm{ v^j(\cdot,t) }_{L^2(K)}   \,dt.
	\end{equation}
	 Then, we plug the result of \cref{lemma:exponential decay} into \cref{truncation error estimate 1}:
	\begin{equation*}\label{truncation error estimate 2}
	\begin{split}
	\abs{I^4_{ij} } &\le \frac{2 }{\abs{K}} \int_{T}^{\infty}e^{- 2\lambda_0 t} \norm{ v^i(\cdot,0) }_{L^2(K)} \norm{ v^j(\cdot,0) }_{L^2(K)}  \,dt\\
	&\le \frac{1}{\abs{K}} \norm{ v^i(\cdot,0) }_{L^{2}(K)} \norm{ v^j(\cdot,0) }_{L^{2}(K)} \frac{1}{\lambda_0} e^{-2 \lambda_0 T }.
	\end{split}
	\end{equation*}
	The result follows by choosing
	\begin{align*}
	C_4 &= \frac{1}{\lambda_0 \abs{K}} \norm{ v^i(\cdot,0) }_{L^2(K)} \norm{ v^j(\cdot,0) }_{L^2(K)}\\
		&= \frac{1}{\lambda_0 \abs{K}} \norm{ \nabla\cdot\left(a(\cdot)\mathbf{e}_i\right) }_{L^2(K)} 
		\norm{ \nabla\cdot\left(a(\cdot)\mathbf{e}_j\right)  }_{L^2(K)}.
	\end{align*}
	
\end{proof}
%
\subsection{Bound for $I^2_{ij}$}
From the definition,
\begin{equation} \label{eq: definition I2}
	I^2_{ij}:= 2\int_{0}^{T} \int_{K_L} \left(v^iv^j - u^iu^j\right) \mu_L \,d\bx \,dt,
\end{equation}
one can notice that the source of the error $I^2_{ij}$ is the mismatch between $u^i$ and $v^i$ on the boundary $\partial K_R$. Therefore, we refer to such an error as the \emph{boundary} error. 
The boundary error converges to zero at an exponential rate, as stated in \cref{lemma: est I2}.

\begin{lemma}\label{lemma: est I2}
	Let assumptions \eqref{eq: hp on a} be satisfied, let $I^2_{ij}$ be defined by \cref{eq: definition I2}, $0<L<R-2$, \add{${T<\frac{2\nacon}{ d}\abs{\tilde{R}-L}^2}$ and $T<\tilde{C}_{2,c}\abs{R-L}$, where the constants $\nacon,\tilde{C}_{2,c}>0 $ are defined in \Cref{thm:Aro68} and \Cref{prop: bound I2c}, respectively.}
	Then, there exists a constant $C>0$ and a constant $0<c\le\nu$, independent of $R$, $L$ and $T$ such that 
	\add{%
		\begin{equation*}
			\abs{I^2_{ij}}
				\le 
				C
				\left[
				 \frac{R^{d-1}}{T^{d/2}} e^{-c\frac{\abs{R-L}^2}{T}}  
				+ \frac{R^{2(d-1)}}{T^{d}} 
				e^{-2c\frac{\abs{R-L}^2}{T}}
				\right].
		\end{equation*}}
\end{lemma}
 
The proof of \cref{lemma: est I2} directly follows from \cref{prop: bound I2b,prop: bound I2c}. We need \cref{def:boundary layer,def:cutoff} in order to define a \emph{boundary error function} which will be used in the estimation of $I^2_{ij}$.
\begin{definition}[Boundary layer]\label{def:boundary layer}
	Let us define a sub-domain $K_{\tilde{R}}\subset K_R$, where $\tilde{R}$ is defined to be the largest integer such that $\tilde{R}\le R-1/2$.
	The \emph{boundary layer} is defined as the set $\Delta := K_{R} \setminus K_{\tilde{R}}$. We observe that $\abs{\Delta} = R^d - \tilde{R}^d \le 2d R^{d-1}$.
\end{definition}
\begin{definition}[Cut-off function]\label{def:cutoff}
	A \emph{cut-off function} on $K_{R}$ is a function ${\rho \in C^{\infty}(K_R,[0,1])}$ such that 
	\begin{equation*}
	\rho(\bx) = 
	\left\lbrace
	\begin{array}{l}
	1 \text{ in } K_{\tilde{R}}\\
	0 \text{ on } \partial K_{R}
	\end{array}
	\right. 
	\quad\text{and}\quad
	\abs{\nabla \rho(\bx)}\le C_{\rho} \,\text{ on }\Delta,
	\end{equation*} 
	where the subdomain $K_{\tilde{R}}$ and the boundary layer $\Delta$ are defined according to \cref{def:boundary layer}.
\end{definition}
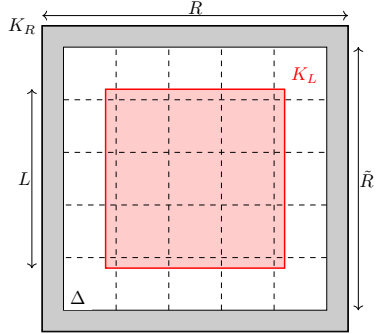
\begin{figure}[h!]
	\centering
	\scalebox{.7}{\begin{tikzpicture}
\draw[very thick](-2.9,-2.9) rectangle(2.9,2.9) ;
\filldraw[fill=black!20, even odd rule] (-2.9,-2.9) rectangle(2.9,2.9)  (-2.5,-2.5) rectangle(2.5,2.5) ;
\filldraw[fill = red!20, draw = red, thick] (-1.7,-1.7) rectangle (1.7,1.7);
\draw[dashed] (-1.5,-2.5) -- (-1.5,2.5);
\draw[dashed] (-.5,-2.5) -- (-.5,2.5);
\draw[dashed] (1.5,-2.5) -- (1.5,2.5);
\draw[dashed] (.5,-2.5) -- (.5,2.5);
\draw[dashed] (-2.5,1.5) -- (2.5,1.5);
\draw[dashed] (-2.5,.5) -- (2.5,.5);
\draw[dashed] (2.5,-.5) -- (-2.5,-.5);
\draw[dashed] (2.5,-1.5) -- (-2.5,-1.5);
%
\draw[<->] (-2.9,3.1) -- (2.9,3.1);
\draw (0,3) node[above] {$R$};
\draw[<->] (3.1,-2.5) -- (3.1,2.5);
\draw (3,0) node[right] {$\tilde{R}$};
\draw[<->] (-3.1,-1.7) -- (-3.1,1.7);
\draw (-3,0) node[left] {$L$};
\draw (-2.9,2.9) node[left] {$K_R$};
\draw (1.7,1.7) node[above right, red, fill = white] {$K_L$};
\draw (-2.5,-2.5) node[fill= white, above right]{$\Delta$};
\end{tikzpicture}}	
	\caption{Scheme of the sampling domain $K_R$ and its subsets $K_L$, $K_{\tilde{R}}$ and $\Delta$.}
\end{figure}
Let us define the boundary error function $\theta^i\in L^2\left([0,+\infty),H^1_0(K_R)\right)$ through the relation ${\theta^i:= u^i - \rho v^i}$.  
For the analysis it is fundamental that $\rho = 1$ in $K_{\tilde{R}}$ and that $L< \tilde{R}$. By the definition of $\theta^i$, we write 
\begin{equation*}
I^2_{ij} = 2 \int_{0}^{T} \int_{K_L} \left[ v^iv^j\left(1-\rho^2  \right) - \add{\rho}\left(\theta^i v^j + v^i \theta^j \right) - \theta^i \theta^j \right] \mu_L \,d\bx \,dt.
\end{equation*}
One readily notice that the first term in the integral vanishes on the integration domain, since $\rho^2(\bx) =1$ for all $\bx\in K_{\tilde{R}} \supset K_{L}$.
So, we have to study the integrals
\begin{equation}\label{eq: defs I2abc}
	I^{2,b}_{ij} := \int_{0}^{T} \int_{K_L} v^i \theta^j \mu_L \,d\bx \,dt, \;\text{ and }\;
	I^{2,c}_{ij} := \int_{0}^{T} \int_{K_L} \theta^i \theta^j \mu_L \,d\bx \,dt. 
\end{equation}
As both integrals depend on the values that the functions $\theta^i$ take over the averaging domain $K_L$, we need to provide pointwise estimates for $\theta^i(\bx,t)$ on $K_L\times[0,T]$. This is done in \cref{subsubsec:est theta} by the use of the fundamental solution of problem \cref{eq:parabolic dirichlet problem}. 
 
\subsubsection{Estimates for $\theta^i$}
\label{subsubsec:est theta}
Here, we derive an upper bound for $\theta^i$ on ${K_L\times [0,T]}$. 
\newadd{%
Let us define
\begin{equation}
	\label{eq:definition F}
	F(\bx,t) := -\nabla(1-\rho(\bx)) \cdot a(\bx) \nabla v^i - \nabla \cdot \left[ a(\bx) v^i \nabla(1-\rho(\bx))  \right].
\end{equation}} 
By definition and linearity of the correctors problem, the function $\theta^i$ satisfies
\begin{equation} \label{eq:theta2}
	\left\lbrace
	\begin{aligned}
		& \frac{\partial \theta^i}{\partial t} - \nabla\cdot( a(\bx)\nabla \theta^i) = F(\bx,t) & \quad &\text{in } K_{R}\times(0,+\infty),\\
		& \theta^i = 0 &  \quad &\text{on }  \partial K_{R} \times (0,+\infty),\\
		& \theta^i(\bx,0) = v^i(\bx,0)(1-\rho(\bx)) & \quad &\text{in } K_{R}.
	\end{aligned}
	\right. 
\end{equation}
\newadd{ %
	The  pointwise values of $\theta^i(\bx,t)$ can be estimated by making use the weak fundamental solution $G(\cdot,\cdot;\by):(\bx,t)\mapsto\R$, defined as the solution of
	\begin{equation} \label{eq: whole definition greens function}
		\left\lbrace
		\begin{aligned}
			& \partial_t G(\cdot,\cdot;\by) - \nabla\cdot(a(\bx)\nabla G(\cdot,\cdot;\by)) = 0 & \quad &\text{in } K_{R}\times(0,+\infty)\\
			& G(\cdot,\cdot;\by) = 0 &  \quad &\text{on }  \partial K_{R} \times (0,+\infty)\\
			& G(\cdot,0;\by) = \delta_{\by} & \quad &\text{in } K_{R},
		\end{aligned}
		\right. 
	\end{equation}
	where $\delta_{\by}$ is the Dirac delta centred at $\by$ and thanks to the results stated in \Cref{thm:Aro68}.
	%
	%
	\begin{theorem}[Theorem 9 in Ref. \refcite{Aro68}]\label{thm:Aro68}
	Assume $a(\cdot) \in \mathcal{M}(\alpha,\beta,K_R)$, $\theta^i(\cdot,0)\in L^2(K_R)$ and $F\in L^q\left((0,T), L^p(K_R)\right)$, with $1<p,q\le \infty$ and $\frac{d}{2p} + \frac{1}{q} \le 1$. Then, 
	\begin{enumerate}[label = \alph*)]
		\item $G(\bx,t;\by) = G(\by,t;\bx)$ for $\bx\in K_R$, $\by\in K_R$ and $t>0$.
		\item For each $\by\in K_R$, $G(\cdot,\cdot;\by)$ is continuous in $K_R\times(0,T)$.
		\item The Nash--Aronson estimate 
		\begin{equation}\label{eq:NashAronson}
			0 \le G(\bx,t;\by) \le \frac{C}{t^{d/2}} e^{-\nu \frac{\abs{\bx-\by}^2}{t}}, 
		\end{equation}
		holds true with $C(\alpha,\beta,d)>0$ and $\nu(\beta,d)>0$.
		\item For fixed $\by$ and arbitrary $\tau\in(0,T)$, $G(\cdot,\cdot;\by)\in L^2((\tau,T), H^1_0(K_R)) $.
		\item For any $t>0$, the weak solution $\theta^i$ of \cref{eq:theta2} can be represented by:
		\begin{multline}\label{eq:theta through green function}
			\theta^i(\bx,t) = \int_{K_R} G(\bx,t;\by) v^i(\by,0)(1-\rho(\by)) \,d\by \\
			+\int_{K_R} \int_{0}^{t} G(\bx,t-s;\by) F(\by,s)  \,ds \,d\by .
		\end{multline}
	\end{enumerate}
	\end{theorem}
\newadd{\begin{remark}\label{remark: bound for nu}
The practical choice of parameters in the numerical simulations requires to estimate the constant $\nu$ of \eqref{eq:NashAronson}. According to Ref. \refcite{Aro68} (equation (7.12)), $\nu$ can be estimated as $\nu = \gamma/8$, where $\gamma$ (which is called $\alpha$ in Ref. \refcite{Aro68}, Lemma 1) depends only on the dimension $d$ and the constant $\beta$ in \Cref{Def_MSpace}.
\end{remark}}
%
%
A straightforward computation using the definition of $F$, \eqref{eq:definition F}, shows that assumptions \eqref{eq: hp on a regularity} and \eqref{eq: hp on v regularity} in \Cref{thm: modelling error parabolic correctors filter} are needed in order to ensure that $F\in L^q\left((0,T), L^p(K_R)\right)$, with $1<p,q\le \infty$ and $\frac{d}{2p} + \frac{1}{q} \le 1$.
In the proofs of our main theorem we will use the following reformulation of $\theta^i$, using the definition of $F$ and integration by parts:
\begin{multline}\label{eq:theta through green function2}
	\theta^i(\bx,t) = \int_{K_R} G(\bx,t;\by) v^i(\by,0)(1-\rho(\by)) \,d\by \\
	-\int_{K_R} \int_{0}^{t} G(\bx,t-s;\by) \nabla_{\by}(1-\rho(\by)) \cdot a(\by) \nabla_{\by} v^i(\by,s)  \,ds \,d\by \\
	+ \int_{K_R} \int_{0}^{t} \nabla_{\by}G(\bx,t-s;\by) \cdot  a(\by) \nabla_{\by}(1-\rho(\by)) v^i(\by,s)  \,ds \,d\by.
\end{multline}
}

\add{An upper bound estimate for $\theta^i(\bx,t)$ on $K_L$ is given in the following lemma, whose proof is postponed to \ref{sec: appendix alter}.
	\begin{lemma}\label{lemma: preliminary estimate sup-theta NEW:)}
		Let assumptions \eqref{eq: hp on a} be satisfied, $0<L<R-2$ and $\theta^i$ be the solution of \eqref{eq:theta2}.
		Then, there exist constants $C_1,C_2>0$ and $0<c\le \nacon$ (where $\nacon$ is the constant in \eqref{eq:NashAronson}), independent of $R$ and $L$, such that, for any $0<t<2\nacon\abs{\tilde{R}-L}^2$,
		\begin{equation}\label{eq: preliminary estimate sup-thet NEW:)}
			\norm{\theta^i(\cdot,t)}_{L^{\infty}(K_L)} \le 
			C_1\norm{\nabla \cdot\left(a(\cdot)\mathbf{e}_i\right)}_{L^2(K)}
			\left( 1 + C_2 \frac{t}{\abs{R-L}} \right)
			\frac{R^{d-1}}{t^{d/2}}
			e^{-c\frac{\abs{R-L}^2}{t}}.
		\end{equation}
	\end{lemma}
}

\subsubsection{Term $I^{2,b}$}
\add{ %
\begin{proposition}\label{prop: bound I2b}
	Let the hypotheses of \cref{lemma: preliminary estimate sup-theta NEW:)} be satisfied. Moreover, let $I^{2,b}_{ij}$ be defined as in \cref{eq: defs I2abc}, $R\ge 1$ and $T<\frac{2\nacon}{d}\abs{\tilde{R}-L}^2$.  
	Then, there exist constants $C_{2,b}>0$ and $0<c\le\nacon$ independent of $R$, $L$, $T$ such that
	\begin{equation}\label{eq:bound I2b}
	\abs{I^{2,b}_{ij}} \le C_{2,b} \frac{R^{d-1}}{T^{d/2}} e^{-c\frac{\abs{R-L}^2}{T}}.
	\end{equation}
\end{proposition}
\begin{proof}
	Applying H\"older inequality on the space integral, we obtain:
	\begin{equation*}
	I^{2,b}_{ij} \le \int_{0}^{T} \norm{v^i(\cdot,t)}_{L^2(K_L)} \norm{\theta^j(\cdot,t)}_{L^{\infty}(K_L)} \norm{\mu_L}_{L^2(K_L)} \,dt.
	\end{equation*}
	By assumption, $\mu_L \in L^{\infty}(K_L)\subset L^2(K_L)$ with continuous inclusion, and by the definition of the filter (see \Cref{sec:notation definition}), we have
	\begin{equation*}
	\norm{\mu_L}_{L^2(K_L)} \le \abs{K_L}^{1/2} \norm{\mu_L}_{L^{\infty}(K_L)} \le C_{\mu} L^{-d/2}.	
	\end{equation*} 
	Next, we estimate $\norm{v^i(\cdot,t)}_{L^2(K_L)}$. Since $v^i$ \add{is $K$-periodic} we have, for integer $L$,
	\begin{equation*}
	\norm{v^i(\cdot,t)}_{L^2(K_L)} = L^{d/2} \norm{v^i(\cdot,t)}_{L^2(K)},
	\end{equation*} 
	while, for non-integer $L$
	\begin{equation*}
	\norm{v^i(\cdot,t)}_{L^2(K_L)} \le \lceil L\rceil^{d/2} \norm{v^i(\cdot,t)}_{L^2(K)}.
	\end{equation*} 
	Finally, we recall the exponential decay of $\norm{v^i(\cdot,t)}_{L^2(K)}$ and we derive the estimate:
	\begin{multline}\label{eq:initial bound I2b}
	I^{2,b}_{ij}
	\le L^{d/2} \norm{\nabla \cdot\left(a(\cdot)\mathbf{e}_i\right)}_{L^2(K)} \int_{0}^{T} e^{-\lambda_0 t} \norm{\theta^j(\cdot,t)}_{L^{\infty}(K_L)} \,dt \, C_{\mu} L^{-d/2} \\
	\le C_{\mu} \norm{\nabla \cdot\left(a(\cdot)\mathbf{e}_i\right)}_{L^2(K)} \int_{0}^{T} e^{-\lambda_0 t} \norm{\theta^j(\cdot,t)}_{L^{\infty}(K_L)} \,dt.
	\end{multline}
	\newadd{Since $T<\frac{2\nacon}{d}\abs{\tilde{R}-L}^2$, the assumption on $t$ of \Cref{lemma: preliminary estimate sup-theta NEW:)} are satisfied and }we can use \cref{eq: preliminary estimate sup-thet NEW:)} to bound the last integral in \cref{eq:initial bound I2b}:
	\begin{multline*}
	\int_{0}^{T} 
	e^{-\lambda_0 t} 
	\norm{\theta^j(\cdot,t)}_{L^{\infty}(K_L)} 
	\,dt \le\\
	\le 
	C_1\norm{\nabla \cdot\left(a(\cdot)\mathbf{e}_i\right)}_{L^2(K)}R^{d-1}
	\int_{0}^{T} 
	\left( 1 + C_2 \frac{t}{\abs{R-L}} \right)
	e^{-\lambda_0 t} 
	t^{-d/2}
	e^{-c\frac{\abs{R-L}^2}{t}}
	\,dt\\
	\le
	C_1\norm{\nabla \cdot\left(a(\cdot)\mathbf{e}_i\right)}_{L^2(K)}R^{d-1}
	\max_{t\in[0,T]}\left(t^{-d/2} 	e^{-c\frac{\abs{R-L}^2}{t}}\right)
	\int_{0}^{T} 
	\left( 1 + C_2 \frac{t}{\abs{R-L}} \right)
	e^{-\lambda_0 t} 
	\,dt\\
	=
	\frac{C_1}{\lambda_0}\norm{\nabla \cdot\left(a(\cdot)\mathbf{e}_i\right)}_{L^2(K)}R^{d-1}
	\left( 1 + \frac{C_2 }{\lambda_0\abs{R-L}} \right)
	\max_{t\in[0,T]}\left(t^{-d/2} 	e^{-c\frac{\abs{R-L}^2}{t}}\right).
	\end{multline*}
	From the fact that $t^{-d/2} e^{-c\frac{\abs{R-L}^2}{t}}$ is strictly increasing in $\left( 0,T\right)$, with $T<\frac{2\nacon}{d}\abs{\tilde{R}-L}^2$, we have
	 $$
	 \max_{t\in[0,T]}\left(t^{-d/2} e^{-c\frac{\abs{R-L}^2}{t}}\right) = T^{-d/2} 	e^{-c\frac{\abs{R-L}^2}{T}}.
	 $$
	Then, we get \cref{eq:bound I2b} by posing 
	\begin{equation*}
		\begin{aligned}
		C_{2,b} &= \frac{C_{\mu}C_1}{\lambda_0}\left( 1 + \frac{C_2 }{\lambda_0\abs{R-L}} \right) \norm{\nabla \cdot\left(a(\cdot)\mathbf{e}_i\right)}_{L^2(K)}^2\\
		&\le \frac{C_{\mu}C_1}{\lambda_0}\left( 1 + \frac{C_2 }{\lambda_0(1-k_o)} \right) \norm{\nabla \cdot\left(a(\cdot)\mathbf{e}_i\right)}_{L^2(K)}^2,
		\end{aligned}
	\end{equation*}
since $R\ge 1$.
\end{proof}
}
%
\subsubsection{Term $I^{2,c}$}
Here, we provide estimates for the term $I^{2,c}_{ij}$ of \cref{eq: defs I2abc}. This term decays faster than $I^{2,b}_{ij}$ and can be considered negligible. 
\add{
	\begin{proposition}\label{prop: bound I2c}
		Let the hypotheses of \cref{lemma: preliminary estimate sup-theta NEW:)} be satisfied. Moreover, let $I^{2,c}_{ij}$ be defined as in \cref{eq: defs I2abc} and $T<\frac{2\nacon}{ d}\abs{\tilde{R}-L}^2$. If $\tilde{C}_{2,c}>0$ is a constant independent of $T$, $R$ and $L$, such that $T<\tilde{C}_{2,c} (R-L)$, then there exist constants $C_{2,c}>0$ and $0< c \le \nu $ independent of $R$, $L$, $T$ such that
		\begin{equation}\label{eq:bound I2c}
			\abs{I^{2,c}_{ij}} \le C_{2,c} \frac{R^{2(d-1)}}{T^{d}}e^{-2c\frac{\abs{R-L}^2}{T}}.
		\end{equation}
	\end{proposition}
	\begin{proof}
			From the positivity of $\mu_L$ and the fact that it has unit mass, we derive the inequality
		\begin{equation*} 
				\abs{\int_{0}^{T} \int_{K_L} \theta^i(\bx,t) \theta^j(\bx,t) \mu_L(\bx) \,d\bx \,dt}
				\le  \max\limits_{i} \int_{0}^{T} \norm{\theta^i(\cdot,t)}^2_{L^{\infty}(K_L)} \,dt.
		\end{equation*} 
	From \Cref{lemma: preliminary estimate sup-theta NEW:)} the term on the right-hand side can be estimated by:
	\begin{multline*}
		\int_{0}^{T}  \norm{\theta^i(\cdot,t)}^2_{L^{\infty}(K_L)} \,dt \le\\
		\le C_1^2 \norm{\nabla \cdot\left(a(\cdot)\mathbf{e}_i\right)}^2_{L^2(K)}
		\int_{0}^{T} 
		\left( 1 + C_2 \frac{t}{\abs{R-L}} \right)^2
		\frac{R^{2(d-1)}}{t^{d}}
		e^{-2c\frac{\abs{R-L}^2}{t}}\,dt \\
		\le\frac{ C_1^2}{2c} \norm{\nabla \cdot\left(a(\cdot)\mathbf{e}_i\right)}^2_{L^2(K)}
		\left( 1 + C_2 \frac{T}{\abs{R-L}} \right)^2
		\frac{T^2}{\abs{R-L}^2}
		\frac{R^{2(d-1)}}{T^d} 
		e^{-2c\frac{\abs{R-L}^2}{T}},
	\end{multline*}
	hence \eqref{eq:bound I2c} follows by defining 
	\begin{multline*}
		C_{2,c}:= \frac{ C_1^2}{2c} \norm{\nabla \cdot\left(a(\cdot)\mathbf{e}_i\right)}^2_{L^2(K)}
		\left( 1 + C_2 \frac{T}{\abs{R-L}} \right)^2 \frac{T^2}{\abs{R-L}^2}\\
		\le \frac{ C_1^2}{2c} \norm{\nabla \cdot\left(a(\cdot)\mathbf{e}_i\right)}^2_{L^2(K)}
		\left( 1 + C_2 \tilde{C}_{2,c} \right)^2 \tilde{C}_{2,c} ^2.
	\end{multline*}
	\end{proof}
}

Now, we are ready to prove \cref{thm: modelling error parabolic correctors filter}.
\begin{proof}[Theorem \ref{thm: modelling error parabolic correctors filter}]
The decomposition \cref{eq:decomposition} implies 
$$
\norm{a^{0}_{R,L,T} - a^0}_{F} \le d^2 \max\limits_{i,j} \left( \abs{I^1_{ij}} + \abs{I^2_{ij}} + \abs{I^3_{ij}} + \abs{I^4_{ij}}\right).
$$
By using the upper bounds in \cref{lemma: I_1,lemma: I_3,lemma: I_4,prop: bound I2b,prop: bound I2c} in the above inequality we get
\begin{equation}\label{eq:final estimate 1}
\begin{aligned}%
\norm{a^{0}_{R,L,T} - a^0}_{F}
\le C \left[
L^{-(q + 1)} + e^{-2 \lambda_0 T} 
\add{+ \frac{R^{d-1}}{T^{d/2}}e^{-c\frac{\abs{R-L}^2}{T}}
+ \frac{R^{2(d-1)}}{T^{d}}e^{-2c\frac{\abs{R-L}^2}{T}}}
\right],
\end{aligned}%
\end{equation}
for some constant $C$ independent of $R$, $L$ and $T$.
Using the optimal values $L=k_oR$ and $T=k_T(R-L)$, with $0<k_o<1$ and $k_T = \sqrt{\frac{c}{2\lambda_0}}$, we write \eqref{eq:final estimate 1} as:
\begin{multline*}
\norm{a^{0}_{R,L,T} - a^0}_{F}
\le C \Bigg[
R^{-(q + 1)} + e^{-\sqrt{2\lambda_0c} (1-k_o)R} 
\add{ + R^{d/2 -1} e^{-\sqrt{2\lambda_0c} (1-k_o)R}  }\\
\add{ + R^{d-2} e^{-2 \sqrt{2\lambda_0c} (1-k_o)R} }
\Bigg]
\end{multline*}
The last term is of higher order than the third one, so it can be omitted. Finally, we get
\begin{equation}\label{eq:final estimate 1bis}
\norm{a^{0}_{R,L,T} - a^0}_{F} \le C 
\left[R^{-(q + 1)} + 
\left( 
1 +
\add{R^{d/2-1}}
\right)
e^{-\sqrt{2\lambda_0c} (1-k_o)R}\right].
\end{equation}

\medskip

%
\end{proof}


\section{Numerical tests}\label{sec:numerical experiments}
In this section we present several numerical tests which support the theoretical results of \Cref{sec:analysis}  and experimentally verify the resonance error bound of \cref{thm: modelling error parabolic correctors filter}. 
We \corr{illustrate} the expected convergence rates by varying the regularity parameter $q$ of the filters, in a periodic, smooth setting, as rigorously proven in the previous sections. Additionally, we compare the convergence rate of the resonance error for the parabolic scheme with that of 
\add{the regularised elliptic method described in Ref. \refcite{Glo11}}.
We also test non-smooth periodic and non-periodic coefficients (a realization of a log-normal field) to test our algorithm beyond the periodic setting and for which we obtain good results as in the smooth periodic case.

In order to numerically assess the convergence rate of the resonance error, we compute the approximations of the homogenized tensor through the described parabolic cell problems on domains of increasing size, $R\in[1,20]$, and calculate the resonance error as the Frobenius norm of the difference between the numerical approximation of $a^0_{R,L,T}$ (here still denoted as $a^0_{R,L,T}$) and the exact $a^0$:
\begin{equation}\label{eq: reserr}
err = \norm{a^0_{R,L,T} - a^0}_F.
\end{equation} 
In the case of periodic coefficients whose homogenized value could not be known exactly (i.e., without discretization error) the reference value is computed by solving the standard elliptic micro problem \cref{eq:periodic micro}
and periodic boundary conditions and using formula \cref{eq:a0}. 
In the random setting no approximation is available without some resonance error. 
In this case, we \corr{take as reference value for the homogenized tensor the one computed from the numerical approximation of the parabolic correctors over the largest domain $R_{max}=20$.}

To compute a numerical approximation of $a^{0}_{R,L,T}$, we use a Finite Elements (FE) discretization for the micro problems \cref{eq:parabolic dirichlet problem} in space, and a stabilised explicit Runge-Kutta method with adaptive time stepping for the time discretization. 
The fourth order explicit stabilised method ROCK4 , Ref. \refcite{Abd02}, with fine tolerance is chosen when we want to kill the discretization error in time, while the second order explicit stabilised method ROCK2, Ref. \refcite{AbM01}, with tolerances adapted to the spatial discretization error is chosen when we test the efficiency of the parabolic approach. 
As we use explicit methods in time, we need a mass matrix that is cheap to invert. This is achieve by using either mass lumping (for low order FEMs) or discontinuous Galerkin methods (for arbitrary order FEMs).

As a second step, the upscaled tensor is approximated by a double integration in space and time. \corr{The spatial integral of the parabolic correctors is computed by using} the FE filtered mass matrix of components
\begin{equation*}
m_{ij} = \int_{K_L} \phi_i(\bx) \phi_j(\bx) \mu_L(\bx) \,d\bx,
\end{equation*}
where $\{\phi_i\}_i$ are the FE basis functions.
The integration in time is performed by the use of a Newton-Cotes formula for non-uniform discretizations\add{, in particular the Simpson's rule}.

In order to optimize the convergence rate of the error with respect to the sampling domain size $R$, we take the optimal values for the averaging domain size $L$ ($K_L\subset K_R$) and for the final time $T$ given by 
\begin{equation}
	\label{eq:choice L T}
	L = k_o R, \text{ and} \quad \add{T = \sqrt{\frac{1}{4\pi^2\alpha\beta}}\abs{R-L}},
\end{equation}	 
where \add{$\alpha$ and} $\beta$ are the \add{ellipticity and} continuity constants of the tensor $a$. \add{This choice for $T$ comes from the optimal value given in \Cref{thm: modelling error parabolic correctors filter}, with the lower bound for $\lambda_0\ge \frac{\alpha \pi^2}{diam(K)^2}$ (see \Cref{remark: value lambda0}). For the practical estimation of the constant $c$ we use $c = (4\beta)^{-1}$ (this value is derived in Ref. \refcite{AKM19}, Chapter 8, Section 5 for unbounded domains). Note that $\beta$ is computed as $\beta = \norm{a}_{L^{\infty}(\R^d)}$. The value of $diam(K)=\sqrt{2}$ is used to derive \cref{eq:choice L T}, as we assigned the periodic cell to be $(-1/2,1/2)^2$. In practical situations where $diam(K)$ is not known exactly, it can often be roughly estimated.} 
The oversampling ratio, $0<k_o<1$, and the order of filters, $q$, can be chosen freely.
\subsection{Two-dimensional periodic case}\label{subsec:numerical 2d smooth periodic}
We consider the upscaling of the $2\times 2$ isotropic tensor:
\begin{equation}\label{eq: tensor GloriaLeBris}
a(\bx) = \left( \frac{2 + 1.8 \sin(2\pi x_1)}{2 + 1.8 \cos(2\pi x_2)} + \frac{2 + \sin(2\pi x_2)}{2 + 1.8 \cos(2\pi x_1)} \right)\textrm{Id}  
\end{equation}
for which the homogenized tensor is
\begin{equation*}
a^0 \approx
\begin{pmatrix}
2.757 & -0.002\\
-0.002 & 3.425
\end{pmatrix} .
\end{equation*}
Here, we compare the performances of the described parabolic approach (``par.'' in the legends), the standard elliptic approach (``ell.'' in the legends) \add{and the regularised elliptic approach of Ref. \refcite{Glo11}}. 
In comparing the 
\add{different} methods, we used a filtered version of \eqref{eq:bar_a0_Intro}, namely
\begin{equation}
	\mathbf{e}_i \cdot a^{0}_{R,L} \mathbf{e}_j := \int_{K_L} \mathbf{e}_i \cdot a(\bx) \left( \mathbf{e}_j + \nabla \chi_R^j(\bx) \right) \mu_L(\bx)\,d\bx, 
\end{equation}
\corr{that improves the error constant for the classical approach. However, we recall that} the standard elliptic method provides a first order convergence rate, independently of the use of oversampling or filtering, as shown in Ref. \refcite{YuE07}. 
By contrast, the use of high order filters in the parabolic \add{and regularised elliptic} schemes improves the convergence rate without affecting the computational cost. 

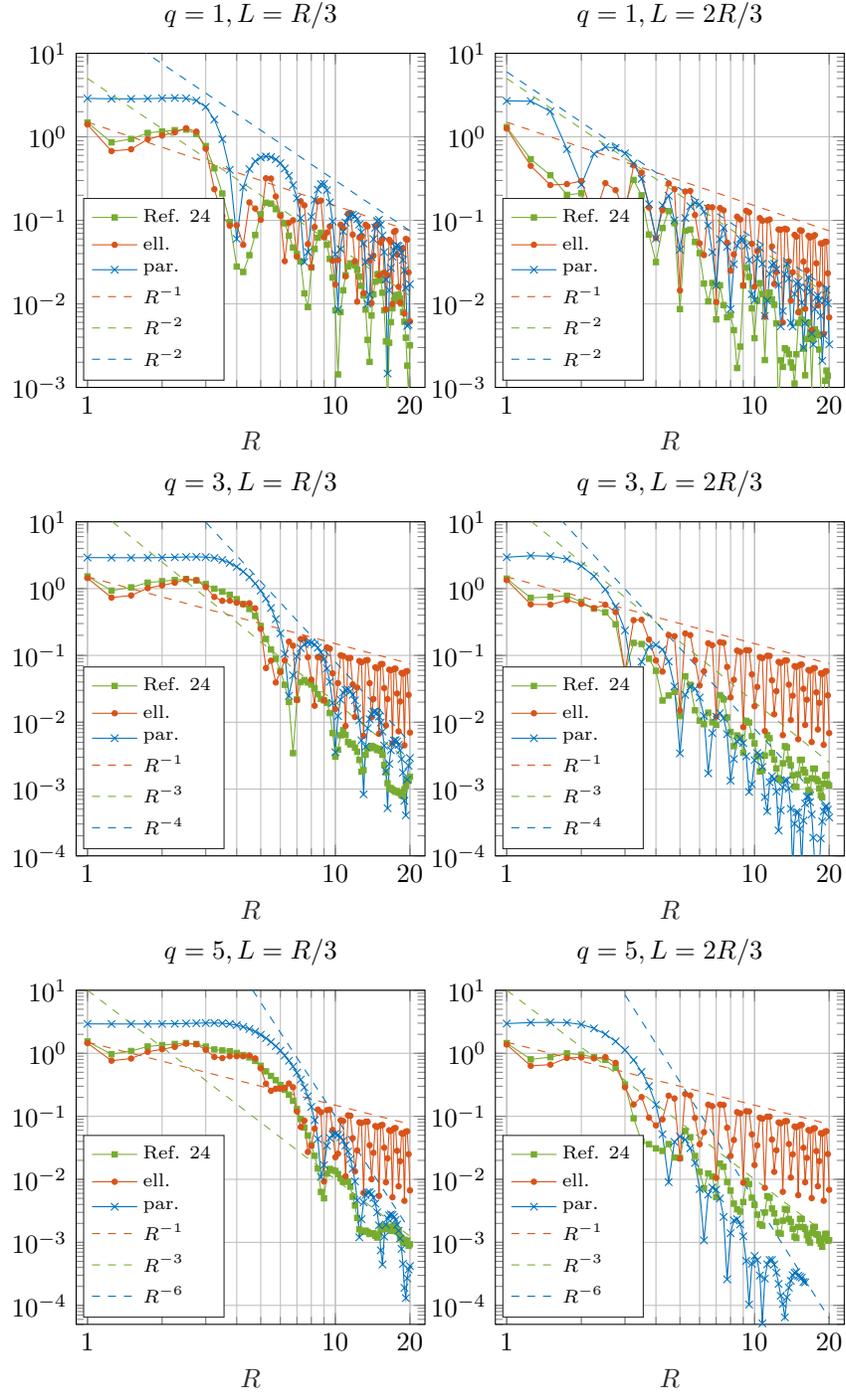
\begin{figure}
	\centering
	\begin{subfigure}{.45\textwidth}
%
%
\definecolor{mycolor1}{rgb}{0.85098,0.32549,0.09804}%
\definecolor{mycolor2}{rgb}{0.00000,0.44706,0.74118}%
\definecolor{mycolor3}{rgb}{0.46667,0.67451,0.18824}%
\begin{tikzpicture}

\begin{axis}[%
width=.85\linewidth,
height=1.75in,
scale only axis,
xmode=log,
xmin=.9,
xmax=23,
xtick={1, 10, 20},
xticklabels={1, 10, 20},
minor xtick={1,2,3,4,5,6,7,8,9,10,20},
xlabel style={font=\color{white!15!black}},
xlabel={$R$},
ymode=log,
ymin=1e-03,
ymax=10,
yminorticks=true,
axis background/.style={fill=white},
title={$q=1, L = R/3$},
xmajorgrids,
xminorgrids,
ymajorgrids,
legend style={font = \scriptsize, legend cell align=left, align=left, draw=white!15!black,at={(0.02,0.02)},anchor=south west }
]

\addplot [color=mycolor3, mark size=1.0pt, mark=square*, mark options={solid, fill=mycolor3, mycolor3}]
table[row sep=crcr]{%
	1	1.48304858858517\\
	1.25	0.858856824420499\\
	1.5	0.943361716110429\\
	1.75	1.1128646318186\\
	2	1.15984432605648\\
	2.25	1.20124982262663\\
	2.5	1.22053583314424\\
	2.75	1.08308574598526\\
	3	0.77567353930962\\
	3.25	0.418852357833339\\
	3.5	0.209583205185018\\
	3.75	0.0956297495720021\\
	4	0.028003322709267\\
	4.25	0.0240988007101022\\
	4.5	0.0382884956518259\\
	4.75	0.0667068130141423\\
	5	0.118368442292016\\
	5.25	0.162827588864628\\
	5.5	0.158720866218248\\
	5.75	0.13605295872319\\
	6	0.116533951026455\\
	6.25	0.092286006012209\\
	6.5	0.0655965570225135\\
	6.75	0.047323663401962\\
	7	0.0395291413010995\\
	7.25	0.0323049139642803\\
	7.5	0.0133587182637502\\
	7.75	0.00917169504792652\\
	8	0.0272948778038346\\
	8.25	0.0453899674461602\\
	8.5	0.0619170115656071\\
	8.75	0.0711000576528569\\
	9	0.0665941653917333\\
	9.25	0.0509640221343735\\
	9.5	0.0345611694640113\\
	9.75	0.0206397089827285\\
	10	0.0083744673664169\\
	10.25	0.00143004800499613\\
	10.5	0.00797126835709024\\
	10.75	0.0145448013668715\\
	11	0.0220310393897051\\
	11.25	0.0280987791317462\\
	11.5	0.030622514406389\\
	11.75	0.0302617618002137\\
	12	0.0288905897577256\\
	12.25	0.0255787186151764\\
	12.5	0.0211836630696374\\
	12.75	0.0164674354792968\\
	13	0.0126213244017655\\
	13.25	0.00854663238834757\\
	13.5	0.00371085712562731\\
	13.75	0.00202214311406181\\
	14	0.00724744326578973\\
	14.25	0.0129995951731239\\
	14.5	0.0183946386923089\\
	14.75	0.0225509288042927\\
	15	0.0228045316917797\\
	15.25	0.0191737443498556\\
	15.5	0.0135831639964263\\
	15.75	0.00837551917188971\\
	16	0.0035502054186627\\
	16.25	0.000505646096959137\\
	16.5	0.00342332657271756\\
	16.75	0.00604245185281777\\
	17	0.00884985260294999\\
	17.25	0.0109798625865586\\
	17.5	0.0124673099205944\\
	17.75	0.0128085390653877\\
	18	0.0127183687925333\\
	18.25	0.0115810364269018\\
	18.5	0.0100672547659932\\
	18.75	0.00797163220166792\\
	19	0.00611405374096785\\
	19.25	0.00393780109485855\\
	19.5	0.00182332582421652\\
	19.75	0.00080971151301347\\
	20	0.00319070544850068\\
};
\addlegendentry{Ref. \refcite{Glo11}}

\addplot [color=mycolor1, mark size=1.0pt, mark=*, mark options={solid, fill=mycolor1, mycolor1}]
  table[row sep=crcr]{%
1	1.40619109562964\\
1.25	0.674657146158381\\
1.5	0.710204297429784\\
1.75	0.935653580060475\\
2	1.0357959640597\\
2.25	1.13028761054714\\
2.5	1.26778024010669\\
2.75	1.15380860498848\\
3	0.719637858891809\\
3.25	0.236063313497818\\
3.5	0.112204935194693\\
3.75	0.0865065390129017\\
4	0.08728767537572\\
4.25	0.0509450576955498\\
4.5	0.163996181654066\\
4.75	0.138414521204661\\
5	0.101394619691568\\
5.25	0.319135876130245\\
5.5	0.316580140271763\\
5.75	0.19308276792426\\
6	0.113515464986025\\
6.25	0.0325258852723912\\
6.5	0.0935129834201454\\
6.75	0.100324915965569\\
7	0.036660221466891\\
7.25	0.169347811450028\\
7.5	0.151268666263082\\
7.75	0.0519540492346947\\
8	0.027515090618861\\
8.25	0.0832644960233152\\
8.5	0.168208521060748\\
8.75	0.172298503036765\\
9	0.0627525135411405\\
9.25	0.0714022017372483\\
9.5	0.0847063641222406\\
9.75	0.0329639910899711\\
10	0.0170628144535195\\
10.25	0.0335635904869293\\
10.5	0.0886823627855895\\
10.75	0.0806604584958567\\
11	0.0213655329118775\\
11.25	0.119574290305631\\
11.5	0.122533111443517\\
11.75	0.0675225447711821\\
12	0.032454662470527\\
12.25	0.0106286606261416\\
12.5	0.0623141601140278\\
12.75	0.0638723745547512\\
13	0.0134377921601834\\
13.25	0.0885727229407801\\
13.5	0.084691167648967\\
13.75	0.0330402219510255\\
14	0.0102026569149643\\
14.25	0.0376929929813594\\
14.5	0.0857501663515227\\
14.75	0.0876052740543824\\
15	0.0219674319220181\\
15.25	0.0541812314466615\\
15.5	0.0597581789070327\\
15.75	0.0236732404836015\\
16	0.00857731153901795\\
16.25	0.0220809104905289\\
16.5	0.0584734765956983\\
16.75	0.05486512461523\\
17	0.00995829135726453\\
17.25	0.0726905784894806\\
17.5	0.0749052140788359\\
17.75	0.0389059082695535\\
18	0.0160334984338971\\
18.25	0.0103482331684888\\
18.5	0.0457329421430909\\
18.75	0.0463289677637356\\
19	0.00763465881353891\\
19.25	0.0598636358009294\\
19.5	0.0586774837117626\\
19.75	0.0238506618361213\\
20	0.00616011394296699\\
};
\addlegendentry{ell.}

\addplot [color=mycolor2, mark=x, mark options={solid, mycolor2}]
table[row sep=crcr]{%
	1	2.88032268639381\\
	1.25	2.8578227152736\\
	1.5	2.84888278747283\\
	1.75	2.86421230071756\\
	2	2.89681184699705\\
	2.25	2.92162948230245\\
	2.5	2.89073828694902\\
	2.75	2.71633201017295\\
	3	2.28320254512093\\
	3.25	1.61471265701826\\
	3.5	0.938105133809764\\
	3.75	0.402301776466359\\
	4	0.0599903151814594\\
	4.25	0.248562472026544\\
	4.5	0.411270818830364\\
	4.75	0.511289642643707\\
	5	0.564952059806721\\
	5.25	0.582643752540652\\
	5.5	0.57129562298554\\
	5.75	0.536560524250932\\
	6	0.483607724110818\\
	6.25	0.417381472658816\\
	6.5	0.342663799996605\\
	6.75	0.263711329713154\\
	7	0.184017819408299\\
	7.25	0.106148182589976\\
	7.5	0.0317249512115377\\
	7.75	0.0417094943815774\\
	8	0.110599474579188\\
	8.25	0.175961709216267\\
	8.5	0.233327840888352\\
	8.75	0.272029075042864\\
	9	0.273575860517768\\
	9.25	0.230263595160246\\
	9.5	0.163818358970491\\
	9.75	0.0969562566279923\\
	10	0.0391317232080748\\
	10.25	0.00857413257671297\\
	10.5	0.0451362125994061\\
	10.75	0.0734816308226788\\
	11	0.0943651855712777\\
	11.25	0.108382974409379\\
	11.5	0.115840367622264\\
	11.75	0.117093668199221\\
	12	0.112586984456029\\
	12.25	0.103005602466243\\
	12.5	0.089205798406636\\
	12.75	0.0721400681314681\\
	13	0.0527439861499135\\
	13.25	0.0317939701163182\\
	13.5	0.00999184449722641\\
	13.75	0.0128923156253428\\
	14	0.0362051322064179\\
	14.25	0.0596733707835999\\
	14.5	0.0816475667675346\\
	14.75	0.09806635514442\\
	15	0.101716770541303\\
	15.25	0.0886307688970786\\
	15.5	0.065860131010025\\
	15.75	0.0416309392357127\\
	16	0.0196382486761485\\
	16.25	0.00146559520417779\\
	16.5	0.0145664232321028\\
	16.75	0.0269922350377316\\
	17	0.036710756804667\\
	17.25	0.0437782329382029\\
	17.5	0.0482210263858976\\
	17.75	0.0499852344633207\\
	18	0.0491706694739406\\
	18.25	0.0459609023713128\\
	18.5	0.0406454783900729\\
	18.75	0.0335724861872352\\
	19	0.0251319180724091\\
	19.25	0.0156577804887128\\
	19.5	0.00544792595040343\\
	19.75	0.00562273694939462\\
	20	0.0171657650295987\\
};
\addlegendentry{par.}

\addplot [color=mycolor1, dashed]
  table[row sep=crcr]{%
1	1.5\\
20	0.075\\
};
\addlegendentry{$R^{-1}$}

\addplot [color=mycolor3, dashed]
table[row sep=crcr]{%
	1	5\\
	20	0.0125\\
};
\addlegendentry{$R^{-2}$}

\addplot [color=mycolor2, dashed]
  table[row sep=crcr]{%
1   30\\
20	0.075\\
};
\addlegendentry{$R^{-2}$}

\end{axis}
\end{tikzpicture}%
	\end{subfigure}
	\begin{subfigure}{.45\textwidth}
%
%
\definecolor{mycolor1}{rgb}{0.85098,0.32549,0.09804}%
\definecolor{mycolor2}{rgb}{0.00000,0.44706,0.74118}%
\definecolor{mycolor3}{rgb}{0.46667,0.67451,0.18824}%
\begin{tikzpicture}

\begin{axis}[%
width=.85\linewidth,
height=1.75in,
scale only axis,
xmode=log,
xmin=.9,
xmax=23,
xtick={1, 10, 20},
xticklabels={1, 10, 20},
minor xtick={1,2,3,4,5,6,7,8,9,10,20},
xlabel style={font=\color{white!15!black}},
xlabel={$R$},
ymode=log,
ymin=1e-03,
ymax=10,
yminorticks=true,
axis background/.style={fill=white},
title={$q=1, L = 2R/3$},
xmajorgrids,
xminorgrids,
ymajorgrids,
legend style={font = \scriptsize, legend cell align=left, align=left, draw=white!15!black,at={(0.02,0.02)},anchor=south west }
]

\addplot [color=mycolor3, mark size=1.0pt, mark=square*, mark options={solid, fill=mycolor3, mycolor3}]
table[row sep=crcr]{%
	1	1.30703235660541\\
	1.25	0.540671977572887\\
	1.5	0.347151354131193\\
	1.75	0.202093324856024\\
	2	0.210779857438291\\
	2.25	0.0807914651864034\\
	2.5	0.0914501832508923\\
	2.75	0.0411574949053817\\
	3	0.153890958360785\\
	3.25	0.304751722118306\\
	3.5	0.199257594833263\\
	3.75	0.0678110359190927\\
	4	0.031643900589295\\
	4.25	0.0809196531970569\\
	4.5	0.12898344176342\\
	4.75	0.0924871545250898\\
	5	0.00866680893408591\\
	5.25	0.0750817276193714\\
	5.5	0.0775449317765619\\
	5.75	0.0551087864061405\\
	6	0.0404982010020126\\
	6.25	0.0175136490430332\\
	6.5	0.0106989949244626\\
	6.75	0.0177869952374328\\
	7	0.00652154889302917\\
	7.25	0.0166915090559399\\
	7.5	0.00966739927488489\\
	7.75	0.0069189279306414\\
	8	0.00740022003387602\\
	8.25	0.00485158639529968\\
	8.5	0.00170693983240225\\
	8.75	0.00308048355370693\\
	9	0.0165561976438418\\
	9.25	0.0273085280490583\\
	9.5	0.0209970799155252\\
	9.75	0.00881805740807302\\
	10	0.00388396829600947\\
	10.25	0.0109182051263059\\
	10.5	0.017726140421927\\
	10.75	0.0131181679342943\\
	11	0.00212462218168327\\
	11.25	0.0114806023287794\\
	11.5	0.0136430322569116\\
	11.75	0.0113645520682519\\
	12	0.00975415405191534\\
	12.25	0.00582623751272837\\
	12.5	0.0005384016020264\\
	12.75	0.00210850227425837\\
	13	0.00176044556872688\\
	13.25	0.00293528267376997\\
	13.5	0.00280514595088686\\
	13.75	0.00249964744326001\\
	14	0.00199136325034511\\
	14.25	0.00172735563598735\\
	14.5	0.00111431962380365\\
	14.75	0.00249053206542175\\
	15	0.00591476800621042\\
	15.25	0.00812550680232361\\
	15.5	0.00646598454115011\\
	15.75	0.00290585145603143\\
	16	0.00139485590831119\\
	16.25	0.00380888673344816\\
	16.5	0.00596149015869244\\
	16.75	0.00442120602727326\\
	17	0.000973371063095817\\
	17.25	0.00376790091650541\\
	17.5	0.00495039640307946\\
	17.75	0.00450947609998555\\
	18	0.00424935964634326\\
	18.25	0.00284027421125989\\
	18.5	0.00097515901530491\\
	18.75	0.000488656399564658\\
	19	0.000829835213890506\\
	19.25	0.00120184162189146\\
	19.5	0.00159001076111774\\
	19.75	0.00137904181339902\\
	20	0.000901218274154357\\
};
\addlegendentry{Ref. \refcite{Glo11}}

\addplot [color=mycolor1, mark size=1.0pt, mark=*, mark options={solid, fill=mycolor1, mycolor1}]
  table[row sep=crcr]{%
1	1.25907553778113\\
1.25	0.450040353169005\\
1.5	0.264365319190342\\
1.75	0.272017000942328\\
2	0.295980919430859\\
2.25	0.106490468353775\\
2.5	0.279131821495914\\
2.75	0.23003291730919\\
3	0.124860688742488\\
3.25	0.455150383865544\\
3.5	0.369639131966222\\
3.75	0.142581675865515\\
4	0.061316330971224\\
4.25	0.132059616703507\\
4.5	0.2769937934146\\
4.75	0.236097156982674\\
5	0.0144832816291603\\
5.25	0.221377448680942\\
5.5	0.225052672889655\\
5.75	0.115834186705192\\
6	0.0530819845466711\\
6.25	0.0384167569007387\\
6.5	0.144149042615842\\
6.75	0.143128005404981\\
7	0.0105315809789373\\
7.25	0.140513046741654\\
7.5	0.131766946555637\\
7.75	0.0540836409839441\\
8	0.0249422217586781\\
8.25	0.0391923190407052\\
8.5	0.110617646554687\\
8.75	0.102430571122861\\
9	0.0165997036879923\\
9.25	0.130188034908825\\
9.5	0.124776909778034\\
9.75	0.0524311661478698\\
10	0.016045737006667\\
10.25	0.0430517722557689\\
10.5	0.10637069165008\\
10.75	0.0990040940046772\\
11	0.00703373683174112\\
11.25	0.0999872075343142\\
11.5	0.102575260275974\\
11.75	0.0485714078326802\\
12	0.017323779541212\\
12.25	0.0241053318513305\\
12.5	0.0777552382727551\\
12.75	0.0773366107221505\\
13	0.00598685131619237\\
13.25	0.0793670680538808\\
13.5	0.0778870553457152\\
13.75	0.0323058706209186\\
14	0.0109299217642724\\
14.25	0.0239228917835964\\
14.5	0.0670507394405822\\
14.75	0.0639245213901345\\
15	0.00794498809771737\\
15.25	0.0761189830650821\\
15.5	0.075203642766148\\
15.75	0.0320381974732089\\
16	0.00872580703355161\\
16.25	0.0252288050799958\\
16.5	0.0653083999556678\\
16.75	0.0623060094106738\\
17	0.00484286630237255\\
17.25	0.0646934368424984\\
17.5	0.0665468210462258\\
17.75	0.0306226296540405\\
18	0.00959080693737778\\
18.25	0.0172224625142368\\
18.5	0.053087733950815\\
18.75	0.0528497125942039\\
19	0.00436807939657807\\
19.25	0.0554499171673234\\
19.5	0.0553159711865845\\
19.75	0.023223528090399\\
20	0.00688323327162958\\
};
\addlegendentry{ell.}

\addplot [color=mycolor2, mark=x, mark options={solid, mycolor2}]
  table[row sep=crcr]{%
1	2.69191204349882\\
1.25	2.67228493682958\\
1.5	2.01925683760912\\
1.75	0.711963782101301\\
2	0.267788152132169\\
2.25	0.625431851892859\\
2.5	0.752776938287267\\
2.75	0.738530178659733\\
3	0.634660803850525\\
3.25	0.480474474103293\\
3.5	0.313748504457981\\
3.75	0.156325132726446\\
4	0.0607472285110975\\
4.25	0.151559338810484\\
4.5	0.19708788301312\\
4.75	0.0998884288612033\\
5	0.0444614691517024\\
5.25	0.106752534816826\\
5.5	0.147571503970901\\
5.75	0.162886061943493\\
6	0.15454141390188\\
6.25	0.126845574418821\\
6.5	0.0867100459112479\\
6.75	0.0415057547105125\\
7	0.0166482370288636\\
7.25	0.059536922405922\\
7.5	0.0815470615215207\\
7.75	0.0482758826852867\\
8	0.00868098639169664\\
8.25	0.0298183908562659\\
8.5	0.0499779844803187\\
8.75	0.0599883627516203\\
9	0.0596468431733869\\
9.25	0.0499548913350065\\
9.5	0.0334274736259845\\
9.75	0.0129973757638773\\
10	0.0111040434310791\\
10.25	0.0348984834954146\\
10.5	0.046948886286816\\
10.75	0.0302352672256674\\
11	0.00718309923615202\\
11.25	0.0109564079711684\\
11.5	0.0228978086865593\\
11.75	0.0294513082579619\\
12	0.0302689932449706\\
12.25	0.0256820155973894\\
12.5	0.0168850898794992\\
12.75	0.00538231763500974\\
13	0.00822781627012121\\
13.25	0.0226472993028183\\
13.5	0.0301206701782839\\
13.75	0.0201516752968502\\
14	0.00591494113332732\\
14.25	0.00534286500394683\\
14.5	0.0131201297263648\\
14.75	0.0175890317356252\\
15	0.0184289659545931\\
15.25	0.01581100292388\\
15.5	0.0103776979139831\\
15.75	0.00299841308716373\\
16	0.00587050553416504\\
16.25	0.0154573964363862\\
16.5	0.0206010606002872\\
16.75	0.0140190337766058\\
17	0.00442054299279373\\
17.25	0.00325639634347229\\
17.5	0.0086949129624312\\
17.75	0.0118928838523444\\
18	0.012626573207927\\
18.25	0.0109187485998866\\
18.5	0.00720419104945996\\
18.75	0.00207992256002248\\
19	0.00419000978638787\\
19.25	0.0110589143626232\\
19.5	0.0148034655563171\\
19.75	0.0101537075596771\\
20	0.00326833620975826\\
};
\addlegendentry{par.}

\addplot [color=mycolor1, dashed]
  table[row sep=crcr]{%
1	1.5\\
20	0.075\\
};
\addlegendentry{$R^{-1}$}

\addplot [color=mycolor3, dashed]
table[row sep=crcr]{%
	1	5\\
	20	0.0125\\
};
\addlegendentry{$R^{-2}$}

\addplot [color=mycolor2, dashed]
  table[row sep=crcr]{%
1   6\\
20	0.015\\
};
\addlegendentry{$R^{-2}$}

\end{axis}
\end{tikzpicture}%
	\end{subfigure}	
	\begin{subfigure}{.45\textwidth}
%
%
\definecolor{mycolor1}{rgb}{0.85098,0.32549,0.09804}%
\definecolor{mycolor2}{rgb}{0.00000,0.44706,0.74118}%
\definecolor{mycolor3}{rgb}{0.46667,0.67451,0.18824}%
\begin{tikzpicture}

\begin{axis}[%
width=.85\linewidth,
height=1.75in,
scale only axis,
xmode=log,
xmin=.9,
xmax=23,
xtick={1, 10, 20},
xticklabels={1, 10, 20},
minor xtick={1,2,3,4,5,6,7,8,9,10,20},
xlabel style={font=\color{white!15!black}},
xlabel={$R$},
ymode=log,
ymin=1e-04,
ymax=10,
yminorticks=true,
axis background/.style={fill=white},
title={$q=3, L = R/3$},
xmajorgrids,
xminorgrids,
ymajorgrids,
legend style={font = \scriptsize, legend cell align=left, align=left, draw=white!15!black,at={(0.02,0.02)},anchor=south west }
]

\addplot [color=mycolor3, mark size=1.0pt, mark=square*, mark options={solid, fill=mycolor3, mycolor3}]
table[row sep=crcr]{%
	1	1.52305518090729\\
	1.25	0.93424494133082\\
	1.5	1.04359934967177\\
	1.75	1.22865677947391\\
	2	1.28734950626895\\
	2.25	1.34247193640998\\
	2.5	1.38971688176963\\
	2.75	1.32567776109541\\
	3	1.16674732646012\\
	3.25	0.994036785258021\\
	3.5	0.899696985266346\\
	3.75	0.808850816756644\\
	4	0.693961804701806\\
	4.25	0.588278090424894\\
	4.5	0.495231538309016\\
	4.75	0.390447791732288\\
	5	0.275504995733238\\
	5.25	0.17646008083734\\
	5.5	0.122769547880096\\
	5.75	0.0889961886181333\\
	6	0.058495068440531\\
	6.25	0.0363887916662509\\
	6.5	0.0200965591195947\\
	6.75	0.00344531609988374\\
	7	0.0219192103489637\\
	7.25	0.0397889859299573\\
	7.5	0.0428021115919268\\
	7.75	0.0394007113149532\\
	8	0.0366514798612469\\
	8.25	0.0319861512210002\\
	8.5	0.0262442829845288\\
	8.75	0.0223180778101603\\
	9	0.0212924680726812\\
	9.25	0.019568206058084\\
	9.5	0.0137387792578637\\
	9.75	0.00687668471202904\\
	10	0.00304529220885533\\
	10.25	0.00388812760171917\\
	10.5	0.00637644360992221\\
	10.75	0.00758512069340406\\
	11	0.00646384300913471\\
	11.25	0.00504591687425637\\
	11.5	0.00466167823187231\\
	11.75	0.00479783156471495\\
	12	0.004176589157873\\
	12.25	0.00361919440910051\\
	12.5	0.00272478070669144\\
	12.75	0.00181651596319274\\
	13	0.00192123842889558\\
	13.25	0.00326381790008064\\
	13.5	0.00406869700593962\\
	13.75	0.00410692825026935\\
	14	0.00442310348464054\\
	14.25	0.00425467342479764\\
	14.5	0.00411793752267878\\
	14.75	0.00383240077873216\\
	15	0.00400602930748216\\
	15.25	0.00376763964950857\\
	15.5	0.00325387022416974\\
	15.75	0.00229121150361804\\
	16	0.00178584847765537\\
	16.25	0.00122347026746166\\
	16.5	0.00102957661815447\\
	16.75	0.00103328684321473\\
	17	0.000985923868644271\\
	17.25	0.000947782628589953\\
	17.5	0.000909452854875491\\
	17.75	0.000931471580420751\\
	18	0.000885570967449248\\
	18.25	0.000875598760530768\\
	18.5	0.00080922389304488\\
	18.75	0.000787902795932845\\
	19	0.000993370938587155\\
	19.25	0.00113053116765933\\
	19.5	0.00139282078954641\\
	19.75	0.00133452675205627\\
	20	0.0015318474525467\\
};
\addlegendentry{Ref. \refcite{Glo11}}

\addplot [color=mycolor1, mark size=1.0pt, mark=*, mark options={solid, fill=mycolor1, mycolor1}]
  table[row sep=crcr]{%
1	1.43663022084111\\
1.25	0.734384719003068\\
1.5	0.788964641371725\\
1.75	1.01587512546924\\
2	1.11805954798519\\
2.25	1.2224140245601\\
2.5	1.38102330391643\\
2.75	1.3348201542983\\
3	1.05310387296556\\
3.25	0.749431712980688\\
3.5	0.65578676272883\\
3.75	0.658631380052708\\
4	0.614768029635039\\
4.25	0.582394378603333\\
4.5	0.603226618356791\\
4.75	0.506724046244167\\
5	0.248863595190394\\
5.25	0.0647911115724934\\
5.5	0.0834385043668323\\
5.75	0.0393690991174557\\
6	0.0562556195466575\\
6.25	0.0844585112514045\\
6.5	0.160989713661212\\
6.75	0.139688600246646\\
7	0.0215104131399853\\
7.25	0.176120786484896\\
7.5	0.178123707472603\\
7.75	0.0941918008868735\\
8	0.0434547132982576\\
8.25	0.0176600901455184\\
8.5	0.0943447248352325\\
8.75	0.0924875761034577\\
9	0.0212578511295095\\
9.25	0.132043651404634\\
9.5	0.126162785922249\\
9.75	0.0538536897242764\\
10	0.0157619308987861\\
10.25	0.038581779595008\\
10.5	0.100457268513218\\
10.75	0.098014139214905\\
11	0.00881309592051586\\
11.25	0.0912781502229118\\
11.5	0.091642500832302\\
11.75	0.0367509443225155\\
12	0.0120182232289469\\
12.25	0.0329930838449585\\
12.5	0.0825452031360433\\
12.75	0.0789000751779749\\
13	0.00627386090690187\\
13.25	0.0839640435284253\\
13.5	0.0852358570026281\\
13.75	0.0383878442560128\\
14	0.0118512686178174\\
14.25	0.0223013822078206\\
14.5	0.0662447017910616\\
14.75	0.0644715650797423\\
15	0.00673581360174178\\
15.25	0.0742816702134065\\
15.5	0.0743653591319728\\
15.75	0.0324084104139046\\
16	0.0089467898610299\\
16.25	0.0222291615805384\\
16.5	0.0611947009477185\\
16.75	0.0599570403144296\\
17	0.00481636284374748\\
17.25	0.062749511348761\\
17.5	0.06346664554775\\
17.75	0.0269004477258792\\
18	0.00731052800664505\\
18.25	0.0204507341631335\\
18.5	0.0546913256159826\\
18.75	0.0532474811478857\\
19	0.00451921458711567\\
19.25	0.0572936435481203\\
19.5	0.0583044628162185\\
19.75	0.0256223360565623\\
20	0.00701191405807704\\
};
\addlegendentry{ell.}

\addplot [color=mycolor2, mark=x, mark options={solid, mycolor2}]
  table[row sep=crcr]{%
1	2.91393024222153\\
1.25	2.90298098664955\\
1.5	2.89565033283202\\
1.75	2.89857503678291\\
2	2.91489140928025\\
2.25	2.94203747224747\\
2.5	2.97131991631188\\
2.75	2.98709517452142\\
3	2.96402520121976\\
3.25	2.86961182853285\\
3.5	2.68536155816025\\
3.75	2.42386330356859\\
4	2.11637105443954\\
4.25	1.79443737648823\\
4.5	1.48146594695263\\
4.75	1.19173628714965\\
5	0.932427165424872\\
5.25	0.705972626275362\\
5.5	0.512021960553457\\
5.75	0.348845233939546\\
6	0.214153043083749\\
6.25	0.10565414881387\\
6.5	0.0248204846805126\\
6.75	0.0503925507334319\\
7	0.0966983710395456\\
7.25	0.128579414174754\\
7.5	0.147535129415068\\
7.75	0.155498883814548\\
8	0.154311613007027\\
8.25	0.145560898929469\\
8.5	0.130700183629964\\
8.75	0.111074417836894\\
9	0.0881731990049686\\
9.25	0.0638661953437568\\
9.5	0.040325928567698\\
9.75	0.0193866090119931\\
10	0.00355696933150718\\
10.25	0.0123721612603916\\
10.5	0.0218103552730251\\
10.75	0.0277753385743197\\
11	0.0306457778036973\\
11.25	0.0309141268583722\\
11.5	0.0291086850085134\\
11.75	0.0257148008360543\\
12	0.0212420702901707\\
12.25	0.0161072745902258\\
12.5	0.0106966275931093\\
12.75	0.00535020817586437\\
13	0.000833647468645165\\
13.25	0.00430798362858419\\
13.5	0.00812288187802275\\
13.75	0.0111087957774812\\
14	0.0131985254135048\\
14.25	0.0143594291824337\\
14.5	0.0145536715764209\\
14.75	0.0138091539510956\\
15	0.0121728815425425\\
15.25	0.0099037542457816\\
15.5	0.00724561883422807\\
15.75	0.0044600855925506\\
16	0.00185051740901488\\
16.25	0.000516393175128905\\
16.5	0.00238167866528473\\
16.75	0.0038148858169932\\
17	0.00478566977159672\\
17.25	0.00531956522729678\\
17.5	0.00540529316448428\\
17.75	0.00514792102117933\\
18	0.0045882059005828\\
18.25	0.00379002066072904\\
18.5	0.00282121423694917\\
18.75	0.00175978294023734\\
19	0.000667612466926564\\
19.25	0.000404720569041419\\
19.5	0.0013837264527642\\
19.75	0.00223301893127241\\
20	0.00290719207044658\\
};
\addlegendentry{par.}

\addplot [color=mycolor1, dashed]
  table[row sep=crcr]{%
1	1.5\\
20	0.075\\
};
\addlegendentry{$R^{-1}$}

\addplot [color=mycolor3, dashed]
table[row sep=crcr]{%
	1	20\\
	20	0.0025\\
};
\addlegendentry{$R^{-3}$}

\addplot [color=mycolor2, dashed]
  table[row sep=crcr]{%
1   800\\
20	0.005\\
};
\addlegendentry{$R^{-4}$}

\end{axis}
\end{tikzpicture}%
	\end{subfigure}
	\begin{subfigure}{.45\textwidth}
%
%
\definecolor{mycolor1}{rgb}{0.85098,0.32549,0.09804}%
\definecolor{mycolor2}{rgb}{0.00000,0.44706,0.74118}%
\definecolor{mycolor3}{rgb}{0.46667,0.67451,0.18824}%
\begin{tikzpicture}

\begin{axis}[%
width=.85\linewidth,
height=1.75in,
scale only axis,
xmode=log,
xmin=.9,
xmax=23,
xtick={1, 10, 20},
xticklabels={1, 10, 20},
minor xtick={1,2,3,4,5,6,7,8,9,10,20},
xlabel style={font=\color{white!15!black}},
xlabel={$R$},
ymode=log,
ymin=1e-04,
ymax=10,
yminorticks=true,
axis background/.style={fill=white},
title={$q=3, L = 2R/3$},
xmajorgrids,
xminorgrids,
ymajorgrids,
legend style={font = \scriptsize, legend cell align=left, align=left, draw=white!15!black,at={(0.02,0.02)},anchor=south west }
]

\addplot [color=mycolor3, mark size=1.0pt, mark=square*, mark options={solid, fill=mycolor3, mycolor3}]
table[row sep=crcr]{%
	1	1.40991717466915\\
	1.25	0.72687452643004\\
	1.5	0.750744583148926\\
	1.75	0.782127791492547\\
	2	0.630787424331795\\
	2.25	0.503531592875639\\
	2.5	0.442516089656808\\
	2.75	0.292444658185343\\
	3	0.0557577807262799\\
	3.25	0.153970585202807\\
	3.5	0.148206689615617\\
	3.75	0.0892482353370219\\
	4	0.0588844836084294\\
	4.25	0.021128505211281\\
	4.5	0.0251157006402753\\
	4.75	0.0284411513416057\\
	5	0.0123634188197279\\
	5.25	0.0487209235144348\\
	5.5	0.0363380808207947\\
	5.75	0.0150835756903718\\
	6	0.0105812780434982\\
	6.25	0.00933193437220018\\
	6.5	0.0139135790473812\\
	6.75	0.00996438117657202\\
	7	0.00914768221236798\\
	7.25	0.0255000661126042\\
	7.5	0.0219131074611783\\
	7.75	0.0117978927483207\\
	8	0.0071672693788024\\
	8.25	0.00437774869543733\\
	8.5	0.00635163880454294\\
	8.75	0.00591074872123803\\
	9	0.0034901020656288\\
	9.25	0.0105183932327771\\
	9.5	0.00991039698997072\\
	9.75	0.00618366747344664\\
	10	0.00482291196275822\\
	10.25	0.00310016640084944\\
	10.5	0.00183849158858411\\
	10.75	0.00189309496850512\\
	11	0.00317384356450604\\
	11.25	0.00639244322875352\\
	11.5	0.00576414160625658\\
	11.75	0.00346456845304282\\
	12	0.00277024037545913\\
	12.25	0.00203621855959981\\
	12.5	0.00145238355299121\\
	12.75	0.00135887338351323\\
	13	0.0024303047331045\\
	13.25	0.0043781800338656\\
	13.5	0.00420849761156113\\
	13.75	0.00278368981857251\\
	14	0.00226854359078401\\
	14.25	0.00160827355639799\\
	14.5	0.00109248442895539\\
	14.75	0.00107549985435936\\
	15	0.00166079315838734\\
	15.25	0.0027243083525285\\
	15.5	0.00274588261341484\\
	15.75	0.00194572473287476\\
	16	0.0017734527238861\\
	16.25	0.00134906981145804\\
	16.5	0.000957201886972889\\
	16.75	0.000891079400411347\\
	17	0.00144705374993611\\
	17.25	0.00202145764030445\\
	17.5	0.00201523239745664\\
	17.75	0.00140714758408822\\
	18	0.00131354082775624\\
	18.25	0.00102253455053138\\
	18.5	0.000809301715535878\\
	18.75	0.000755151472368397\\
	19	0.00119503745971773\\
	19.25	0.00157464264574724\\
	19.5	0.00164005563108664\\
	19.75	0.00118398973708174\\
	20	0.00113117601770751\\
};
\addlegendentry{Ref. \refcite{Glo11}}

\addplot [color=mycolor1, mark size=1.0pt, mark=*, mark options={solid, fill=mycolor1, mycolor1}]
  table[row sep=crcr]{%
1	1.34156504940221\\
1.25	0.581219142922743\\
1.5	0.575128619402185\\
1.75	0.670734617458375\\
2	0.589644384796332\\
2.25	0.509156552602509\\
2.5	0.571635598917922\\
2.75	0.445552475897999\\
3	0.056145553262856\\
3.25	0.336799140506301\\
3.5	0.341225075274196\\
3.75	0.172785759421934\\
4	0.0838988770726446\\
4.25	0.0572454895487309\\
4.5	0.201857078680531\\
4.75	0.193904894915254\\
5	0.0149871314357203\\
5.25	0.213324881564269\\
5.5	0.20161630101531\\
5.75	0.0848473441805444\\
6	0.0355547317019381\\
6.25	0.0579354286315614\\
6.5	0.152140468646186\\
6.75	0.142209355975622\\
7	0.0121752456751561\\
7.25	0.157927050636306\\
7.5	0.154306129035663\\
7.75	0.0673068161604748\\
8	0.0235531698541899\\
8.25	0.0431841722532737\\
8.5	0.118148607055655\\
8.75	0.113120298799818\\
9	0.00846798298250039\\
9.25	0.119941319182637\\
9.5	0.119375078995574\\
9.75	0.0521229521406694\\
10	0.017383581279427\\
10.25	0.0337701778634147\\
10.5	0.0944720970301092\\
10.75	0.0911963853506274\\
11	0.00738743008120816\\
11.25	0.0986381337882526\\
11.5	0.0983537623499005\\
11.75	0.0423740935453786\\
12	0.0131535647336681\\
12.25	0.0291713049526786\\
12.5	0.0800994666165957\\
12.75	0.0774332532136282\\
13	0.00640914800980135\\
13.25	0.0838379425641481\\
13.5	0.0841936314273383\\
13.75	0.0365238023558777\\
14	0.0108084171355695\\
14.25	0.0249656280402829\\
14.5	0.0691391847144703\\
14.75	0.0672503047278651\\
15	0.00556666066006357\\
15.25	0.0724219023200526\\
15.5	0.073074361311986\\
15.75	0.0317001493174756\\
16	0.00911463900120728\\
16.25	0.0217557870668143\\
16.5	0.0605605917401606\\
16.75	0.0590789786598764\\
17	0.00504568945272039\\
17.25	0.0640477470438434\\
17.5	0.0647263134891237\\
17.75	0.0280102459705329\\
18	0.00780038016375042\\
18.25	0.0195099148442885\\
18.5	0.054109041463264\\
18.75	0.0528573632590349\\
19	0.00459138631994045\\
19.25	0.057436485003095\\
19.5	0.0582123329470321\\
19.75	0.0252681803885188\\
20	0.00688555113213407\\  	
};
\addlegendentry{ell.}

\addplot [color=mycolor2, mark=x, mark options={solid, mycolor2}]
  table[row sep=crcr]{%
1	2.94698436084933\\
1.25	3.1063773390963\\
1.5	3.05184551950613\\
1.75	2.76224821652847\\
2	2.18192240062822\\
2.25	1.52985301083937\\
2.5	0.968297379319691\\
2.75	0.540303252038773\\
3	0.236378837421077\\
3.25	0.0375773108768422\\
3.5	0.0804918362026736\\
3.75	0.134773332689709\\
4	0.143945219373663\\
4.25	0.122149286288635\\
4.5	0.0810602335605433\\
4.75	0.0343986278418164\\
5	0.00342028359565053\\
5.25	0.0256973568184043\\
5.5	0.0338829344728852\\
5.75	0.0317587524211076\\
6	0.0234084808272195\\
6.25	0.0124525816484361\\
6.5	0.00171857527089363\\
6.75	0.00687744985359154\\
7	0.0121705181409682\\
7.25	0.0136751230836607\\
7.5	0.011474044931212\\
7.75	0.00663324169396187\\
8	0.00133632006601774\\
8.25	0.00281287103070458\\
8.5	0.00516575700118595\\
8.75	0.00573849717135802\\
9	0.00487577239913258\\
9.25	0.00309344364144376\\
9.5	0.000906707899447274\\
9.75	0.00117719444068837\\
10	0.00271091504502465\\
10.25	0.00345534752960689\\
10.5	0.00321174649613521\\
10.75	0.00214245127472125\\
11	0.000756070455056087\\
11.25	0.000464269394509113\\
11.5	0.00126749250118685\\
11.75	0.00156036224913248\\
12	0.00141284903083519\\
12.25	0.000920132478087001\\
12.5	0.000239534994129228\\
12.75	0.00048122114517202\\
13	0.00106467486060941\\
13.25	0.00134700502863873\\
13.5	0.00136156105539445\\
13.75	0.00101336667871013\\
14	0.000506999861736234\\
14.25	6.23322583449987e-05\\
14.5	0.000303275399115303\\
14.75	0.000451482580297215\\
15	0.000446577488157455\\
15.25	0.000253228824757347\\
15.5	7.20647306818667e-05\\
15.75	0.000337990488355857\\
16	0.000608274320502111\\
16.25	0.000769924451971754\\
16.5	0.00077762067149303\\
16.75	0.000638591435985106\\
17	0.00041727841949605\\
17.25	0.000187218651311539\\
17.5	6.86580160801753e-05\\
17.75	7.5384083893643e-05\\
18	7.25798464654352e-05\\
18.25	6.78115020575904e-05\\
18.5	0.000181914590340249\\
18.75	0.000327967405715162\\
19	0.000465498696472905\\
19.25	0.000539738092283203\\
19.5	0.00055851637035616\\
19.75	0.000491882231837855\\
20	0.000376089836753202\\  	
};
\addlegendentry{par.}

\addplot [color=mycolor1, dashed]
  table[row sep=crcr]{%
1	1.5\\
20	0.075\\
};
\addlegendentry{$R^{-1}$}

\addplot [color=mycolor3, dashed]
table[row sep=crcr]{%
	1	20\\
	20	0.0025\\
};
\addlegendentry{$R^{-3}$}

\addplot [color=mycolor2, dashed]
  table[row sep=crcr]{%
1   80\\
20	0.0005\\
};
\addlegendentry{$R^{-4}$}

\end{axis}
\end{tikzpicture}%
	\end{subfigure}	
	\begin{subfigure}{.45\textwidth}
%
%
\definecolor{mycolor1}{rgb}{0.85098,0.32549,0.09804}%
\definecolor{mycolor2}{rgb}{0.00000,0.44706,0.74118}%
\definecolor{mycolor3}{rgb}{0.46667,0.67451,0.18824}%
\begin{tikzpicture}

\begin{axis}[%
width=.85\linewidth,
height=1.75in,
scale only axis,
xmode=log,
xmin=.9,
xmax=23,
xtick={1, 10, 20},
xticklabels={1, 10, 20},
minor xtick={1,2,3,4,5,6,7,8,9,10,20},
xlabel style={font=\color{white!15!black}},
xlabel={$R$},
ymode=log,
ymin=5e-05,
ymax=10,
yminorticks=true,
axis background/.style={fill=white},
title={$q=5, L = R/3$},
xmajorgrids,
xminorgrids,
ymajorgrids,
legend style={font = \scriptsize, legend cell align=left, align=left, draw=white!15!black,at={(0.02,0.02)},anchor=south west }
]

\addplot [color=mycolor3, mark size=1.0pt, mark=square*, mark options={solid, fill=mycolor3, mycolor3}]
table[row sep=crcr]{%
	1	1.53778144967737\\
	1.25	0.9665760539315\\
	1.5	1.08865092686131\\
	1.75	1.28167880371375\\
	2	1.34676261688359\\
	2.25	1.40577274047952\\
	2.5	1.45681625971465\\
	2.75	1.4059603602181\\
	3	1.27571103117979\\
	3.25	1.14938420486164\\
	3.5	1.11125980906652\\
	3.75	1.07799468245348\\
	4	1.01240104353745\\
	4.25	0.939083209807734\\
	4.5	0.860352708449728\\
	4.75	0.756455579624967\\
	5	0.633083620594653\\
	5.25	0.516766674679479\\
	5.5	0.437065436179163\\
	5.75	0.37375730764904\\
	6	0.31387377689388\\
	6.25	0.26315009893528\\
	6.5	0.219984518224721\\
	6.75	0.176124602155143\\
	7	0.13040573216672\\
	7.25	0.0908540757661536\\
	7.5	0.0664151629024471\\
	7.75	0.0497724284817874\\
	8	0.0348931726406173\\
	8.25	0.0236271437838429\\
	8.5	0.0148358806153434\\
	8.75	0.00632177443780425\\
	9	0.00499727991896783\\
	9.25	0.0123371632036866\\
	9.5	0.0146186028397935\\
	9.75	0.0140347390786818\\
	10	0.0136059186425347\\
	10.25	0.01212714429078\\
	10.5	0.0103129676648404\\
	10.75	0.00911436201233749\\
	11	0.00934626116059163\\
	11.25	0.00922053656008136\\
	11.5	0.00763656816298556\\
	11.75	0.00529249305352903\\
	12	0.00383982845894732\\
	12.25	0.00243918450159751\\
	12.5	0.00161052055336233\\
	12.75	0.00149345815057586\\
	13	0.00149116895261013\\
	13.25	0.00153342602115256\\
	13.5	0.0014665843002086\\
	13.75	0.00135729825831365\\
	14	0.00139403390357396\\
	14.25	0.00139882157735864\\
	14.5	0.00127021051806709\\
	14.75	0.00119305438340642\\
	15	0.00132998640655597\\
	15.25	0.00159255339114179\\
	15.5	0.00181399583456687\\
	15.75	0.00164852289677234\\
	16	0.00177611114906928\\
	16.25	0.00162065276836118\\
	16.5	0.00163465579666184\\
	16.75	0.00152139060119703\\
	17	0.00177799005837659\\
	17.25	0.00177309624164826\\
	17.5	0.00180148545001028\\
	17.75	0.00148307240479126\\
	18	0.0014603831646175\\
	18.25	0.00120067571855395\\
	18.5	0.0011576495209308\\
	18.75	0.000984013350050475\\
	19	0.00109979638056309\\
	19.25	0.00101763048488954\\
	19.5	0.00106869261838327\\
	19.75	0.000879700104587318\\
	20	0.000932372331395383\\
};
\addlegendentry{Ref. \refcite{Glo11}}

\addplot [color=mycolor1, mark size=1.0pt, mark=*, mark options={solid, fill=mycolor1, mycolor1}]
  table[row sep=crcr]{%
1	1.45107310932514\\
1.25	0.762659965635439\\
1.5	0.825743116099335\\
1.75	1.05604218087363\\
2	1.15946505994696\\
2.25	1.26497242431464\\
2.5	1.42504969712291\\
2.75	1.38790912375697\\
3	1.13226919845808\\
3.25	0.875824256293475\\
3.5	0.839918520124978\\
3.75	0.898474585865183\\
4	0.899477749261431\\
4.25	0.895852757484313\\
4.5	0.924265656951975\\
4.75	0.829128046252268\\
5	0.576405595647951\\
5.25	0.328182048062141\\
5.5	0.25393697321479\\
5.75	0.273666820918522\\
6	0.275042322889833\\
6.25	0.28443223459254\\
6.5	0.331938374508637\\
6.75	0.289030697297087\\
7	0.119356852440026\\
7.25	0.0676784796945037\\
7.5	0.0853599027582132\\
7.75	0.0272038860686937\\
8	0.0342564641047653\\
8.25	0.0637861513441202\\
8.5	0.127881867625622\\
8.75	0.116054034857778\\
9	0.00925088124874197\\
9.25	0.125015486634135\\
9.5	0.126936618050634\\
9.75	0.060847936208253\\
10	0.022417801446319\\
10.25	0.025512716662723\\
10.5	0.0867776082102045\\
10.75	0.0843071220204191\\
11	0.0112100372075667\\
11.25	0.103558434060492\\
11.5	0.102180275861835\\
11.75	0.0450129693592092\\
12	0.0133789501044281\\
12.25	0.0287390871612102\\
12.5	0.080263451868337\\
12.75	0.0782635067001474\\
13	0.00607284726088959\\
13.25	0.0819274365570988\\
13.5	0.0822396980287552\\
13.75	0.0347566326534732\\
14	0.0100473535294677\\
14.25	0.0263866766734699\\
14.5	0.069995536364455\\
14.75	0.0678387867833568\\
15	0.00546365349545143\\
15.25	0.0720518426122002\\
15.5	0.0729272500175814\\
15.75	0.0317725323098519\\
16	0.00899772095214353\\
16.25	0.021623800949457\\
16.5	0.0601757122809323\\
16.75	0.0586439895978298\\
17	0.00519180076847877\\
17.25	0.0643602694305769\\
17.5	0.0650599585228756\\
17.75	0.0283693338121204\\
18	0.00775533741087036\\
18.25	0.0194247550476332\\
18.5	0.0539540686573465\\
18.75	0.0527971689118265\\
19	0.00455773035662372\\
19.25	0.057186826090581\\
19.5	0.0579568291451752\\
19.75	0.0250764484063984\\
20	0.00667999837873815\\
};
\addlegendentry{ell.}

\addplot [color=mycolor2, mark=x, mark options={solid, mycolor2}]
  table[row sep=crcr]{%
1	2.92689918700008\\
1.25	2.92148018931026\\
1.5	2.91729152271756\\
1.75	2.91839345238233\\
2	2.92794985423701\\
2.25	2.94676168739398\\
2.5	2.97270162903708\\
2.75	3.00077789461538\\
3	3.02285926408836\\
3.25	3.02684317704309\\
3.5	2.99773754965795\\
3.75	2.92315570585742\\
4	2.79903102502194\\
4.25	2.63032170507331\\
4.5	2.4277516095798\\
4.75	2.20393130753409\\
5	1.9707178029958\\
5.25	1.73784461865993\\
5.5	1.51254788642073\\
5.75	1.29979022483313\\
6	1.10268902055527\\
6.25	0.922945379846626\\
6.5	0.761214495827834\\
6.75	0.617449464179344\\
7	0.491149957709648\\
7.25	0.381480533402268\\
7.5	0.287382779322156\\
7.75	0.207743949809105\\
8	0.141358181406247\\
8.25	0.0870510811329436\\
8.5	0.0436753706502956\\
8.75	0.0109583605484765\\
9	0.0168055932182112\\
9.25	0.0340787306424638\\
9.5	0.0452739658061752\\
9.75	0.0511385072543934\\
10	0.0526442116875861\\
10.25	0.0508565034459488\\
10.5	0.0466832257662351\\
10.75	0.0409732791778276\\
11	0.0344164980371921\\
11.25	0.0275896067258869\\
11.5	0.0209222931377109\\
11.75	0.0147567885536807\\
12	0.0092774831458656\\
12.25	0.00465039711742496\\
12.5	0.00119996493657583\\
12.75	0.00233418174326731\\
13	0.00431124767905381\\
13.25	0.00563142133463415\\
13.5	0.00621347521343969\\
13.75	0.00626572426063853\\
14	0.00586829405481552\\
14.25	0.00511342460437621\\
14.5	0.00412615247663652\\
14.75	0.00299920354552248\\
15	0.00186179438201511\\
15.25	0.00074740930567144\\
15.5	0.000445065818841266\\
15.75	0.00123350676936554\\
16	0.00191316582794097\\
16.25	0.00239162654688259\\
16.5	0.00267235092143333\\
16.75	0.002776934700948\\
17	0.00272028373194317\\
17.25	0.00252187465310548\\
17.5	0.00226211168477179\\
17.75	0.00191702769452665\\
18	0.0015369283698097\\
18.25	0.00115206675010363\\
18.5	0.000788036125229282\\
18.75	0.000456446282044692\\
19	0.000189972774964211\\
19.25	0.000130652106140774\\
19.5	0.000268618957384343\\
19.75	0.000370380402744389\\
20	0.00042048664662784\\
};
\addlegendentry{par.}

\addplot [color=mycolor1, dashed]
  table[row sep=crcr]{%
1	1.5\\
20	0.075\\
};
\addlegendentry{$R^{-1}$}

\addplot [color=mycolor3, dashed]
table[row sep=crcr]{%
	1	10\\
	20	0.00125\\
};
\addlegendentry{$R^{-3}$}

\addplot [color=mycolor2, dashed]
  table[row sep=crcr]{%
1   100000\\
20	0.0015625\\
};
\addlegendentry{$R^{-6}$}

\end{axis}
\end{tikzpicture}%
	\end{subfigure}
	\begin{subfigure}{.45\textwidth}
%
%
\definecolor{mycolor1}{rgb}{0.85098,0.32549,0.09804}%
\definecolor{mycolor2}{rgb}{0.00000,0.44706,0.74118}%
\definecolor{mycolor3}{rgb}{0.46667,0.67451,0.18824}%
\begin{tikzpicture}

\begin{axis}[%
width=.85\linewidth,
height=1.75in,
scale only axis,
xmode=log,
xmin=.9,
xmax=23,
xtick={1, 10, 20},
xticklabels={1, 10, 20},
minor xtick={1,2,3,4,5,6,7,8,9,10,20},
xlabel style={font=\color{white!15!black}},
xlabel={$R$},
ymode=log,
ymin=5e-05,
ymax=10,
yminorticks=true,
axis background/.style={fill=white},
title={$q=5, L = 2R/3$},
xmajorgrids,
xminorgrids,
ymajorgrids,
legend style={font = \scriptsize, legend cell align=left, align=left, draw=white!15!black,at={(0.02,0.02)},anchor=south west }
]

\addplot [color=mycolor3, mark size=1.0pt, mark=square*, mark options={solid, fill=mycolor3, mycolor3}]
table[row sep=crcr]{%
	1	1.45149529661479\\
	1.25	0.801168728951878\\
	1.5	0.867735342708877\\
	1.75	0.994665384650648\\
	2	0.948218954694382\\
	2.25	0.867236302536218\\
	2.5	0.784647811706122\\
	2.75	0.593501825776606\\
	3	0.319748780965076\\
	3.25	0.0930451478270539\\
	3.5	0.0402483313412846\\
	3.75	0.0361641899549752\\
	4	0.0310000509042512\\
	4.25	0.0282457507533138\\
	4.5	0.0362659751642975\\
	4.75	0.0221698807233511\\
	5	0.0218173765598374\\
	5.25	0.0584410722401673\\
	5.5	0.0465899790381184\\
	5.75	0.0235034560710715\\
	6	0.0131626252829559\\
	6.25	0.00720466695734999\\
	6.5	0.0124058400892296\\
	6.75	0.0111459418873193\\
	7	0.00521587167048509\\
	7.25	0.0178217439255619\\
	7.5	0.0158209812231162\\
	7.75	0.00909746847304649\\
	8	0.00661818585801532\\
	8.25	0.0041171236483177\\
	8.5	0.00333920196035421\\
	8.75	0.00307389791886763\\
	9	0.00437165216868103\\
	9.25	0.00945718183696895\\
	9.5	0.00843007379982573\\
	9.75	0.0051121164878048\\
	10	0.00393542304679867\\
	10.25	0.00275284176405927\\
	10.5	0.00193348869400293\\
	10.75	0.00183753523606909\\
	11	0.00297946827387449\\
	11.25	0.00540980046440361\\
	11.5	0.00510722020351178\\
	11.75	0.00343941680055184\\
	12	0.0028731418847715\\
	12.25	0.00211440330534795\\
	12.5	0.00141723869290663\\
	12.75	0.0013744838241261\\
	13	0.00220793287167972\\
	13.25	0.00338264384115581\\
	13.5	0.00329249289368309\\
	13.75	0.00235358970701922\\
	14	0.0021117281908426\\
	14.25	0.00164648579107706\\
	14.5	0.00123322351083015\\
	14.75	0.0011656830227591\\
	15	0.00175580204378695\\
	15.25	0.00235769952344385\\
	15.5	0.00235274487048649\\
	15.75	0.00174484327968757\\
	16	0.00163526000279369\\
	16.25	0.00130527741758484\\
	16.5	0.00107540308972192\\
	16.75	0.000992468514627011\\
	17	0.00140213422098463\\
	17.25	0.00171530845279965\\
	17.5	0.00176651533635958\\
	17.75	0.00135021340515068\\
	18	0.0013217070548358\\
	18.25	0.00106797908430765\\
	18.5	0.000948439761128173\\
	18.75	0.000852221349592775\\
	19	0.00115164924426525\\
	19.25	0.00129948986984625\\
	19.5	0.00137393996727951\\
	19.75	0.00106502975627821\\
	20	0.00108477146355376\\
};
\addlegendentry{Ref. \refcite{Glo11}}

\addplot [color=mycolor1, mark size=1.0pt, mark=*, mark options={solid, fill=mycolor1, mycolor1}]
  table[row sep=crcr]{%
1	1.37559922266825\\
1.25	0.631907897614203\\
1.5	0.660396705841595\\
1.75	0.844736378809887\\
2	0.85793922509015\\
2.25	0.831706557106092\\
2.5	0.870647489809131\\
2.75	0.704603491705214\\
3	0.288392890684122\\
3.25	0.154457969692784\\
3.5	0.201813752963198\\
3.75	0.0954995774230498\\
4	0.0716695664492894\\
4.25	0.0895255165617995\\
4.5	0.212899595843847\\
4.75	0.192145103446339\\
5	0.0213513129159623\\
5.25	0.226215184452259\\
5.5	0.214823862107795\\
5.75	0.0942587497839496\\
6	0.0360842803786841\\
6.25	0.0559796466045435\\
6.5	0.153168795664482\\
6.75	0.145106641221001\\
7	0.0105359636523203\\
7.25	0.153821755355367\\
7.5	0.15122916782454\\
7.75	0.0659505106613681\\
8	0.0235275109479071\\
8.25	0.0421768891073765\\
8.5	0.116506971313888\\
8.75	0.11146441201262\\
9	0.0088467591277706\\
9.25	0.120592856356757\\
9.5	0.119456255270587\\
9.75	0.0516923555293623\\
10	0.0168554507396761\\
10.25	0.0346761220638227\\
10.5	0.0951152494213134\\
10.75	0.0915794286196042\\
11	0.00735085883653043\\
11.25	0.0986779126238691\\
11.5	0.0986275192972269\\
11.75	0.0427700839166393\\
12	0.0131565763844187\\
12.25	0.0289738585888185\\
12.5	0.079959975891064\\
12.75	0.0774545938276486\\
13	0.00633896435841781\\
13.25	0.083487133096789\\
13.5	0.083876195702028\\
13.75	0.0363377899033302\\
14	0.0106960958458417\\
14.25	0.0249570465710826\\
14.5	0.0689628765795455\\
14.75	0.0670451640171702\\
15	0.00561143793373458\\
15.25	0.0724592866292446\\
15.5	0.0730641796527109\\
15.75	0.0316595404686088\\
16	0.00899571427311253\\
16.25	0.0219133351214006\\
16.5	0.0606296020019421\\
16.75	0.0591173321884994\\
17	0.00503380303271656\\
17.25	0.0640012001753356\\
17.5	0.0647320575202348\\
17.75	0.0280742843127722\\
18	0.00776560814931706\\
18.25	0.0194924169361404\\
18.5	0.0540485957756139\\
18.75	0.0528294048390724\\
19	0.00457630211318116\\
19.25	0.0573284836958834\\
19.5	0.0581141732290846\\
19.75	0.0252208728882617\\
20	0.00682770782894624\\  	
};
\addlegendentry{ell.}

\addplot [color=mycolor2, mark=x, mark options={solid, mycolor2}]
  table[row sep=crcr]{%
1	2.96065043786042\\
1.25	3.07309667785524\\
1.5	3.10081821670731\\
1.75	3.06962137296763\\
2	2.86341924210699\\
2.25	2.47945606479657\\
2.5	2.01130045802253\\
2.75	1.54511320439291\\
3	1.12890392816236\\
3.25	0.782273528806163\\
3.5	0.507975080806208\\
3.75	0.30080589787174\\
4	0.152040226378583\\
4.25	0.0520021020034604\\
4.5	0.00903161946062605\\
4.75	0.039317983442921\\
5	0.0477971419818237\\
5.25	0.0426655005077773\\
5.5	0.0311168124071914\\
5.75	0.0182196361044268\\
6	0.00702721979067302\\
6.25	0.00107435741160238\\
6.5	0.00573517370847478\\
6.75	0.00741007075743779\\
7	0.00687025993327094\\
7.25	0.00497965412136498\\
7.5	0.00253124058643057\\
7.75	0.000259207819980264\\
8	0.00140028220694288\\
8.25	0.00222766833130398\\
8.5	0.00233314498440841\\
8.75	0.00192344479058927\\
9	0.00124412670571245\\
9.25	0.000508415023836305\\
9.5	0.000102900177554351\\
9.75	0.000457829120769752\\
10	0.000608400426965556\\
10.25	0.000521860071418423\\
10.5	0.000298230572009061\\
10.75	5.17874042862676e-05\\
11	0.000265990351177061\\
11.25	0.000448205301906395\\
11.5	0.000524967263025148\\
11.75	0.0005264578621253\\
12	0.00044414654628709\\
12.25	0.000330401786020955\\
12.5	0.00021942728940769\\
12.75	0.000142060940781473\\
13	0.000105089421338759\\
13.25	6.46541834352693e-05\\
13.5	0.000133968143370639\\
13.75	0.000191544530060711\\
14	0.000246494749452827\\
14.25	0.000301285357920866\\
14.5	0.000328770561338506\\
14.75	0.000337559173444639\\
15	0.000295278769071996\\
15.25	0.000285151692119151\\
15.5	0.000275174290521764\\
15.75	0.000235759566680781\\
16	0.000231219607285892\\
};
\addlegendentry{par.}

\addplot [color=mycolor1, dashed]
  table[row sep=crcr]{%
1	1.5\\
20	0.075\\
};
\addlegendentry{$R^{-1}$}

\addplot [color=mycolor3, dashed]
table[row sep=crcr]{%
	1	10\\
	20	0.00125\\
};
\addlegendentry{$R^{-3}$}

\addplot [color=mycolor2, dashed]
  table[row sep=crcr]{%
1   8000\\
20	0.0000625\\
};
\addlegendentry{$R^{-6}$}

\end{axis}
\end{tikzpicture}%
	\end{subfigure}	
	\caption{Comparison of the resonance error in the standard elliptic, regularised elliptic and parabolic models for the tensor \cref{eq: tensor GloriaLeBris}. The size of the sampling domain, $R$, is reported on the $x$-axis; the resonance error \eqref{eq: reserr} is reported on the $y$-axis.}
	\label{fig:compare}
\end{figure}

\add{All the} approaches are solved using $\mathbb{P}1$ finite element discretization in space with $64$ points per periodic cell. Mass lumping has been used in order to perform the time integration, which is carried out via the fourth order explicit stabilised ROCK4 method, see Ref. \refcite{Abd02}, with $tol=10^{-6}$ (since we want to kill the discretization error in time).
%
Finally Simpson's quadrature rule is used for computing the time integral defining homogenized coefficients.


Results are depicted in \cref{fig:compare}. As expected, one cannot reach a convergence rate higher than 1 for the standard elliptic approach, in contrast to the \add{regularised elliptic and} parabolic methods. 
\add{For low order filters ($q=1$), we observe that both the parabolic approach and the regularised elliptic approach have the same convergence rate but the latter has a smaller prefactor. As the parameter $q$ gets larger, the parabolic approach starts to result in smaller resonance errors asymptotically (as $R$ increases). The advantage of the parabolic approach over the regularised counterpart is more pronounced for larger values of $q$. This is also discussed in \Cref{sec:computational cost}, where we investigate the computational efficiency.}
\add{For all the three methods,} we notice a longer ``flat'' region in the convergence plot for small values of $k_o$ and high order filters. Intuitively, for any given $R$, the region where the filter is ``not almost zero'' decreases for smaller $k_o$ and larger $q$. Hence, we need larger values of $R$ for the averaging integral to contain enough data and the error to decrease with the expected rate.

\add{%
\begin{figure}[h!]
	\centering
	\input{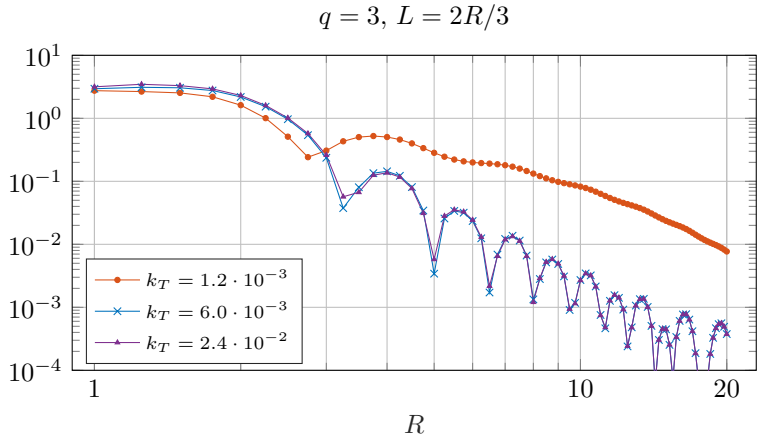}
	\caption{Effect of using different values of $k_T$. The size of the sampling domain, $R$, is reported on the $x$-axis; the resonance error \eqref{eq: reserr} is reported on the $y$-axis.}
	\label{fig:different kT}
\end{figure}
We conclude this section by showing the effect that different values of $k_T$ have on the convergence rate. In \Cref{fig:different kT}, the resonance error is plotted against the value of $R$, for the choice of $k_o=2/3$, $q=3$, and three different values for $k_T$. The results clearly show that if $k_T$ is much smaller than the optimal value then the error convergence is severely hampered (see the first plot in \Cref{fig:different kT}), so the error is dominated by the truncation error ($e^{-2\lambda_0 T}$ in \Cref{thm: modelling error parabolic correctors filter}). However, note that, as $R$ gets larger the exponentially decaying errors will be dominated by the averaging error and the asymptotic decay rate of \Cref{thm: modelling error parabolic correctors filter} will be recovered.
On the other hand, the results do not vary sensibly if larger values for $k_T$ are used, even though they are not the ``optimal'' ones, see the plots for $k_T=6.0\cdot 10^{-3}$ and  $k_T=2.4\cdot 10^{-2}$. In those cases the main contribution to the resonance error comes from the averaging error terms, which do not depend on $k_T$.
The fact that the same convergence behaviour is achieved for a large range of $k_T$ values allows to choose it with a high degree of flexibility. %
}
\subsection{\add{Non-smooth coefficients}}\label{subsec:numerical 2d nonsmooth periodic}
\add{%
In the example above, we have considered smooth coefficients. We will now present the results for an example with the non-smooth periodic coefficients:
\begin{equation}\label{eq:truncPyramids}
a(x) = \begin{cases}
		1 & x\in \mathbb{Z} + K_{1/4}, \\
		12\norm{x-[x]}_{\infty} -\frac{1}{2} & x\in \mathbb{Z} + (K_{3/4}\setminus K_{1/4}), \\
		4 & \text{elsewhere},
	\end{cases}
\end{equation}
where the components of $[x]$ are the integer values closest to $x$, component-wise.
The diffusion coefficient field is depicted in \cref{fig:truncPyramids}.
\begin{figure}[h!]
	\centering
	\includegraphics[width=.4\textwidth]{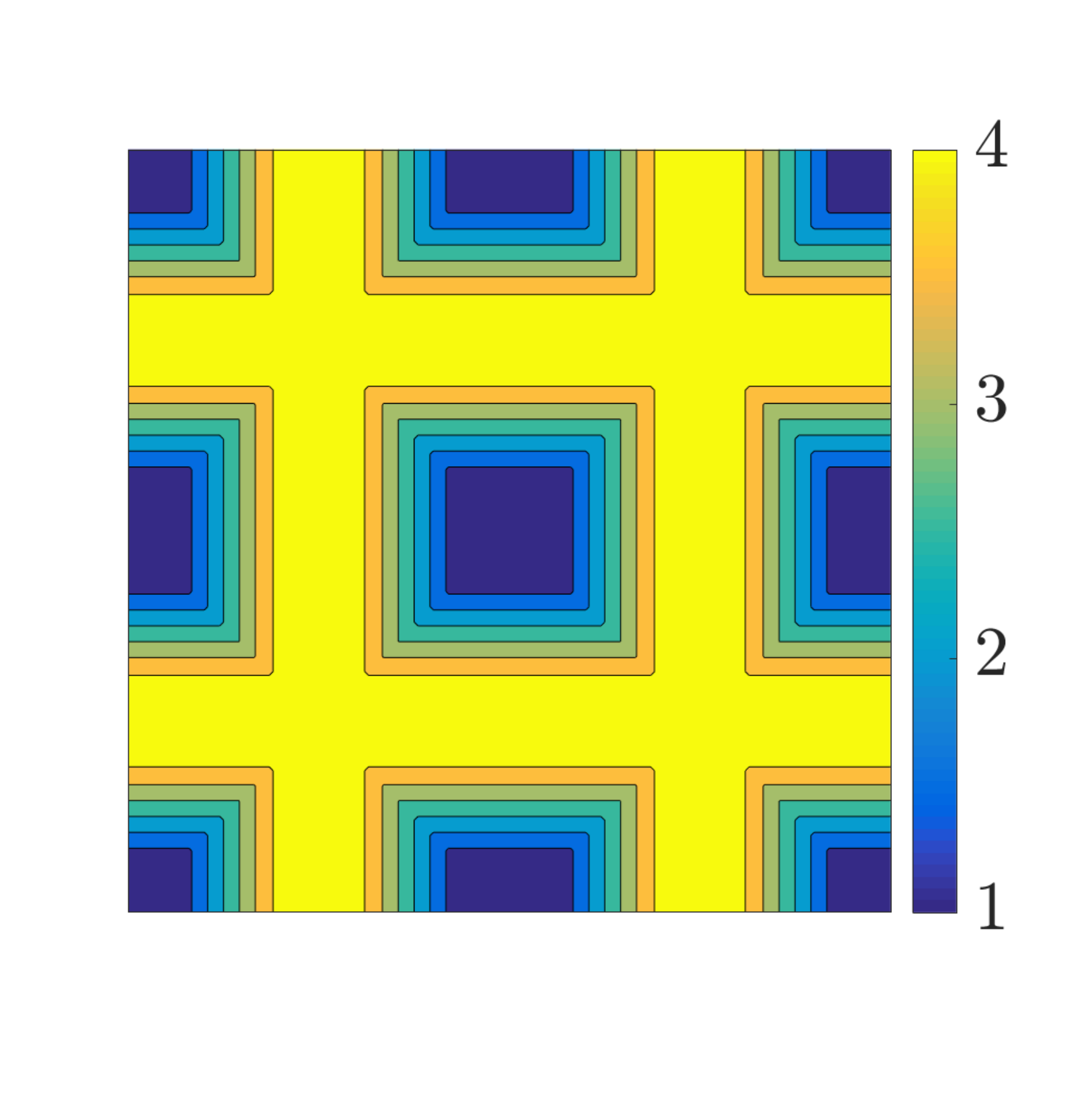}
	\caption{Visualization of the diffusion coefficient field \eqref{eq:truncPyramids}}
	\label{fig:truncPyramids}
\end{figure}
The field \eqref{eq:truncPyramids} could be considered as an approximation of a composite material made up of low-diffusive square inclusions into an homogeneous medium.
The case of discontinuous coefficients can still be tackled numerically, but an appropriate discontinuous Galerkin method needs to be implemented to get high orders in spatial discretization and thereby observe the asymptotic rate of the convergence.
Convergence plots pictured in \cref{fig:nonSmooth} demonstrates that the theoretical results also apply to the case of non-smooth coefficients. The test is done with $\mathbb{P}1$ finite element discretization on a uniform grid of size $h=1/128$ and the ROCK4 time integration scheme with $tol = 10^{-6}$. The final time $T$ has been chosen according to \eqref{eq:choice L T}.
The results agree with both the previous example and the theoretical convergence rates. As before, the effect of having a smaller ratio $L/R=k_o$ is the increase in the length of the initial ``flat'' region, indicating that larger values of $k_o$ are preferable.
The resonance error converges as $R^{-q-1}$, where $q$ is the order of the filters, independently from other parameters. 
Both the results for the regularised elliptic and parabolic methods are presented: the errors for both methods are comparable for $k_o=2/3$ and $q=2$, while the parabolic scheme shows faster decay of the error when $q=4,5$.}
\begin{figure}[h!]
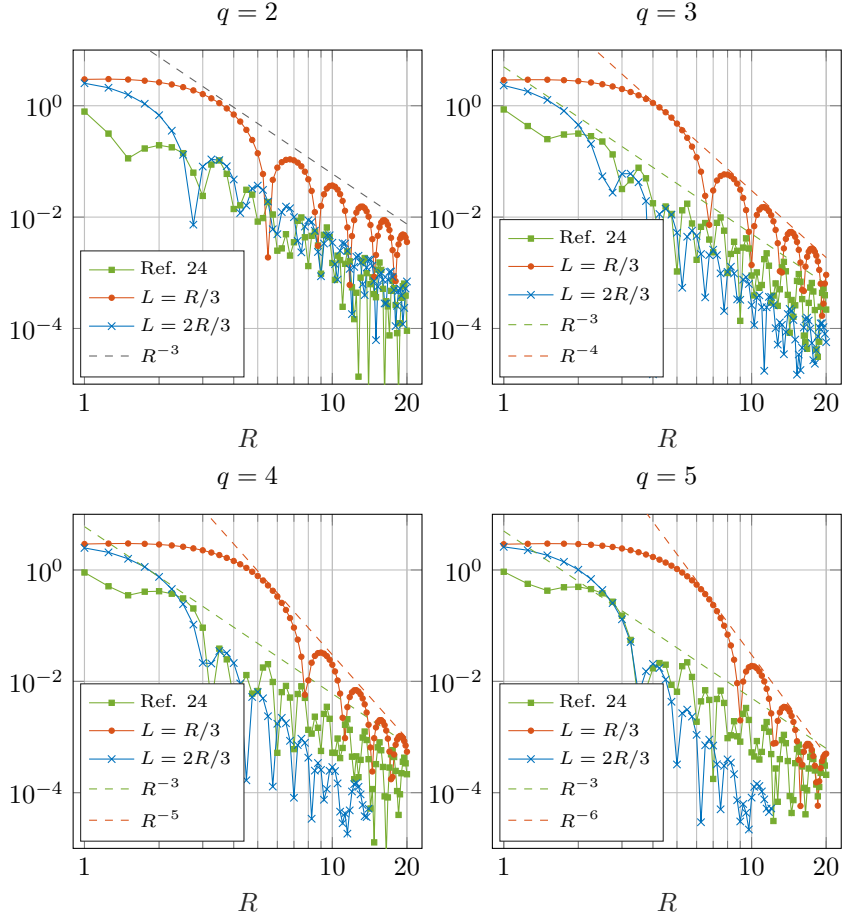

	\centering
	\begin{subfigure}[h]{.45\textwidth}
%
%
\definecolor{mycolor1}{rgb}{0.85098,0.32549,0.09804}%
\definecolor{mycolor2}{rgb}{0.00000,0.44706,0.74118}%
\definecolor{mycolor3}{rgb}{0.46667,0.67451,0.18824}%
\definecolor{mycolor4}{rgb}{0.38039,0.38039,0.38039}%
\begin{tikzpicture}
	
	\begin{axis}[%
		width=.85\linewidth,
		height=1.75in,
		scale only axis,
		xmode=log,
		xmin=.9,
		xmax=23,
		xtick={1, 10, 20},
		xticklabels={1, 10, 20},
		minor xtick={1,2,3,4,5,6,7,8,9,10,20},
		xlabel style={font=\color{white!15!black}},
		xlabel={$R$},
		ymode=log,
		ymin=1e-05,
		ymax=10,
		yminorticks=true,
		axis background/.style={fill=white},
		title={$q=2$},
		xmajorgrids,
		xminorgrids,
		ymajorgrids,
		yminorgrids,
		legend style={font=\scriptsize, legend cell align=left, align=left, draw=white!15!black,at={(0.02,0.02)},anchor=south west }
		]

\addplot [color=mycolor3, mark size=1.0pt, mark=square*, mark options={solid, fill=mycolor3, mycolor3}]
table[row sep=crcr]{%
	1	0.785879363821785\\
	1.25	0.316237945750281\\
	1.5	0.11345304116964\\
	1.75	0.170601484324546\\
	2	0.195044829925254\\
	2.25	0.17898439838605\\
	2.5	0.140812167619047\\
	2.75	0.0628918447581746\\
	3	0.0240275157202207\\
	3.25	0.0878770653754345\\
	3.5	0.104203473461307\\
	3.75	0.0597722686852201\\
	4	0.0139816668879351\\
	4.25	0.016277323973004\\
	4.5	0.0309687433659651\\
	4.75	0.0253535251926131\\
	5	0.00831723630432259\\
	5.25	0.00959549713456774\\
	5.5	0.0186635388510031\\
	5.75	0.0111413945632562\\
	6	0.0024971434903246\\
	6.25	0.002741420711113\\
	6.5	0.00492743489049628\\
	6.75	0.002037974703418\\
	7	0.00352880406115455\\
	7.25	0.00857266251758015\\
	7.5	0.0098547355035572\\
	7.75	0.00433975151642378\\
	8	0.00132141992348272\\
	8.25	0.00460535235526531\\
	8.5	0.00544658293667302\\
	8.75	0.00309230056359972\\
	9	0.000963877616180889\\
	9.25	0.00479590055561042\\
	9.5	0.00651233613963154\\
	9.75	0.00438619619645022\\
	10	0.00149982335978618\\
	10.25	0.000745717023187076\\
	10.5	0.00203533750419766\\
	10.75	0.00169097527043113\\
	11	0.000245100918091207\\
	11.25	0.00144749797914345\\
	11.5	0.00240977013278193\\
	11.75	0.00160613408019673\\
	12	0.000547976011855833\\
	12.25	0.000147010953913871\\
	12.5	0.000438775080357947\\
	12.75	1.36366121096527e-05\\
	13	0.000818244504366543\\
	13.25	0.00160176048480018\\
	13.5	0.00183978024528073\\
	13.75	0.000984934597166275\\
	14	7.11924859975213e-06\\
	14.25	0.000624105471740627\\
	14.5	0.000827758273014648\\
	14.75	0.00041469846746089\\
	15	0.000363996557253723\\
	15.25	0.0011416691805443\\
	15.5	0.00152850022673611\\
	15.75	0.00112569674119434\\
	16	0.000514819763506947\\
	16.25	5.28677970822906e-06\\
	16.5	0.000308202180314238\\
	16.75	0.00024933825211643\\
	17	7.5036710537338e-05\\
	17.25	0.000474291517329388\\
	17.5	0.000716229290265999\\
	17.75	0.000529027206263974\\
	18	0.000263517418360587\\
	18.25	8.23052559126111e-05\\
	18.5	5.392772846026e-06\\
	18.75	0.000123990688039401\\
	19	0.000355359885669025\\
	19.25	0.000573897824265698\\
	19.5	0.000639954764480402\\
	19.75	0.000389375251277475\\
	20	9.11942099537306e-05\\
};
\addlegendentry{Ref. \refcite{Glo11}}

\addplot [color=mycolor1, mark size=1.0pt, mark=*, mark options={solid, fill=mycolor1, mycolor1}]
  table[row sep=crcr]{%
1	2.97474906824926\\
1.25	3.01150053172779\\
1.5	2.95423043936651\\
1.75	2.81555055567317\\
2	2.6327508416919\\
2.25	2.41088375536959\\
2.5	2.15745435700262\\
2.75	1.88699835550928\\
3	1.61626687715983\\
3.25	1.35710606122253\\
3.5	1.11567470377502\\
3.75	0.894568004415171\\
4	0.695605821708877\\
4.25	0.520019670253342\\
4.5	0.368987550387115\\
4.75	0.242615683458913\\
5	0.140016280980241\\
5.25	0.0593657362909355\\
5.5	0.00188880278190727\\
5.75	0.046653859662887\\
6	0.0774836064486727\\
6.25	0.0967826538258492\\
6.5	0.10664229188132\\
6.75	0.10878550615491\\
7	0.104548301550618\\
7.25	0.0950994550821838\\
7.5	0.0816522959738502\\
7.75	0.0654574601456136\\
8	0.0476738672810408\\
8.25	0.0295351614229497\\
8.5	0.0122832908391842\\
8.75	0.00305750697388452\\
9	0.0158420705971201\\
9.25	0.0257132400954555\\
9.5	0.0325042361170663\\
9.75	0.0362020908309964\\
10	0.0369683195585194\\
10.25	0.0351329970181421\\
10.5	0.0312106579395314\\
10.75	0.0258035688344741\\
11	0.0195277924376579\\
11.25	0.0129456433691149\\
11.5	0.00652215382915739\\
11.75	0.000604758236870326\\
12	0.00456507443569845\\
12.25	0.00882086770393818\\
12.5	0.0120727838480706\\
12.75	0.0142708507396032\\
13	0.0153921000363443\\
13.25	0.0154500177724425\\
13.5	0.0145142494204538\\
13.75	0.012709765393081\\
14	0.0102060530578931\\
14.25	0.00722253394090716\\
14.5	0.0040183536497\\
14.75	0.000849099418085575\\
15	0.002083928248562\\
15.25	0.00461151136671832\\
15.5	0.00661447424525105\\
15.75	0.00800514298168554\\
16	0.00873238776179087\\
16.25	0.00879747859097604\\
16.5	0.0082646177583938\\
16.75	0.00724060138724579\\
17	0.00585573689658689\\
17.25	0.00424741025585317\\
17.5	0.00254678864284622\\
17.75	0.000867243181952166\\
18	0.000706947287490006\\
18.25	0.00209294396331184\\
18.5	0.00324181150442998\\
18.75	0.00410869031939166\\
19	0.00465942820209772\\
19.25	0.00487253426133106\\
19.5	0.00474606752588719\\
19.75	0.00429980833662985\\
20	0.0035752689294308\\
};
\addlegendentry{$L = R/3$}

\addplot [color=mycolor2, mark size=2.0pt, mark=x, mark options={solid, fill=mycolor2, mycolor2}]
table[row sep=crcr]{%
	1	2.53389932492851\\
	1.25	2.1031321835823\\
	1.5	1.58035942971353\\
	1.75	1.08622370880813\\
	2	0.676137668347889\\
	2.25	0.356076777132325\\
	2.5	0.131709674919168\\
	2.75	0.0073072201897454\\
	3	0.0808337737139145\\
	3.25	0.108663714803718\\
	3.5	0.105457512484043\\
	3.75	0.0818145203443968\\
	4	0.0473518193217077\\
	4.25	0.0118461700049604\\
	4.5	0.0161765811461961\\
	4.75	0.0325923783184493\\
	5	0.0368043412178898\\
	5.25	0.0308218224602808\\
	5.5	0.0190395305360782\\
	5.75	0.00604646881847119\\
	6	0.00493140419363186\\
	6.25	0.0122949973637404\\
	6.5	0.0154468113804887\\
	6.75	0.0144095189677437\\
	7	0.0099762335739821\\
	7.25	0.00374192635051339\\
	7.5	0.00232786135861162\\
	7.75	0.00676811797040907\\
	8	0.00876547103215433\\
	8.25	0.00818182987434783\\
	8.5	0.00569880868287967\\
	8.75	0.00236665892046404\\
	9	0.000866730198161361\\
	9.25	0.00335305456066872\\
	9.5	0.00470121931740432\\
	9.75	0.00471002221659727\\
	10	0.00347237242636408\\
	10.25	0.00142942878136658\\
	10.5	0.000755101346620211\\
	10.75	0.00249698696255767\\
	11	0.00341105114564855\\
	11.25	0.00334286405203705\\
	11.5	0.00247380530713181\\
	11.75	0.00117629273822638\\
	12	0.000181103819084277\\
	12.25	0.0012791746463386\\
	12.5	0.00193829803398884\\
	12.75	0.00202043234419694\\
	13	0.00152541625118475\\
	13.25	0.000623978707067456\\
	13.5	0.00039749783466575\\
	13.75	0.00124493612044772\\
	14	0.00172456097308567\\
	14.25	0.00173433640816888\\
	14.5	0.00133883700442013\\
	14.75	0.000707482385869581\\
	15	6.17749803330224e-05\\
	15.25	0.000563036489693531\\
	15.5	0.00092768935774738\\
	15.75	0.000995459437310409\\
	16	0.000752498091381953\\
	16.25	0.000281505390052083\\
	16.5	0.000281432802434253\\
	16.75	0.000753175935719761\\
	17	0.00103281967633475\\
	17.25	0.00105334814471305\\
	17.5	0.000842234285432379\\
	17.75	0.00048973349348219\\
	18	0.000110274977014472\\
	18.25	0.000255371196317484\\
	18.5	0.000474744839782264\\
	18.75	0.000523910018256933\\
	19	0.000388300783010914\\
	19.25	0.000119273627031268\\
	19.5	0.00023444526547224\\
	19.75	0.00052236586327095\\
	20	0.000698476361969905\\
};
\addlegendentry{$L = 2R/3$}

\addplot [color=mycolor4, dashed]
  table[row sep=crcr]{%
1	60\\
20	0.0075\\
};
\addlegendentry{$R^{-3}$}

\end{axis}
\end{tikzpicture}%
	\end{subfigure}
	\begin{subfigure}[h]{.45\textwidth}
		\input{"img/ALTER_nonSmooth/compareq=3Ldiff"}
	\end{subfigure}	

	\begin{subfigure}[h]{.45\textwidth}
%
%
\definecolor{mycolor1}{rgb}{0.85098,0.32549,0.09804}%
\definecolor{mycolor2}{rgb}{0.00000,0.44706,0.74118}%
\definecolor{mycolor3}{rgb}{0.46667,0.67451,0.18824}%
\definecolor{mycolor4}{rgb}{0.38039,0.38039,0.38039}%
\begin{tikzpicture}
	
	\begin{axis}[%
		width=.85\linewidth,
		height=1.75in,
		scale only axis,
		xmode=log,
		xmin=.9,
		xmax=23,
		xtick={1, 10, 20},
		xticklabels={1, 10, 20},
		minor xtick={1,2,3,4,5,6,7,8,9,10,20},
		xlabel style={font=\color{white!15!black}},
		xlabel={$R$},
		ymode=log,
		ymin=1e-05,
		ymax=10,
		yminorticks=true,
		axis background/.style={fill=white},
		title={$q=4$},
		xmajorgrids,
		xminorgrids,
		ymajorgrids,
		yminorgrids,
		legend style={font=\scriptsize, legend cell align=left, align=left, draw=white!15!black,at={(0.02,0.02)},anchor=south west }
		]
		
\addplot [color=mycolor3, mark size=1.0pt, mark=square*, mark options={solid, fill=mycolor3, mycolor3}]
  table[row sep=crcr]{%
	1	0.900916454431067\\
	1.25	0.509264120745785\\
	1.5	0.350835185975695\\
	1.75	0.407654282350716\\
	2	0.415155204252758\\
	2.25	0.374847898763255\\
	2.5	0.308214280415443\\
	2.75	0.203506794109749\\
	3	0.0919858843322193\\
	3.25	0.00399002282015104\\
	3.5	0.0386737446500871\\
	3.75	0.0248137885046809\\
	4	0.00480418665912459\\
	4.25	0.00778148682724686\\
	4.5	0.012964309712487\\
	4.75	0.00602121664904944\\
	5	0.00671338276174725\\
	5.25	0.0176931608100096\\
	5.5	0.0203261632260666\\
	5.75	0.00974058081334941\\
	6	0.000519833922752992\\
	6.25	0.00640124527695772\\
	6.5	0.00828126102253006\\
	6.75	0.00511057865743557\\
	7	0.000605996359916524\\
	7.25	0.00592411874401722\\
	7.5	0.00815043162764541\\
	7.75	0.00504076000628795\\
	8	0.00134787259667251\\
	8.25	0.00115232100253633\\
	8.5	0.0023183375985927\\
	8.75	0.00146332584779583\\
	9	0.000562766563981701\\
	9.25	0.00260245408204288\\
	9.5	0.00348754333476708\\
	9.75	0.00212411692756021\\
	10	0.0005236280964854\\
	10.25	0.000517593381244954\\
	10.5	0.000938314126190944\\
	10.75	0.00041780881582807\\
	11	0.000635829679841734\\
	11.25	0.00166220969189768\\
	11.5	0.00208638264070017\\
	11.75	0.0013451279582688\\
	12	0.000444134884124546\\
	12.25	0.000175084572014996\\
	12.5	0.000449094858528467\\
	12.75	0.000189937168123721\\
	13	0.000388663053719104\\
	13.25	0.000978548567727573\\
	13.5	0.00125397530856024\\
	13.75	0.000883014927102504\\
	14	0.00039985346612185\\
	14.25	5.20113318589733e-05\\
	14.5	0.000116382762225352\\
	14.75	1.28142798435572e-05\\
	15	0.00032603820969659\\
	15.25	0.000651467247859233\\
	15.5	0.000801918775296061\\
	15.75	0.000578090408255331\\
	16	0.000290044681317449\\
	16.25	8.59333261550286e-05\\
	16.5	7.15920796424602e-06\\
	16.75	8.22400796426341e-05\\
	17	0.00028318364114736\\
	17.25	0.000488606447219403\\
	17.5	0.000581867223005073\\
	17.75	0.000436980010853491\\
	18	0.000246461359402036\\
	18.25	0.000107054969282847\\
	18.5	3.99147235458512e-05\\
	18.75	9.3618702169136e-05\\
	19	0.000223505897304235\\
	19.25	0.000359990537795677\\
	19.5	0.000426014463476572\\
	19.75	0.000336634545426291\\
	20	0.000214566748866722\\
};
\addlegendentry{Ref. \refcite{Glo11}}
  	
  	\addplot[color=mycolor1, mark size=1.0pt, mark=*, mark options={solid, fill=mycolor1, mycolor1}] 
  	table[row sep=crcr]{%
  		1	2.92432452418902\\
  		1.25	2.97874703184443\\
  		1.5	2.98961470629122\\
  		1.75	2.94812828624641\\
  		2	2.87104040895127\\
  		2.25	2.75860351776989\\
  		2.5	2.61443723456019\\
  		2.75	2.44411481449107\\
  		3	2.25494993832527\\
  		3.25	2.05491510662576\\
  		3.5	1.85089935614559\\
  		3.75	1.64843432079256\\
  		4	1.45240915124549\\
  		4.25	1.2657860770035\\
  		4.5	1.09053572725305\\
  		4.75	0.928151364670607\\
  		5	0.77984604010426\\
  		5.25	0.646327942320648\\
  		5.5	0.527686393414567\\
  		5.75	0.423546189168509\\
  		6	0.333394265997039\\
  		6.25	0.25632240348617\\
  		6.5	0.191211693073638\\
  		6.75	0.136885658500972\\
  		7	0.0922256342143925\\
  		7.25	0.0561611274834543\\
  		7.5	0.0276562341599491\\
  		7.75	0.00574322823518855\\
  		8	0.0104029250492101\\
  		8.25	0.021584144431714\\
  		8.5	0.0285732807368465\\
  		8.75	0.0321019061733068\\
  		9	0.032846052843976\\
  		9.25	0.0314371255850231\\
  		9.5	0.0284653074715883\\
  		9.75	0.0244618631440855\\
  		10	0.0198651274021247\\
  		10.25	0.0150587468537274\\
  		10.5	0.0103627721291388\\
  		10.75	0.00602525215493961\\
  		11	0.00221788796562985\\
  		11.25	0.000956331162497223\\
  		11.5	0.00343731292213025\\
  		11.75	0.00522395911830068\\
  		12	0.00635732392123639\\
  		12.25	0.00690199049999826\\
  		12.5	0.00693834871888253\\
  		12.75	0.00655715546316298\\
  		13	0.00585535704353114\\
  		13.25	0.00492882989044236\\
  		13.5	0.00386706470011046\\
  		13.75	0.00275173961619713\\
  		14	0.00165830941106891\\
  		14.25	0.000649256829569933\\
  		14.5	0.000240539276086125\\
  		14.75	0.000952473250528211\\
  		15	0.00148715940901643\\
  		15.25	0.00183522905222459\\
  		15.5	0.00200370518420003\\
  		15.75	0.00200948772472836\\
  		16	0.0018755170817294\\
  		16.25	0.00163051442321088\\
  		16.5	0.00130631640249859\\
  		16.75	0.000935163758448571\\
  		17	0.000547567571422998\\
  		17.25	0.000174807170965028\\
  		17.5	0.000191728280188337\\
  		17.75	0.000485684445768061\\
  		18	0.000724325550239457\\
  		18.25	0.000897807469284952\\
  		18.5	0.0010039216348378\\
  		18.75	0.00104477373893021\\
  		19	0.00102596175170166\\
  		19.25	0.000955731061769414\\
  		19.5	0.000844235865927458\\
  		19.75	0.000703005220058095\\
  		20	0.000544384527366024\\
  	};
  	\addlegendentry{$L = R/3$}

\addplot [color=mycolor2, mark size=2.0pt, mark=x, mark options={solid, fill=mycolor2, mycolor2}]
table[row sep=crcr]{%
	1	2.48679072083329\\
	1.25	2.07179353570641\\
	1.5	1.59037835388893\\
	1.75	1.13430715633067\\
	2	0.751496278144204\\
	2.25	0.454747551602151\\
	2.5	0.243434153184076\\
	2.75	0.10453172091188\\
	3	0.021386930104722\\
	3.25	0.0208875829813927\\
	3.5	0.0350095842036202\\
	3.75	0.0321371687807077\\
	4	0.0212830869187369\\
	4.25	0.00916368240768589\\
	4.5	0.000167029778942755\\
	4.75	0.00523293054130354\\
	5	0.00640359589850836\\
	5.25	0.00495618098121324\\
	5.5	0.0023600892602947\\
	5.75	0.000128763605177209\\
	6	0.00173008788298852\\
	6.25	0.00220874219533317\\
	6.5	0.00177110651946684\\
	6.75	0.000855326859384464\\
	7	8.12018404039947e-05\\
	7.25	0.000716377346282868\\
	7.5	0.000932552372836631\\
	7.75	0.000786398581300585\\
	8	0.000425521696278785\\
	8.25	3.40526012734138e-05\\
	8.5	0.000249360748156991\\
	8.75	0.0003410597676943\\
	9	0.000257348220281648\\
	9.25	7.48307878514274e-05\\
	9.5	0.000116013725358204\\
	9.75	0.000247297434944231\\
	10	0.000288894932501992\\
	10.25	0.000244691265239832\\
	10.5	0.000147063284618199\\
	10.75	4.56629641907995e-05\\
	11	2.82658222623998e-05\\
	11.25	4.23510250374919e-05\\
	11.5	1.83824070449053e-05\\
	11.75	4.8164659481781e-05\\
	12	0.000105062063562816\\
	12.25	0.000142167137790178\\
	12.5	0.000148184369772014\\
	12.75	0.000127534848341919\\
	13	9.33769709543504e-05\\
	13.25	6.00836431810633e-05\\
	13.5	3.7823062976106e-05\\
	13.75	3.65525165458756e-05\\
	14	5.17107935483305e-05\\
};
\addlegendentry{$L = 2R/3$} 

\addplot [color=mycolor3, dashed]
table[row sep=crcr]{%
	1	6\\
	20	0.00075\\
};
\addlegendentry{$R^{-3}$}

\addplot [color=mycolor1, dashed]
  table[row sep=crcr]{%
1	3000\\
20	0.0009375\\
};
\addlegendentry{$R^{-5}$}


\end{axis}
\end{tikzpicture}%
	\end{subfigure}	
	\begin{subfigure}[h]{.45\textwidth}
		\input{"img/ALTER_nonSmooth/compareq=5Ldiff"}
	\end{subfigure}
	\caption{Comparison of the resonance error in the regularised elliptic and parabolic models for the non-smooth tensor \cref{eq:truncPyramids}. The size of the sampling domain, $R$, is reported on the $x$-axis; the resonance error \eqref{eq: reserr} is reported on the $y$-axis.}
	\label{fig:nonSmooth}
\end{figure}
\subsection{\add{A non-periodic case}}
In the last numerical test, we provide an example for a 
\add{non-periodic} tensor %
\add{which violates the periodicity assumption made in \Cref{sec:analysis}}.
With this test, we do not aim at proving any theoretical convergence rate of the error, but rather to verify numerically that the periodicity assumption is not necessary for achieving fast decaying rates of the boundary error.
We consider a single realization of a stationary log-normal random field with Gaussian isotropic covariance: 
\begin{equation}\label{eq: tensor log-normal}
\log a(\cdot) \sim \mathcal{N}(\mu,Cov(\bx-\by)), \quad Cov(\mathbf z) = \sigma^2 e^{-\frac{\abs{\mathbf z}^2}{2\ell^2}},
\end{equation}
where $\mu$ and $\sigma^2$ are the mean and the variance of the field and $\ell$ is the correlation length. An example of such a field is depicted in \cref{fig:normal field}.
\add{Here, the choice of $\ell=0.2$ has been made, in order to include enough oscillations in the unit cell $[-1/2, 1/2]^2$}
We are not interested in evaluating the statistical error, but only the boundary error, which is
\begin{equation*}
\norm{a^{0}_{R,L,T} - a^{0}_{\infty,L,T}}_F.
\end{equation*}
In practice, we will consider $a^{0}_{R_{max},L,T}$ for the large value $R_{max}=20$ in place of $a^{0}_{\infty,L,T}$ as a reference \corr{for evaluating the resonance error. The new reference $a^{0}_{R_{max},L,T}$ is computed using the numerical approximation of the parabolic corrector on $K_{R_{max}}$ with periodic BCs}. 
The test is done with a $\mathbb{P}1$ finite element discretization on a uniform grid of size $h=1/20$ and the ROCK4 time integration scheme with $tol = 10^{-5}$.
\corr{In \cref{fig:convergence plot stochastic} we show that the resonance error decays with a rate comprised between $3$ and $4$ with respect to $R$.}
\begin{figure}
	\centering
	\begin{subfigure}[b]{.45\textwidth}
		\includegraphics[width = 1\textwidth]{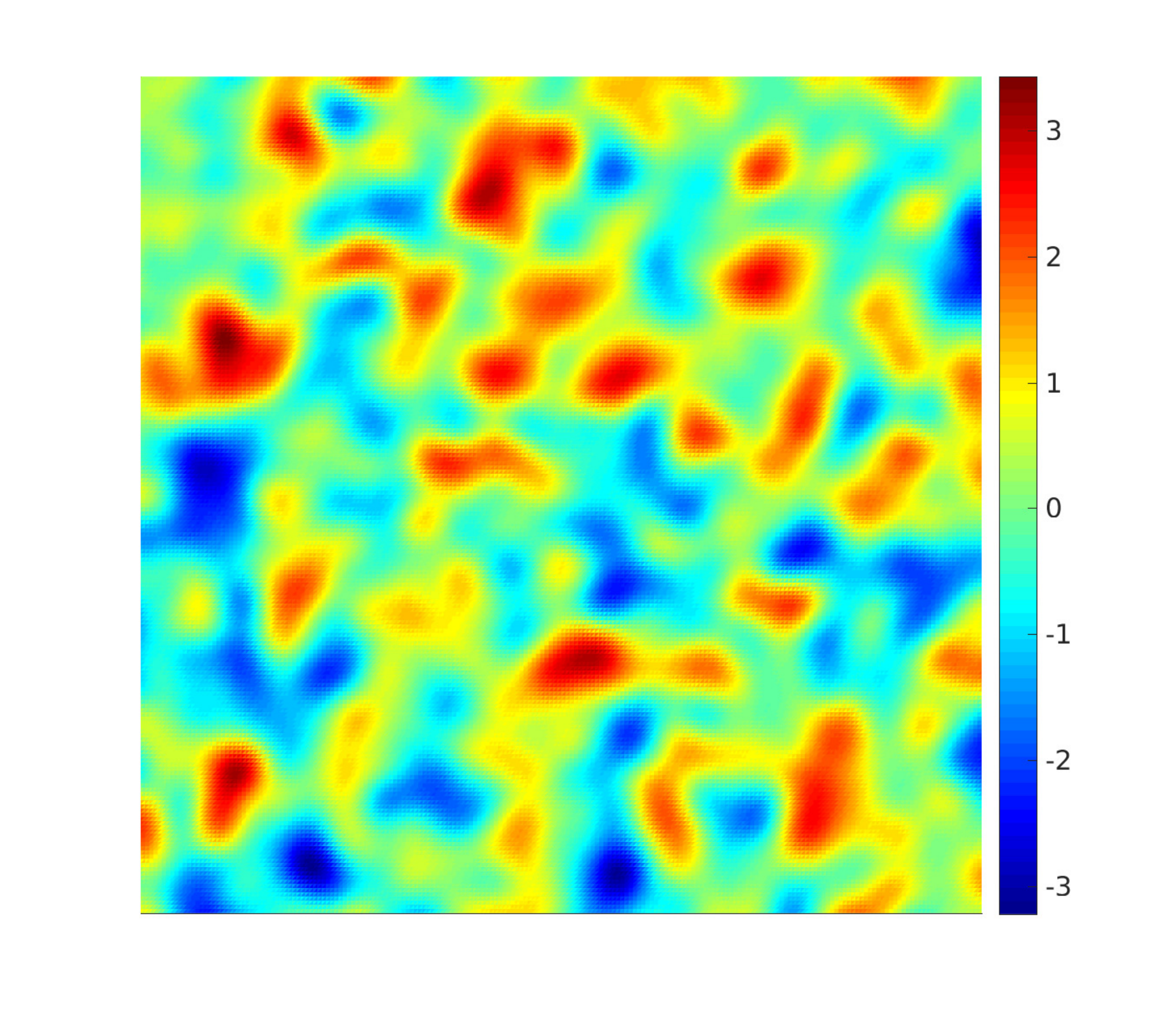}
		\subcaption{Realization of the field on the square $[-2, 2]^2$. The colour scale is logarithmic.}
		\label{fig:normal field}
	\end{subfigure}
	\begin{subfigure}[b]{.45\textwidth}
%
%
\definecolor{mycolor1}{rgb}{0.85098,0.32549,0.09804}%
\definecolor{mycolor2}{rgb}{0.00000,0.44706,0.74118}%
\definecolor{mycolor3}{rgb}{0.46667,0.67451,0.18824}%
\definecolor{mycolor4}{rgb}{0.38039,0.38039,0.38039}%
\begin{tikzpicture}

\begin{axis}[%
width=.8\linewidth,
height=.7\linewidth,
at={(-3.063in,1.18in)},
scale only axis,
xmode=log,
xmin=.9,
xmax=23,
xtick={1, 10, 20},
xticklabels={1, 10, 20},
minor xtick={1,2,3,4,5,6,7,8,9,10,20},
xlabel style={font=\color{white!15!black}},
xlabel={$R$},
ymode=log,
ymin=1e-05,
ymax=1,
yminorticks=false,
axis background/.style={fill=white},
xmajorgrids,
xminorgrids,
ymajorgrids,
legend style={legend cell align=left, font=\scriptsize, align=left, draw=white!15!black,at={(.02,.02)},anchor=south west }
]


\addplot [color=mycolor1, mark size=1.0pt, mark=*, mark options={solid, fill=mycolor1, mycolor1}]
table[row sep=crcr]{%
	1	0.446209183590206\\
	1.5	0.327157960303222\\
	2	0.0285015803466817\\
	2.5	0.0636280387060642\\
	3	0.171876785470852\\
	3.5	0.0759601734002965\\
	4	0.0195227902631044\\
	4.5	0.0289474782092297\\
	5	0.0271394211641465\\
	5.5	0.0250624999135389\\
	6	0.0142304391278437\\
	6.5	0.0113321567677394\\
	7	0.00671195726765494\\
	7.5	0.00296862515377239\\
	8	0.00446635239041404\\
	8.5	0.00335586082141209\\
	9	0.0033954849492272\\
	9.5	0.00141946050254499\\
	10	0.000504042968808571\\
	10.5	0.00129975036537144\\
	11	0.000937372754467924\\
	11.5	0.00129536454463042\\
	12	0.00154155142169671\\
	12.5	0.00243871640048221\\
	13	0.00212263814783664\\
	13.5	0.00137595744146674\\
	14	0.000967396455817116\\
	14.5	0.000852472223485223\\
	15	0.000610337278191042\\
	15.5	0.00045328515698434\\
	16	0.000214562544916308\\
	16.5	0.000424827982427101\\
	17	0.000931641819466589\\
	17.5	0.00080445412389768\\
	18	0.000280484061137676\\
	18.5	0.000183906622730878\\
	19	0.000245915260964846\\
	19.5	0.000319377297312295\\
};
\addlegendentry{$q=0$}

\addplot [color=mycolor2, mark=x, mark options={solid, mycolor2}]
  table[row sep=crcr]{%
1	0.493037722676099\\
1.5	0.319652303986697\\
2	0.0176244723968141\\
2.5	0.0375557000096753\\
3	0.14512439411506\\
3.5	0.0756078082733692\\
4	0.00363572212274466\\
4.5	0.0156049113071433\\
5	0.010041202420625\\
5.5	0.0113014227094179\\
6	0.00704860101880345\\
6.5	0.00595032967061182\\
7	0.00385529073537161\\
7.5	0.0013561599849149\\
8	0.00172937715829957\\
8.5	0.00121972959060523\\
9	0.00136214685120062\\
9.5	0.000405205764698602\\
10	7.07834242678063e-05\\
10.5	0.000537392769405121\\
11	0.000380726044853175\\
11.5	0.00045446889009626\\
12	0.000251717177081168\\
12.5	0.000497311325195704\\
13	0.000391364260979021\\
13.5	0.000165768994491999\\
14	0.000154775422255893\\
14.5	0.000153366811337874\\
15	5.82543680969771e-05\\
15.5	4.21614570056961e-05\\
16	0.000112076584826593\\
16.5	6.56490815914013e-05\\
17	0.000170927943480004\\
17.5	0.000194884460739952\\
18	9.41710840445809e-05\\
18.5	1.22988486860877e-05\\
19	6.17912395137754e-05\\
19.5	0.000130367698478628\\
};
\addlegendentry{$q=2$}

\addplot [color=mycolor1, dashed]
table[row sep=crcr]{%
	1	8\\
	20	0.001\\
};
\addlegendentry{$R^{-3}$}

\addplot [color=mycolor2, dashed]
table[row sep=crcr]{%
	1	3.2\\
	20	0.00002\\
};
\addlegendentry{$R^{-4}$}

\end{axis}
\end{tikzpicture}%
		\subcaption{Resonance error. $L=2R/3$.}
		\label{fig:convergence plot stochastic}
	\end{subfigure}	
	\caption{Log-normal random field \cref{eq: tensor log-normal} with $\mu=0$, $\sigma^2=1$ and $\ell=0.2$, and resonance error for the parabolic cell problem with filter order $q$ and final time $T=\frac{\abs{R-L}}{10}$. The size of the sampling domain, $R$, is reported on the $x$-axis; the resonance error \eqref{eq: reserr} is reported on the $y$-axis.}
\end{figure}
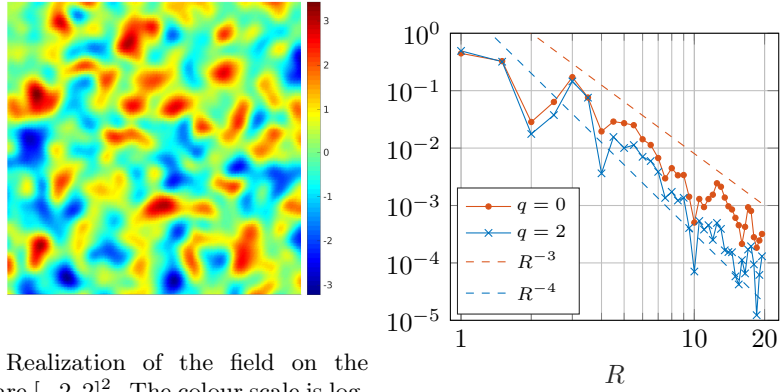		


\section{Computational efficiency} \label{sec:computational cost}
The goal of this section is to study the computational efficiency of the parabolic method. We will also provide theoretical and computational efficiency comparisons with the classical and the regularised elliptic methods.
Our analysis and numerical computations show that, for sufficiently high order filters, the computational cost is lower for the parabolic model than for the elliptic ones. 
\subsection{Regularised elliptic model}
Let us consider the regularised elliptic homogenization scheme described in Ref. \refcite{Glo11}, which reads as
\begin{equation}\label{eq:Gloria_eq}
\dfrac{1}{T} \psi^{i}_{R,T} - \nabla \cdot \left(a(\bx) \nabla \psi^{i}_{R,T}(\bx) \right)  = \nabla \cdot a(\bx) \mathbf{e}_i, \quad \text{ in } K_R
\end{equation}
equipped with homogeneous Dirichlet BCs on $\partial K_{R}$. We partition the domain $K_R$ with uniform simplicial elements of size $h$ and we introduce a finite elements space $S_h\subset H^1_0(K_R)$ made of piecewise polynomial functions of degree $s$ on the simplices. 
The finite elements discretization of the corrector problem reads: Find $\psi^i_{R,T,h}\in S_h$ such that 
	\begin{equation}\label{eq:ell cell FE pb}
	\int_{ K_{R}} \frac{1}{T} \psi^i_{R,T,h} w_h + a(\bx)\left(\nabla\psi^i_{R,T,h} +\mathbf{e}_i\right)\cdot \nabla w_h \,d\bx= 0 ,\quad\forall w_h\in S_h,\quad i=1,\dots,d,
	\end{equation} 
	and the upscaled tensor is defined as (we only indicate the $R$ parameter, as the optimal values $L$ and $T$ depend on $R$)
	\begin{equation}\label{eq: dis upscale formula}
	a^{0,R,h}_{ij} =  \int_{K_L} \left( \nabla  \psi^i_{R,T,h} + \mathbf{e}_i \right) \cdot a(\bx) \left(\nabla  \psi^j_{R,T,h} + \mathbf{e}_j \right) \mu_L(\bx) \,d\bx.
	\end{equation} 
	Hence, the total error for the upscaled coefficients is:
	\begin{equation*}
		|a^{0,R,h}_{ij} -a^{0}_{ij}| \lesssim h^{s+1} + R^{-4} (\log R)^8,
	\end{equation*}
	where the second term in the error estimate is the modelling error.
	The finite elements corrector $\psi^i_{R,h}$ is computed by solving the linear system
	\begin{equation}\label{eq: dis ell cell pb}
	A_h \mathbf{v}_i = \mathbf{b}_i,\text{ for }i=1,\dots,d,
	\end{equation} 
	where $A_h$ is a $N\times N$ symmetric positive definite matrix and $\mathbf{v}_i$ and $\mathbf{b}_i$ are the coordinates of, respectively, $\psi_h$ and $-\nabla\cdot\left(a(y)\mathbf{e}_i\right)$ in the finite element space and given a lagrangian basis. 
	Here, ${N=\mathcal{O}(R^dh^{-d})}$ is the dimension of the space $S_h$.
	The linear system can be solved in several ways using direct or iterative methods, whose cost depends on $N$. For example, for LU factorization the number of operations is $\mathcal{O}(N^{3/2})$ Ref. \refcite{GeN88}, for Conjugate Gradient (CG) it is $\mathcal{O}(\sqrt{\kappa}N)$\footnote{$\kappa$ is the condition number} Ref. \refcite{Saa03}, while for multigrid (MG) it is $\mathcal{O}(N)$. 
	\add{In the following analysis we will consider the LU factorization as linear solver, as it is the one that we used in practice to solve \eqref{eq:ell cell FE pb}.} We require the total errors to scale as a given tolerance $tol$, so $R = \mathcal{O}(tol^{-1/4})$ (we omitted the logarithmic term) and $h= \mathcal{O}(tol^{1/(s+1)})$. Hence, the total cost is 
\[
\add{C =  \mathcal{O}(N^{3/2}) = \mathcal{O}(R^{3d/2} h^{-3d/2}) = \mathcal{O}(tol^{-\frac{3d}{8}-\frac{3d}{2(s+1)}}).}
\]
\label{sec:Reg_Elliptic}
\subsection{Parabolic case with explicit stabilised time integration methods}
\label{sec:parabolic cost}
Let us consider the parabolic cell problem \cref{eq:parabolic dirichlet problem} with the upscaling formula \cref{eq: a0 parabolic dirichlet}.
As in the elliptic case, one can discretize \cref{eq:parabolic dirichlet problem} in space and compute an approximation $u_h^i(t)$ of $u^i(\cdot,t)$ in the $N$-dimensional finite elements space $S_h$. 
\corr{For simplicity of notation, we will omit the superscript $i$.}
For a given basis of $S_h$, the function $u_h(t)$ is uniquely determined by the vectorial function $\mathbf{w}_h:[0,T]\mapsto \R^N$, that solve the semi-discrete problem:
\begin{equation}\label{eq: semidisc parab cell pb}
\frac{d}{dt}\mathbf{w}_h   = - M_h^{-1}A_h  \mathbf{w}_h.
\end{equation}
We assume that the mass matrix $M_h$ is easy to invert (which hold, e.g., in the case of mass lumping or discontinuous Galerkin FEs), so that the cost of the right-hand side evaluation is negligible with respect to the solution of the ODE system. The differential equation \eqref{eq: semidisc parab cell pb} is solved by an explicit stabilised time integration scheme of order $r$. Examples of second order methods are RKC2 (Ref. \refcite{HoS80}) and ROCK2 (Ref. \refcite{AbM01}), while ROCK4 (Ref. \refcite{Abd02}) is a fourth order method. The fully discrete problem reads
\begin{equation*}
\mathbf{W}_{k} = \Phi_h( \mathbf{W}_{k-1} ), \text{ for }k=1,\dots, N_t,
\end{equation*}
where the function $\Phi_h$ identifies the time integration method \corr{and $N_t$ the number of time steps}.
The computed sequence $\lbrace \mathbf{W}_{k}\rbrace_{k=0}^{N_t} \subset \R^N$ is an approximation, at times $t_k = k\Delta t$, of $\mathbf{w}(t_k)$ and it determines (via the finite elements basis) a sequence $\lbrace U_{k}\rbrace_{k=0}^{N_t}\subset S_h$. The discrete approximation of the homogenized tensor is 
\begin{equation*}
a^{0,R,h,\Delta t}_{ij} = \int_{K_L} a_{ij}(y) \mu_L(\bx) \,d\bx - 2  \mathcal{Q} \left( \int_{K_L} U_{k} U^j_{k} \mu_L(y) \,d\bx, \Delta t\right) ,
\end{equation*} 
where $\mathcal{Q}(\cdot,\Delta t)$ is a quadrature rule on the discretization  $t_k = k\Delta t$ of order at least $r$ \corr{(where $r$ is the order of the time integration scheme)}. Hence, the total error for the upscaled coefficients is:
\begin{equation}\label{eq: tol rate}
|a^{0,R,h, \Delta t}_{ij} -a^{0}_{ij}| \le C \left( h^{s+1} +  \Delta t^r + R^{-(q+1)} \right),
\end{equation}
where we have assumed that, for sufficiently large $R$, the term $R^{-(q+1)}$ dominates the exponential term in the resonance error bound. This is also the convergence rate that we reported in the numerical examples of \Cref{subsec:numerical 2d smooth periodic,subsec:numerical 2d nonsmooth periodic}.
\corr{Here, the constant $C$ grows linearly with the final time $T$, whose optimal value scales as $R-L$. However, the ratio $(R-L)/\sqrt{8\beta\lambda_0}$
is in general ${\cal O}(1)$, so we can consider $T=\mathcal{O}(1)$ in the range of values used for $R$ and $L$.} 
In order for the error to scale as $tol$, we require that all the three summands in \eqref{eq: tol rate} scale as $tol$:
\[
R = \mathcal{O}(tol^{-\frac{1}{q+1}}),\quad h= \mathcal{O}(tol^{\frac{1}{s+1}}),\quad \Delta t= \mathcal{O}(tol^{\frac{1}{r}}).
\]

The global computational cost is $\mathcal{O}(Nn_SN_t)$, where  $N_t= T/\Delta t$ is the number of time steps,  $n_S$ is the number of function evaluations (stages) per time step for a stabilised method and $N=\mathcal{O}(R^dh^{-d})$ is the cost of each function evaluation which, in the linear case, is the cost of multiplying a sparse $N\times N$ matrix by a vector in $\R^N$. 
Since we are using a stabilised method we need to satisfy the weak stability condition $\rho \Delta t = c n_S^2$, where $\rho$ is the spectral radius of the Jacobian of the ODE \eqref{eq: semidisc parab cell pb} and $n_S$ is the number of stages for each time step. As $\rho$ is the spectral radius of $M_h^{-1}A_h$, it scales as $h^{-2}$. Therefore, $n_S = \mathcal{O}(\Delta t^{1/2} h^{-1})$.
From the fact that $T = \mathcal{O}(1)$ one derives that the total cost is
\[
Cost = \mathcal{O} (R^d h^{-d} \Delta t^{1/2} h^{-1} \Delta t^{-1} ) = \mathcal{O} ( tol^{-\frac{d}{q+1} - \frac{d+1}{s+1} - \frac{1}{2r}}).
\]
\subsection{Comparison of the parabolic and the \add{regularised} elliptic models}
We conclude the paper by comparing the computational cost of the elliptic regularised and the parabolic models.
In \cref{table:comp cost tol2}, we summarize the dependency of computational cost and the error on resonance and discretization parameters, as well as the scaling of the cost for a given tolerance.
\begin{table}[h!]
	\centering
	\begin{tabular}{|c| c c |}
		\hline
		Cell problem & Parabolic & Regularised Elliptic, Ref. \refcite{Glo11} \\
		\hline
		Error & $R^{-q-1}+ h^{s+1}+ \Delta t^r$ & $R^{-4}(\log R)^8+ h^{s+1}$ \\[1ex]
		Computational cost & $R^{d}h^{-d-1}\Delta t^{-\frac{1}{2}}$  &  $R^{3d/2} h^{-3d/2}$ \\[1ex]
		Computational cost ($tol$) & $tol^{-\frac{d}{q+1} - \frac{d+1}{s+1} - \frac{1}{2r}}$ & $tol^{-\frac{3d}{8}-\frac{3d}{2(s+1)}}$\\[1ex]
		\hline 
	\end{tabular}
	\caption{Error and computational cost for two homogenization approaches.}
	\label{table:comp cost tol2} 
\end{table}

The parabolic approach will be asymptotically more efficient than the regularised elliptic one if:
\begin{equation}
	\label{eq: efficiency condition}
	\add{\frac{d}{q+1} +  \frac{d+1}{s+1} +  \frac{1}{2r} < \frac{3d}{8} + \frac{3d}{2(s+1)}}.
\end{equation}
So, for $\mathbb{P}1$ finite element with ROCK2 time integration ($r=2$ and $s=1$), the parabolic approach is more efficient than the regularised elliptic one if $q>3$ for problems in dimension $d=2$ or if $q\ge 2$ for dimension $d=3$. 

\newadd{%
\begin{remark}
	As the theoretical estimates show, the particular advantage of the parabolic approach (when compared to Ref. \refcite{Glo11,GlH16}) will be observed for three dimensional problems with large values for $q$. 
	In particular, we expect the parabolic corrector to be particularly efficient in 3D, first because of \eqref{eq: efficiency condition}, second because of cost issues in relation with implementation. Indeed, while no matrix inversion is involved for solving the parabolic corrector problem in any dimension via explicit stabilised solvers (e.g. ROCK methods) this is not the case for the elliptic corrector problem. In 3D, even though the linear solver with $\mathcal{O}(N)$ complexity such as multigrid can be implemented, the prefactor will depend on the roughness of the conductivity tensor and the type of preconditioner used. However, in what follows we will concentrate on a 2D comparison as an exhaustive comparison with other approaches is not the aim of the paper.
\end{remark}
}
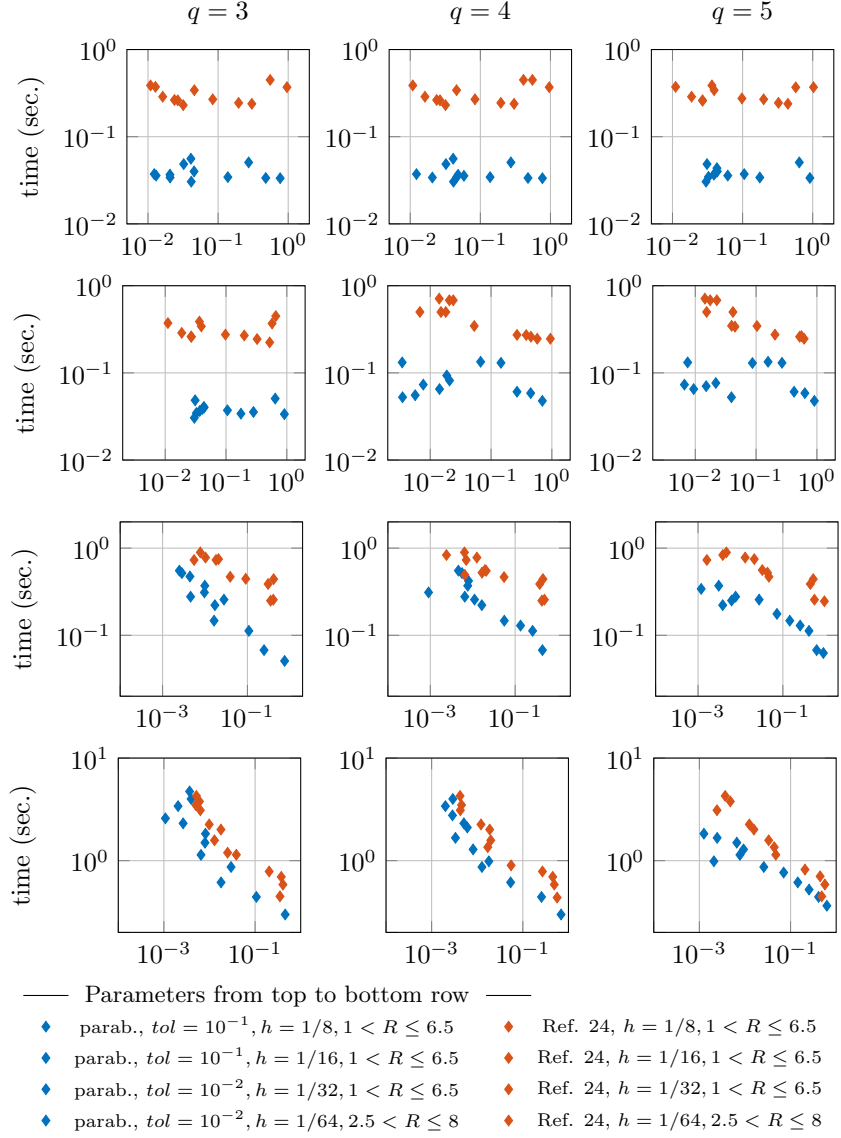
\begin{figure}
	\centering
%
%
\definecolor{mycolor1}{rgb}{0.00000,0.44706,0.74118}%
\definecolor{mycolor2}{rgb}{0.85098,0.32549,0.09804}%
\begin{tikzpicture}

\begin{axis}[%
width=.2\linewidth,
height=.12\textheight,
scale only axis,
xmode=log,
xmin=0.005,
xmax=2,
xminorticks=false,
ymode=log,
ymin=0.01,
ymax=1,
yminorticks=false,
ylabel={time (sec.)},
axis background/.style={fill=white},
title style={font=\bfseries},
title={$q=3$},
xmajorgrids,
ymajorgrids,
legend style={font=\color{white!15!black},draw=none, font=\scriptsize},
legend columns = {2},
legend entries = {{\small Parameters from top to bottom row             },{  \hphantom{invisible label}             },
				  {parab., $tol=10^{-1}, h=1/8,1<R\le6.5\quad$   },{  Ref. \refcite{Glo11}, $h=1/8,1<R\le6.5$}, %
			      {parab., $tol=10^{-1}, h=1/16,1<R\le6.5\quad$  },{  Ref. \refcite{Glo11}, $h=1/16,1<R\le6.5$}, %
	              {parab., $tol=10^{-2}, h=1/32,1<R\le6.5\quad$  },{  Ref. \refcite{Glo11}, $h=1/32,1<R\le6.5$}, %
                  {parab., $tol=10^{-2}, h=1/64,2.5<R\le8\quad$  },{  Ref. \refcite{Glo11}, $h=1/64,2.5<R\le8$}, %
},
legend to name = named
]


\addplot []
table[row sep=crcr]{%
	0 0\\
	100 100\\
};

\addplot []
table[row sep=crcr]{%
	0 0\\
	100 100\\
};


\addplot [only marks, mark=diamond*, mark options={solid, fill=mycolor1, mycolor1}]
  table[row sep=crcr]{%
0.767695079782247	0.033422\\
0.478273638088535	0.033621\\
0.272713910050792	0.050738\\
0.137697232001911	0.034395\\
0.0206611126388339	0.034042\\
0.0122765749256725	0.037244\\
0.0130514731501192	0.035822\\
0.0205571442939387	0.036731\\
0.0322991219532991	0.048495\\
0.0413978190828478	0.030488\\
0.0454123218297895	0.039985\\
0.0409382466246105	0.055897\\
};

\addplot [only marks, mark=diamond*, mark options={solid, fill=mycolor2, mycolor2}]
  table[row sep=crcr]{%
0.960420062725136	0.369777\\
0.552944999609993	0.448169\\
0.302585647320809	0.238407\\
0.196088326943248	0.244346\\
0.0838105809049866	0.268689\\
0.0457962137957358	0.341892\\
0.0318236446705344	0.230291\\
0.0108702783785496	0.387571\\
0.012782534300552	0.372979\\
0.0238351221580071	0.263679\\
0.0265656891032923	0.261489\\
0.0162122825960638	0.287416\\
};


\addplot [only marks, mark=diamond*, mark options={solid, fill=mycolor1, mycolor1}]
table[row sep=crcr]{%
	0 0\\
	100 100\\
};

\addplot [only marks, mark=diamond*, mark options={solid, fill=mycolor2, mycolor2}]
table[row sep=crcr]{%
	0 0 \\
	100 100\\
};

\addplot [only marks, mark=diamond*, mark options={solid, fill=mycolor1, mycolor1}]
table[row sep=crcr]{%
	0 0\\
	100 100\\
};

\addplot [only marks, mark=diamond*, mark options={solid, fill=mycolor2, mycolor2}]
table[row sep=crcr]{%
	0 0 \\
	100 100\\
};

\addplot [only marks, mark=diamond*, mark options={solid, fill=mycolor1, mycolor1}]
table[row sep=crcr]{%
	0 0\\
	100 100\\
};

\addplot [only marks, mark=diamond*, mark options={solid, fill=mycolor2, mycolor2}]
table[row sep=crcr]{%
	0 0 \\
	100 100\\
};

\end{axis}
\end{tikzpicture}%
%
%
%
\definecolor{mycolor1}{rgb}{0.00000,0.44706,0.74118}%
\definecolor{mycolor2}{rgb}{0.85098,0.32549,0.09804}%
\begin{tikzpicture}

\begin{axis}[%
width=.2\linewidth,
height=.12\textheight,
scale only axis,
xmode=log,
xmin=0.005,
xmax=2,
xminorticks=false,
ymode=log,
ymin=0.01,
ymax=1,
yminorticks=false,
axis background/.style={fill=white},
title style={font=\bfseries},
title={$q=4$},
xmajorgrids,
ymajorgrids
]
\addplot [only marks, mark=diamond*, mark options={solid, fill=mycolor1, mycolor1}]
  table[row sep=crcr]{%
  	0.767695079782247	0.033422\\
  	0.478273638088535	0.033621\\
  	0.272713910050792	0.050738\\
  	0.137697232001911	0.034395\\
  	0.0582836186887982	0.035581\\
  	0.0206611126388339	0.034042\\
  	0.0122765749256725	0.037244\\
  	0.0322991219532991	0.048495\\
  	0.0413978190828478	0.030488\\
  	0.0462754715411727	0.034631\\
  	0.0481195137657285	0.036546\\
  	0.0409382466246105	0.055897\\
};

\addplot [only marks, mark=diamond*, mark options={solid, fill=mycolor2, mycolor2}]
  table[row sep=crcr]{%
  	0.960420062725136	0.369777\\
  	0.552944999609993	0.448169\\
  	0.412353519634973	0.448259\\
  	0.302585647320809	0.238407\\
  	0.196088326943248	0.244346\\
  	0.0838105809049866	0.268689\\
  	0.0457962137957358	0.341892\\
  	0.0318236446705344	0.230291\\
  	0.0108702783785496	0.387571\\
  	0.0238351221580071	0.263679\\
  	0.0265656891032923	0.261489\\
  	0.0162122825960638	0.287416\\
};

\end{axis}
\end{tikzpicture}%
%
%
%
\definecolor{mycolor1}{rgb}{0.00000,0.44706,0.74118}%
\definecolor{mycolor2}{rgb}{0.85098,0.32549,0.09804}%
\begin{tikzpicture}
\begin{axis}[%
width=.2\linewidth,
height=.12\textheight,
scale only axis,
xmode=log,
xmin=0.005,
xmax=2,
xminorticks=false,
ymode=log,
ymin=0.01,
ymax=1,
yminorticks=false,
axis background/.style={fill=white},
title style={font=\bfseries},
title={$q=5$},
xmajorgrids,
ymajorgrids
]
\addplot [only marks, mark=diamond*, mark options={solid, fill=mycolor1, mycolor1}]
  table[row sep=crcr]{%
0.911970857398475	0.033621\\
0.644000712567679	0.050738\\
0.176035221247206	0.034042\\
0.105542917054883	0.037244\\
0.0614750594158689	0.035822\\
0.0390319326704425	0.036731\\
0.0310682316986121	0.048495\\
0.0301235348159865	0.030488\\
0.032633879165347	0.034631\\
0.0409298579814018	0.038686\\
0.0435333848203213	0.039985\\
0.0430302229545827	0.043591\\
};

\addplot [only marks, mark=diamond*, mark options={solid, fill=mycolor2, mycolor2}]
  table[row sep=crcr]{%
1.02653406136752	0.369777\\
0.572758733088948	0.368901\\
0.442677911439342	0.238407\\
0.324562599002467	0.244346\\
0.19965953626148	0.268689\\
0.0978045988401324	0.274795\\
0.039242878566206	0.341892\\
0.0365333199791811	0.387571\\
0.0268516008400151	0.25886\\
0.0110805838409538	0.372879\\
0.0268077080607975	0.261489\\
0.0187089300803927	0.287416\\
};

\end{axis}
\end{tikzpicture}%
\\
%
%
\definecolor{mycolor1}{rgb}{0.00000,0.44706,0.74118}%
\definecolor{mycolor2}{rgb}{0.85098,0.32549,0.09804}%
\begin{tikzpicture}

\begin{axis}[%
width=.2\linewidth,
height=.12\textheight,
scale only axis,
xmode=log,
xmin=0.002,
xmax=2,
xminorticks=false,
ymode=log,
ymin=0.01,
ymax=1,
yminorticks=false,
ylabel={time (sec.)},
axis background/.style={fill=white},
xmajorgrids,
ymajorgrids
]
\addplot [only marks, mark=diamond*, mark options={solid, fill=mycolor1, mycolor1}]
  table[row sep=crcr]{%
0.911970857398475	0.033621\\
0.644000712567679	0.050738\\
0.282745974937313	0.035581\\
0.176035221247206	0.034042\\
0.105542917054883	0.037244\\
0.0310682316986121	0.048495\\
0.0301235348159865	0.030488\\
0.032633879165347	0.034631\\
0.0369643456692366	0.036546\\
0.0409298579814018	0.038686\\
0.0435333848203213	0.039985\\
0.0433526495367936	0.040653\\
};

\addplot [only marks, mark=diamond*, mark options={solid, fill=mycolor2, mycolor2}]
  table[row sep=crcr]{%
0.658726125647642	0.448169\\
0.572758733088948	0.368901\\
0.522601034674627	0.223633\\
0.324562599002467	0.244346\\
0.19965953626148	0.268689\\
0.0978045988401324	0.274795\\
0.039242878566206	0.341892\\
0.0365333199791811	0.387571\\
0.0268516008400151	0.25886\\
0.0110805838409538	0.372879\\
0.0268077080607975	0.261489\\
0.0187089300803927	0.287416\\
};

\end{axis}
\end{tikzpicture}%
%
%
%
\definecolor{mycolor1}{rgb}{0.00000,0.44706,0.74118}%
\definecolor{mycolor2}{rgb}{0.85098,0.32549,0.09804}%
\begin{tikzpicture}

\begin{axis}[%
width=.2\linewidth,
height=.12\textheight,
scale only axis,
xmode=log,
xmin=0.002,
xmax=2,
xminorticks=false,
ymode=log,
ymin=0.01,
ymax=1,
yminorticks=false,
axis background/.style={fill=white},
xmajorgrids,
ymajorgrids
]
\addplot [only marks, mark=diamond*, mark options={solid, fill=mycolor1, mycolor1}]
  table[row sep=crcr]{%
  	0.697836708603367	0.047739\\
  	0.448516660454741	0.058571\\
  	0.267312300607474	0.06073\\
  	0.145483423736705	0.130359\\
  	0.0674017536025963	0.134018\\
  	0.00350207675449897	0.052561\\
  	0.00565747070715328	0.05545\\
  	0.00345534516203086	0.131849\\
  	0.00768028927707275	0.073333\\
  	0.0142094508563669	0.065216\\
  	0.0207759274903286	0.081779\\
  	0.0186870077526677	0.093333\\
};

\addplot [only marks, mark=diamond*, mark options={solid, fill=mycolor2, mycolor2}]
  table[row sep=crcr]{%
  	0.945572909600258	0.246177\\
  	0.575279497192267	0.246253\\
  	0.452307524582404	0.259997\\
  	0.378808606500008	0.271732\\
  	0.267105516428891	0.272706\\
  	0.0528894568735571	0.344979\\
  	0.0149975087744774	0.499092\\
  	0.0180088369073562	0.499289\\
  	0.0067517402696616	0.499519\\
  	0.0206206863932177	0.680994\\
  	0.0238846402362308	0.681165\\
  	0.0140338004155043	0.710663\\
};

\end{axis}
\end{tikzpicture}%
%
%
%
\definecolor{mycolor1}{rgb}{0.00000,0.44706,0.74118}%
\definecolor{mycolor2}{rgb}{0.85098,0.32549,0.09804}%
\begin{tikzpicture}

\begin{axis}[%
width=.2\linewidth,
height=.12\textheight,
scale only axis,
xmode=log,
xmin=0.002,
xmax=2,
xminorticks=false,
ymode=log,
ymin=0.01,
ymax=1,
yminorticks=false,
axis background/.style={fill=white},
xmajorgrids,
ymajorgrids
]
\addplot [only marks, mark=diamond*, mark options={solid, fill=mycolor1, mycolor1}]
  table[row sep=crcr]{%
  	0.908351905242452	0.047739\\
  	0.634305712685099	0.058571\\
  	0.421540525529039	0.06073\\
  	0.266755122016592	0.130359\\
  	0.157299573099088	0.134018\\
  	0.0870562890737136	0.129722\\
  	0.0393065929830137	0.052561\\
  	0.0074898513654973	0.131849\\
  	0.00662491955573239	0.073333\\
  	0.00940854593174891	0.065216\\
  	0.0151014774980961	0.070424\\
  	0.0217547852737691	0.076493\\
};

\addplot [only marks, mark=diamond*, mark options={solid, fill=mycolor2, mycolor2}]
  table[row sep=crcr]{%
  	0.616480088161145	0.246253\\
  	0.556937943953865	0.264084\\
  	0.520397486217817	0.259997\\
  	0.20467652012638	0.274474\\
  	0.103399449333795	0.344979\\
  	0.0449614719155726	0.339312\\
  	0.0396155275795763	0.346574\\
  	0.0416461031060193	0.499092\\
  	0.0153756315542518	0.499519\\
  	0.0175664860827607	0.680994\\
  	0.0226096734161711	0.681165\\
  	0.0144061214387406	0.710663\\
};

\end{axis}
\end{tikzpicture}%
\\
%
%
%
\definecolor{mycolor1}{rgb}{0.00000,0.44706,0.74118}%
\definecolor{mycolor2}{rgb}{0.85098,0.32549,0.09804}%
\begin{tikzpicture}

\begin{axis}[%
width=.2\linewidth,
height=.12\textheight,
scale only axis,
xmode=log,
xmin=0.0001,
xmax=2,
xminorticks=false,
xtick={.0001,.001,.1},
xticklabels = {\hphantom{$10^{-4}$},$10^{-3}$,$10^{-1}$},
ymode=log,
ymin=0.02,
ymax=2,
yminorticks=false,
ytick={.02,.1,1,2},
yticklabels = {\hphantom{$10^{-2}$},$10^{-1}$,$10^0$,\phantom{$10^0$}},
ylabel={time (sec.)},
axis background/.style={fill=white},
xmajorgrids,
ymajorgrids
]
\addplot [only marks, mark=diamond*, mark options={solid, fill=mycolor1, mycolor1}]
  table[row sep=crcr]{%
  	0.754182577189824	0.050798\\
  	0.246700849983505	0.067597\\
  	0.108152750524525	0.112373\\
  	0.0163678806348604	0.147042\\
  	0.0279899677389759	0.256595\\
  	0.0171182985658236	0.221616\\
  	0.00457577589465769	0.277125\\
  	0.00968718769010797	0.31046\\
  	0.00989414719212365	0.370192\\
  	0.00441735803296992	0.47255\\
  	0.00285439998212579	0.516954\\
  	0.00248743895367685	0.551888\\
};

\addplot [only marks, mark=diamond*, mark options={solid, fill=mycolor2, mycolor2}]
  table[row sep=crcr]{%
  	0.351724786883321	0.249571\\
  	0.407574159078849	0.255879\\
  	0.414977037398549	0.440926\\
  	0.30768194933141	0.387574\\
  	0.0913288129781017	0.44457\\
  	0.0392533584500015	0.469024\\
  	0.0055289475689747	0.730457\\
  	0.0182033320428317	0.731615\\
  	0.020844444192223	0.748228\\
  	0.0102809162895998	0.779707\\
  	4.9999589082329e-06	0.833498\\
  	0.00776032095663722	0.893864\\
};

\end{axis}
\end{tikzpicture}%
%
	\hspace*{.04cm}
%
%
\definecolor{mycolor1}{rgb}{0.00000,0.44706,0.74118}%
\definecolor{mycolor2}{rgb}{0.85098,0.32549,0.09804}%
\begin{tikzpicture}

\begin{axis}[%
width=.2\linewidth,
height=.12\textheight,
scale only axis,
xmode=log,
xmin=0.0001,
xmax=2,
xminorticks=false,
xtick={.0001,.001,.1},
xticklabels = {\hphantom{$10^{-4}$},$10^{-3}$,$10^{-1}$},
ymode=log,
ymin=0.02,
ymax=2,
yminorticks=false,
ytick={.02,.1,1,2},
yticklabels = {\hphantom{$10^{-2}$},$10^{-1}$,$10^0$,\phantom{$10^0$}},
axis background/.style={fill=white},
xmajorgrids,
ymajorgrids
]
\addplot [only marks, mark=diamond*, mark options={solid, fill=mycolor1, mycolor1}]
  table[row sep=crcr]{%
  	0.439983802184408	0.067597\\
  	0.258181724203408	0.112373\\
  	0.134172867074546	0.129278\\
  	0.055466019330445	0.147042\\
  	0.0110359381789079	0.256595\\
  	0.0163690222433888	0.221616\\
  	0.00649873283544928	0.277125\\
  	0.000908145449928009	0.31046\\
  	0.00762879688543787	0.370192\\
  	0.00794610306272797	0.420944\\
  	0.00560359419691765	0.516954\\
  	0.00458942900967006	0.551888\\
};

\addplot [only marks, mark=diamond*, mark options={solid, fill=mycolor2, mycolor2}]
  table[row sep=crcr]{%
  	0.426501177269661	0.249571\\
  	0.489405702027512	0.255879\\
  	0.452759057467647	0.440983\\
  	0.379973281958345	0.387574\\
  	0.0547865894026036	0.465526\\
  	0.00650439820124242	0.493167\\
  	0.0165777145151334	0.522708\\
  	0.0207855178329725	0.548276\\
  	0.0194206006950133	0.560154\\
  	0.00701441245427224	0.731024\\
  	0.012527343689475	0.779707\\
  	0.00240241024881112	0.833498\\
  	0.00638532449865146	0.893864\\
};

\end{axis}
\end{tikzpicture}%
%
	\hspace*{.04cm}
%
%
\definecolor{mycolor1}{rgb}{0.00000,0.44706,0.74118}%
\definecolor{mycolor2}{rgb}{0.85098,0.32549,0.09804}%
\begin{tikzpicture}

\begin{axis}[%
width=.2\linewidth,
height=.12\textheight,
scale only axis,
xmode=log,
xmin=0.0001,
xmax=2,
xminorticks=false,
xtick={.0001,.001,.1},
xticklabels = {\hphantom{$10^{-4}$},$10^{-3}$,$10^{-1}$},
ymode=log,
ymin=0.02,
ymax=2,
yminorticks=false,
ytick={.02,.1,1,2},
yticklabels = {\hphantom{$10^{-2}$},$10^{-1}$,$10^0$,\phantom{$10^0$}},
axis background/.style={fill=white},
xmajorgrids,
ymajorgrids
]
\addplot [only marks, mark=diamond*, mark options={solid, fill=mycolor1, mycolor1}]
  table[row sep=crcr]{%
0.901699550017491	0.062616\\
0.626337588398807	0.067597\\
0.412857376640237	0.112373\\
0.255711771467013	0.129278\\
0.145583162219825	0.147042\\
0.0724698550496457	0.175973\\
0.0275425878097178	0.256595\\
0.00381014600172498	0.221616\\
0.0061894879418177	0.250608\\
0.00769121327108434	0.277125\\
0.00118671108263161	0.340572\\
0.00307654533275387	0.370192\\
};

\addplot [only marks, mark=diamond*, mark options={solid, fill=mycolor2, mycolor2}]
  table[row sep=crcr]{%
  	0.95269257339353	0.245857\\
  	0.554999325872034	0.255879\\
  	0.520380027371982	0.440983\\
  	0.44450720401614	0.387574\\
  	0.0468857541459299	0.469024\\
  	0.0433662235990934	0.522708\\
  	0.0327292192713506	0.560154\\
  	0.00161255592889085	0.731024\\
  	0.0211553720085781	0.748228\\
  	0.0129075386946695	0.779707\\
  	0.00383122309872509	0.833498\\
  	0.00465833442650771	0.893864\\
};

\end{axis}
\end{tikzpicture}%
	\hspace*{.06cm}\\
%
%
%
\definecolor{mycolor1}{rgb}{0.00000,0.44706,0.74118}%
\definecolor{mycolor2}{rgb}{0.85098,0.32549,0.09804}%
\begin{tikzpicture}

\begin{axis}[%
width=.2\linewidth,
height=.12\textheight,
scale only axis,
xmode=log,
xmin=0.0001,
xmax=1,
xtick={.0001,.001,.1,2},
xticklabels = {\hphantom{$10^{-4}$},$10^{-3}$,$10^{-1}$,\phantom{$10^0$}},
xminorticks=false,
ymode=log,
ymin=0.2,
ymax=10,
yminorticks=false,
ytick={.2,1,10},
yticklabels = {\hphantom{$10^{-2}$},$10^{0}$,$10^1$},
ylabel={time (sec.)},
axis background/.style={fill=white},
xmajorgrids,
ymajorgrids
]
\addplot [only marks, mark=diamond*, mark options={solid, fill=mycolor1, mycolor1}]
  table[row sep=crcr]{%
0.456733486351556	0.298875\\
0.106660930477909	0.44323\\
0.0179148724448513	0.614355\\
0.0296946575572751	0.867559\\
0.00651076127068031	1.138981\\
0.0079639318824345	1.497278\\
0.00816030711191594	1.832083\\
0.00263898333025701	2.311726\\
0.00107857302728158	2.584608\\
0.00204634248554261	3.400817\\
0.00395264441063648	3.994812\\
0.00366597808544203	4.728608\\
};

\addplot [only marks, mark=diamond*, mark options={solid, fill=mycolor2, mycolor2}]
  table[row sep=crcr]{%
0.351006686820844	0.44886\\
0.407629864427	0.584904\\
0.374772291915832	0.695747\\
0.203377581700108	0.784987\\
0.038792905497896	1.14231\\
0.0249360172950625	1.195157\\
0.012865322636998	1.579314\\
0.0177982465813773	2.010975\\
0.00985190891654209	2.254969\\
0.00629439180611747	3.100872\\
0.00500294052596302	3.486745\\
0.00603178597660068	3.778279\\
0.00515004627944599	4.274486\\
};

\end{axis}
\end{tikzpicture}%
%
	\hspace*{.04cm}
%
%
\definecolor{mycolor1}{rgb}{0.00000,0.44706,0.74118}%
\definecolor{mycolor2}{rgb}{0.85098,0.32549,0.09804}%
\begin{tikzpicture}

\begin{axis}[%
width=.2\linewidth,
height=.12\textheight,
scale only axis,
xmode=log,
xmin=0.0001,
xmax=1,
xminorticks=false,
xtick={.0001,.001,.1,2},
xticklabels = {\hphantom{$10^{-4}$},$10^{-3}$,$10^{-1}$,\phantom{$10^0$}},
ymode=log,
ymin=0.2,
ymax=10,
yminorticks=false,
ytick={.2,1,10},
yticklabels = {\hphantom{$10^{-2}$},$10^{0}$,$10^1$},
axis background/.style={fill=white},
xmajorgrids,
ymajorgrids
]
\addplot [only marks, mark=diamond*, mark options={solid, fill=mycolor1, mycolor1}]
  table[row sep=crcr]{%
  	0.689603047207949	0.298875\\
  	0.256806865430349	0.44323\\
  	0.0539709026386373	0.614355\\
  	0.0126867511983305	0.867559\\
  	0.0180024773465771	0.986444\\
  	0.00811889947921503	1.28804\\
  	0.00333004287478223	1.668305\\
  	0.00606073643793527	2.105729\\
  	0.00509437918458808	2.311726\\
  	0.00286716814000316	2.764583\\
  	0.00200389732917858	3.400817\\
  	0.00291390056693802	3.994812\\
};

\addplot [only marks, mark=diamond*, mark options={solid, fill=mycolor2, mycolor2}]
  table[row sep=crcr]{%
  	0.564473840877217	0.436659\\
  	0.489245634302892	0.584904\\
  	0.452933800913293	0.695747\\
  	0.269212124055572	0.784987\\
  	0.0553123720678119	0.900324\\
  	0.0170181567868515	1.354284\\
  	0.0198173867489069	1.579314\\
  	0.0187842509098716	2.010975\\
  	0.0121045648321277	2.254969\\
  	0.00426090047986302	3.100872\\
  	0.0045258670783365	3.486745\\
  	0.00420274763172317	4.274486\\
};

\end{axis}
\end{tikzpicture}%
%
	\hspace*{.04cm}
%
%
\definecolor{mycolor1}{rgb}{0.00000,0.44706,0.74118}%
\definecolor{mycolor2}{rgb}{0.85098,0.32549,0.09804}%
\begin{tikzpicture}

\begin{axis}[%
width=.2\linewidth,
height=.12\textheight,
scale only axis,
xmode=log,
xmin=0.0001,
xmax=1,
xminorticks=false,
xtick={.0001,.001,.1,2},
xticklabels = {\hphantom{$10^{-4}$},$10^{-3}$,$10^{-1}$,\phantom{$10^0$}},
ymode=log,
ymin=0.2,
ymax=10,
yminorticks=false,
ytick={.2,1,10},
yticklabels = {\hphantom{$10^{-2}$},$10^{0}$,$10^1$},
axis background/.style={fill=white},
xmajorgrids,
ymajorgrids
]
\addplot [only marks, mark=diamond*, mark options={solid, fill=mycolor1, mycolor1}]
  table[row sep=crcr]{%
0.625409285607719	0.363603\\
0.411619513377544	0.44323\\
0.254373253184966	0.522643\\
0.144185304646778	0.614355\\
0.0706376981288215	0.76636\\
0.0258871446606396	0.867559\\
0.00208280056581248	0.986444\\
0.00775153424051596	1.138981\\
0.00931989055720888	1.28804\\
0.00662518130367827	1.497278\\
0.00245692145807021	1.668305\\
0.00125825868189529	1.832083\\
};

\addplot [only marks, mark=diamond*, mark options={solid, fill=mycolor2, mycolor2}]
  table[row sep=crcr]{%
0.482706274944598	0.44886\\
0.56528014933279	0.585494\\
0.444781872271301	0.704214\\
0.206965970387026	0.81956\\
0.0474198968841926	1.14231\\
0.0438438946494637	1.354284\\
0.0331514640155835	1.579314\\
0.0157072449598516	2.010975\\
0.0124869309357084	2.254969\\
0.00243612424202325	3.100872\\
0.00477882148509347	3.778279\\
0.00368861256022676	4.274486\\
};

\end{axis}
\end{tikzpicture}%
	\hspace*{.05cm}%
%

	\ref{named}
	\caption{Comparison of the elapsed times ($y$-axis) against the approximation error \eqref{eq: reserr} ($x$-axis) for two homogenization approaches.}
	\label{fig: cheap tests}
\end{figure}

A comparison of the computational efficiency for the method described in the present paper and the regularised elliptic method of Ref. \refcite{Glo11} is discussed below. 
We take $L=2R/3$ and $T = \sqrt{\frac{1}{4\pi^2\alpha\beta}}\abs{R-L}$ (see \eqref{eq:choice L T}) for the parabolic case, and $L=2R/3$ and $T = L^2(\log L)^{-4}$ (as in Ref. \refcite{Glo11}) for the regularised elliptic case.	 
In the plots of \Cref{fig: cheap tests}, the elapsed real times spent to solve the parabolic equation \eqref{eq:parabolic dirichlet problem}  and the regularised elliptic equation \eqref{eq:Gloria_eq} are plotted against the corresponding error.
We consider the coefficients \eqref{eq:truncPyramids} and discretized the cell problem in space with $\mathbb{P}1$ finite elements on a triangular mesh of size $h$. The time integration for the parabolic problem is performed by the ROCK2 method with adaptive time stepping. So, the time integration scheme is fed with an absolute tolerance, $atol$, and a relative tolerance, $rtol$, instead of a fixed time discretization $\Delta t$. In our computations, we took $atol=rtol=tol$. 
It is also worth mentioning that $h$ and $tol$ have to be refined concurrently to avoid error saturation due to too coarse discretizations in time or space.
The tests were run for different choices of discretization and modelling parameters on a machine with 2 processors running at 2.6 GHz, with 14 cores each.
The value of $q$ varies along the rows, while the discretization parameters ($h$ and $tol$) vary along the columns. 
In each plot, we depict the error and elapsed computation time for increasing values of $R$. In the first three rows $R\in[1, 6.5]$, while $R\in [2.5, 8]$ for the last one.
The motivation for choosing $R$ in such ranges is the following: for coarse discretizations and large sampling domains, the discretization error is dominating thus undermining the property of high convergence of the methods. The threshold $R\approx6.5$, in the three coarser tests, has been empirically determined as the maximum domain size for which the discretization error does not exceed the resonance one. 
For the same reason, we excluded the values $R<2$ from the finest computations, because for $R<2$ the homogenized limit can be approximated more efficiently with coarser discretizations.

The plots of \Cref{fig: cheap tests} show that, for coarse values of tolerance $tol$, the parabolic method can reach the same error as the  regularised one with less computational effort.
In the two coarsest experiments ($h=1/8$, $tol = 10^{-1}$ and $h=1/16$, $tol = 10^{-1}$) the computational cost only slightly grows with increasing $R$, meaning that the elapsed time is dominated by overhead rather than actual computation. 
In the third row ($h=1/32$, $tol = 10^{-2}$), the growth of the computational cost is shown more clearly and the computational time for the parabolic method is always shorter than for the elliptic method, in particular for higher order filters. 
Such a positive effect of using high order filters is visible also in the last row ($h=1/64$, $tol = 10^{-2}$ and $q=5$).

In conclusion, the use of high order filers is beneficial for the parabolic case, as it allows to rapidly achieve high accuracy approximations with relatively low computational cost. 

\section{Conclusion}
In this work, we propose a novel approach for numerical homogenization, based on the solution of parabolic cell problems. We rigorously prove, by estimates of the fundamental solution, an arbitrary convergence rate for the resonance error in the periodic setting, but numerical tests demonstrate the same rates also for the boundary error in a non-periodic case. 
If filters of high order are used, the computation of the parabolic solutions by means of stabilised explicit solvers is asymptotically more efficient than the inversion of the discretized elliptic operator, required by elliptic approaches. 

\appendix
\section{\add{Proof of \cref{lemma: preliminary estimate sup-theta NEW:)}}}
\label{sec: appendix alter}
\add{
	\newcommand{\etd}{\ensuremath{\eta_{\Delta}}}
\newcommand{\doubleint}[2][\ensuremath{\,d\bx\,dt}]{\ensuremath{\int_{0}^{T}\int_{\Delta_2} #2 #1 }}
\newcommand{\const}{{\color{red}const}}
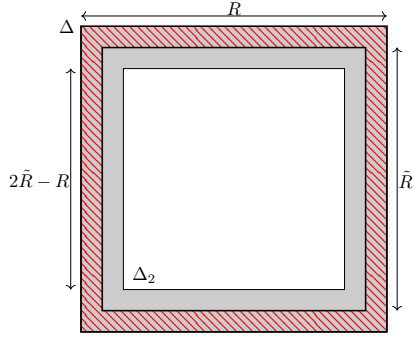
\begin{figure}[h!]
	\centering
	\scalebox{.7}{\begin{tikzpicture}
\draw[very thick](-2.9,-2.9) rectangle(2.9,2.9) ;
\filldraw[fill=black!20, even odd rule] (-2.9,-2.9) rectangle(2.9,2.9)  (-2.1,-2.1) rectangle(2.1,2.1) ;
\filldraw[fill=black!20, even odd rule, ,pattern=north west lines, pattern color=red] (-2.9,-2.9) rectangle(2.9,2.9)  (-2.5,-2.5) rectangle(2.5,2.5) ;

\draw[thick](-2.5,-2.5) rectangle(2.5,2.5) ;
%
%
\draw[<->] (-2.9,3.1) -- (2.9,3.1);
\draw (0,3) node[above] {$R$};
\draw[<->] (3.1,-2.5) -- (3.1,2.5);
\draw (3,0) node[right] {$\tilde{R}$};
\draw[<->] (-3.1,-2.1) -- (-3.1,2.1);
\draw (-3,0) node[left] {$2\tilde{R}-R$};
\draw (-2.9,2.9) node[left] {$\Delta$};
\draw (-2.05,-2.1) node[fill= white, above right]{$\Delta_2$};
\end{tikzpicture}}	
	\caption{The boundary layers $\Delta$ and $\Delta_2$.}
\end{figure}

\begin{lemma}\label{lemma: caccioppoli inequality green}
	Let $G(\bx,t;\by)$ be the fundamental solution \eqref{eq: whole definition greens function} in $K_R\times (0,+\infty)$ for coefficients $a(\cdot)\in \mathcal{M}\left(\alpha,\beta,K_R\right)$. 
	Recall that $\Delta= K_R\setminus R_{\tilde{R}}$, where $\tilde{R}$ is given in \Cref{def:boundary layer}.
	Let $\Delta_2:= K_R\setminus K_{2\tilde{R}-R}$ and $\by\in K_R $ such that $\mathop{dist}\left(\by,\Delta_2\right)\ge\epsilon$ for arbitrarily small $\epsilon>0$. Then, there exist a constant $M>0$ such that 
	\begin{equation*}
		\norm{\nabla G(\cdot,\cdot;\by) }_{L^2\left((0,T)\times \Delta\right)}
		\le
		\frac{2M\beta}{\alpha}
		\norm{ G(\cdot,\cdot;\by) }_{L^2\left((0,T)\times \Delta_2\right)}.
	\end{equation*}
\end{lemma}
\begin{proof}
	This is essentially a Caccioppoli estimate.
	First of all, we notice that $\Delta\subset \Delta_2$.
	Let us consider the smooth function $\etd\in C^{\infty}(\Delta_2)$ satisfying:
	\begin{equation}\label{eq: prop eta delta}
		0 \le \etd \le 1, \quad 
		\etd\equiv 1 \quad\text{in }\Delta, \quad 
		\etd=0 \quad\text{in }\partial \Delta_2 \setminus \partial \Delta, \quad  
		\text{and }\abs{\nabla\etd}\le M.
	\end{equation}
	Let us consider the weak form of the parabolic problem solved by the fundamental solution on $K_R\times(0,T)$ with $\phi:=\etd^2 G$ as a test function:
	\begin{equation}\label{eq: weak form}
	- \doubleint{ \partial_t G\phi} = \doubleint{\nabla\phi\cdot a(\bx)\nabla G},
	\end{equation}
	The identity above is true because $\phi= 0$ on $\partial\Delta_2$: indeed $G=0$ on the outer boundary $\partial K_R$, while $\etd=0$ on the inner boundary $\partial\Delta_2\setminus\partial \Delta$.
	Now, we bound from below the right-hand side of \eqref{eq: weak form}:
	\begin{equation}
		\label{eq: lower bound tmp}
		\begin{aligned}
			&\doubleint{\nabla\phi\cdot a(\bx)\nabla G}\\ 
			&= \doubleint{\etd^2 \nabla G \cdot a(\bx)\nabla G} + \doubleint{2 \etd G \nabla\etd\cdot a(\bx)\nabla G}\\
			&\ge \alpha \doubleint{\etd^2 \abs{\nabla G}^2 } + \doubleint{2 \etd G \nabla\etd\cdot a(\bx)\nabla G}.
		\end{aligned}
	\end{equation}	 
	The last term can be bounded by H\"older and Young inequalities:
	\begin{multline*}
		\doubleint{2 \etd G \nabla\etd\cdot a(\bx)\nabla G} \ge - 2\beta \doubleint{\abs{\etd G \nabla \etd\cdot \nabla G }}\\
		\ge - \frac{\alpha}{2} \doubleint{\etd^2 \abs{\nabla G}^2} - \frac{2\beta^2}{\alpha} \doubleint{\abs{\nabla \etd}^2 G^2}.
	\end{multline*}
	Hence, \eqref{eq: lower bound tmp} becomes 
	\begin{multline}
		\label{eq: lower bound}
		\doubleint{\nabla\phi\cdot a(\bx)\nabla G}\\ 
		\ge \frac{\alpha}{2} \doubleint{\etd^2 \abs{\nabla G}^2 } - \frac{2\beta^2M^2}{\alpha} \doubleint{G^2}.
	\end{multline}	 
	We now bound from above the left-hand side of \eqref{eq: weak form}:
	\begin{multline}\label{eq: upper bound}
		- \doubleint{ \partial_t G\phi} = -\doubleint{\frac{1}{2}\frac{d}{dt} \left(G^2 \etd^2\right)} \\
		= \frac{1}{2} \int_{\Delta_2}\underbrace{\abs{G(\bx,0;\by)}^2}_{\text{$=0$, since $\abs{x-y}>0$.}} \abs{\etd}^2\,d\bx - \frac{1}{2} \int_{\Delta_2}\abs{G(\bx,T;\by)}^2 \abs{\etd}^2 \,d\bx\le0
	\end{multline}
Therefore, from \eqref{eq: prop eta delta}, \eqref{eq: weak form}, \eqref{eq: lower bound} and \eqref{eq: upper bound}, we conclude that 
\begin{equation*}
	\int_{0}^{T} \int_{\Delta} \abs{\nabla G}^2 \,d\bx\,dt \le  \doubleint{\etd^2 \abs{\nabla G}^2 }  \\
	\le\left(\frac{2M\beta}{\alpha}\right)^2 \doubleint{G^2}.
\end{equation*}
\end{proof}

\begin{proof}[\Cref{lemma: preliminary estimate sup-theta NEW:)}]
	From \eqref{eq:theta through green function2} and from the symmetry of the fundamental solution, $G(\bx,t;\by)=G(\by,t;\bx)$, we estimate
	\begin{multline}\label{eq: initial estimate theta}
		\abs{\theta^i(\bx,t)} \le 
		\norm{G(\cdot,t;\bx)}_{L^2\left(\Delta\right)} \norm{\nabla\cdot\left(a(\cdot)\mathbf{e}_i\right)}_{L^2\left(\Delta\right)} \\
		+C_{\rho}\beta  \norm{G(\cdot,\cdot;\bx)}_{L^2\left((0,t)\times\Delta\right)} \norm{\nabla v^i}_{L^2\left((0,t)\times\Delta\right)} \\
		+C_{\rho}\beta  \norm{\nabla G(\cdot,\cdot;\bx)}_{L^2\left((0,t)\times\Delta\right)} \norm{v^i}_{L^2\left((0,t)\times\Delta\right)} ,
	\end{multline}
	where $C_{\rho}$ is an upper bound for the gradient of the cut-off function, $\nabla \rho$, of \Cref{def:cutoff}.
	
	Let us now consider a covering of $\Delta$, defined as $\Delta_K:=  \bigcup_{\by\in\partial K_{\frac{R + \tilde{R}}{2}}} K+\by$. 
	Then, $\abs{\Delta_K}= c(d) \abs{\partial K_{\frac{R + \tilde{R}}{2}}} \mathrm{diam}(K) \le C R^{d-1}$.
	By exploiting the periodic structure of $v^i$ we have that, for each $t>0$,
	\begin{gather*}
		\norm{v^i(\cdot,t)}_{L^2(\Delta)}
		\le 
		\norm{v^i(\cdot,t)}_{L^2(\Delta_K)}
		\le
		\left(\frac{\abs{\Delta_K}}{\abs{K}}\right)^{1/2}\norm{v^i(\cdot,t)}_{L^2(K)}, \\
		\norm{\nabla v^i(\cdot,t)}_{L^2(\Delta)}
		\le 
		\norm{\nabla v^i(\cdot,t)}_{L^2(\Delta_K)}
		\le 
		\left(\frac{\abs{\Delta_K}}{\abs{K}}\right)^{1/2}\norm{\nabla v^i(\cdot,t)}_{L^2(K)}.
	\end{gather*}
	Thus, it is sufficient to bound $\norm{v^i(\cdot,t)}_{L^2(K)}$ and $\norm{\nabla v^i(\cdot,t)}_{L^2(K)}$ in the periodic cell to derive corresponding estimates on $\Delta$. From standard energy estimates\footnote{\add{By testing the solution $v^i$ against itself, we know that $\frac{1}{2}\frac{d}{dt}\norm{v}^2_{L^2(K)} + \alpha\norm{\nabla v}^2_{L^2(K)}\le 0$. Thus, $\norm{\nabla v^i}^2_{L^2\left((0,+\infty)\times K\right)}\le \frac{1}{2\alpha}\norm{ v(\cdot,0)}^2_{L^2(K)}$ by integration in time. }} and the Poincar\'e-Wirtinger inequality in the space $W^1_{per}(K)$, it is easy to see that 
	\begin{subequations}\label{eq:estimates norms v}
		\begin{equation}
			\norm{\nabla\cdot\left(a(\cdot)\mathbf{e}_i\right)}_{L^2\left(\Delta\right)} 
			\le \left(\frac{\abs{\Delta_K}}{\abs{K}}\right)^{1/2} \norm{\nabla\cdot\left(a(\cdot)\mathbf{e}_i\right)}_{L^2\left(K\right)},
		\end{equation}
		\begin{multline}
			\norm{\nabla v^i}_{L^2\left((0,t)\times\Delta\right)} 
			\le \left(\frac{\abs{\Delta_K}}{\abs{K}}\right)^{1/2} \norm{\nabla v^i}_{L^2\left((0,t)\times K\right)} \\
			\le \left(\frac{\abs{\Delta_K}}{\abs{K}}\right)^{1/2} \frac{1}{\sqrt{2\alpha}} \norm{\nabla\cdot\left(a(\cdot)\mathbf{e}_i\right)}_{L^2\left(K\right)} ,
		\end{multline}
		\begin{multline}
			\norm{v^i}_{L^2\left((0,t)\times\Delta\right)} 
			\le \left(\frac{\abs{\Delta_K}}{\abs{K}}\right)^{1/2} \norm{v^i}_{L^2\left((0,t)\times K\right)} \\
			\le \left(\frac{\abs{\Delta_K}}{\abs{K}}\right)^{1/2} \frac{C_P}{\sqrt{2\alpha}} \norm{\nabla\cdot\left(a(\cdot)\mathbf{e}_i\right)}_{L^2\left(K\right)},	
		\end{multline}
	\end{subequations}
	where the Poincar\'e-Wirtinger constant $C_P$ is independent of $R$.

Let us now consider two arbitrary constants $b,c$ such that $b\ge \frac{\abs{R-L}}{\abs{2\tilde{R} -R -L}}$ and $0<c \le \nacon \frac{\abs{2\tilde{R} - R -L}^2}{\abs{R-L}^2}$ (since $L<R-2$, we know that $2\tilde{R} -R -L> 0$).
Then,  
\begin{equation}
	\label{eq:aux_const}
	\begin{aligned}
		\frac{1}{\abs{\tilde{R}-L}} &\le \frac{1}{\abs{2\tilde{R} -R -L}} \le \frac{b}{\abs{R-L}},\quad\text{and}\\
		e^{-\nacon\frac{\abs{\tilde{R}-L}^2}{  t}} &\le e^{-\nacon\frac{\abs{2\tilde{R}-R-L}^2}{  t}} \le e^{-c\frac{\abs{R-L}^2}{  t}} ,
	\end{aligned}
\end{equation}
where $\nacon$ is the constant in \eqref{eq:NashAronson} and $t>0$.%
By using \eqref{eq:NashAronson}, the first term in \eqref{eq: initial estimate theta} is estimated as
\begin{multline}
	\label{eq: norm L2 green}
	\norm{G(\cdot,t;\bx)}_{L^2\left(\Delta\right)} 
	\le \left[\int_{\Delta} \frac{C^2}{t^{d}} e^{-2\nacon\frac{\abs{\bx-\by}^2}{t}}\,d\by\right]^{\frac{1}{2}}\\
	\le  \frac{C \abs{\Delta}^{1/2}}{t^{d/2}} e^{-\nacon\frac{\abs{\tilde{R}-L}^2}{  t}}
	\le  \frac{C \abs{\Delta}^{1/2}}{t^{d/2}} e^{-c\frac{\abs{R-L}^2}{  t}},
\end{multline}
since $\abs{\bx-\by}\ge\abs{\tilde{R}-L}$ for $\bx\in K_L$ and $\by \in\Delta$, and thanks to \eqref{eq:aux_const}. Likewise, the $L^2\left((0,t)\times\Delta\right)$-norm of the fundamental solution is bounded by
\begin{equation*}
	\norm{G(\cdot,\cdot;\bx)}_{L^2\left((0,t)\times\Delta\right)} 
	\le C\abs{\Delta}^{1/2} \left[\int_{0}^{t} \frac{1}{s^d} e^{-2\nacon\frac{\abs{\tilde{R}-L}^2}{ s}}\,ds\right]^{\frac{1}{2}}.
\end{equation*}
The time integral can be bounded in different ways, depending on the dimension $d$: If $d\le 2$ it is straightforward to prove that
\begin{equation*}
	\int_{0}^{t} \frac{1}{s^d} e^{-2\nacon\frac{\abs{\tilde{R}-L}^2}{ s}}\,ds \le \frac{1}{2\nacon\abs{\tilde{R}-L}^2} t^{2-d} e^{-2\nacon\frac{\abs{\tilde{R}-L}^2}{t}},
\end{equation*}
while, when $d\ge 3$, the integral can be computed exactly\footnote{%
	The function $F(s) = \left(\frac{1}{2\nacon\abs{\tilde{R}-L}^2}\right)^{d-1} \left[ \sum_{k=0}^{d-2} \frac{(d-2)!}{k!} \left(2\nacon\frac{\abs{\tilde{R}-L}^2}{ t}\right)^k \right] e^{-2\nacon\frac{\abs{\tilde{R}-L}^2}{ t}}$ is the primitive of $f(s) = s^{-d}e^{-2\nacon\frac{\abs{\tilde{R}-L}^2}{ s}}$.%
}
and it can be bounded as
\begin{equation*}
	\int_{0}^{t} \frac{1}{s^d} e^{-2\nacon\frac{\abs{\tilde{R}-L}^2}{ s}}\,ds 
	\le
	\begin{cases}
		(d-2)! \left(\frac{1}{2\nacon\abs{\tilde{R}-L}^2}\right)^{d-1}, &  \text{ if } \frac{2\nacon\abs{\tilde{R}-L}^2}{ t} \le 1,\\
		\frac{ \kappa}{2\nacon\abs{\tilde{R}-L}^2} t^{2-d} e^{-2\nacon\frac{\abs{\tilde{R}-L}^2}{ t}},& \text{ if }\frac{2\nacon\abs{\tilde{R}-L}^2}{ t} > 1,
	\end{cases}
\end{equation*}
for some $\kappa>0$.
By assuming that $t\le 2\nacon \abs{\tilde{R}-L}^2$ and by \eqref{eq:aux_const}, one concludes that there exists $C>0$, independent of $R$ and $L$, such that:
\begin{equation}
	\label{eq: norm L2L2 green}
	\begin{aligned}
		\norm{G(\cdot,\cdot;\bx)}_{L^2\left((0,t)\times\Delta\right)} 
		&\le  \frac{  C \abs{\Delta}^{1/2}}{\sqrt{2\nacon}\abs{\tilde{R}-L} } \frac{1}{t^{d/2-1}} e^{-\nacon\frac{\abs{\tilde{R}-L}^2}{ t}}\\
		&\le \frac{  C \abs{\Delta}^{1/2}b}{\sqrt{2\nacon}\abs{R-L} } \frac{1}{t^{d/2-1}} e^{-c\frac{\abs{R-L}^2}{ t}}.
	\end{aligned}
\end{equation}
Finally, we use the result of \Cref{lemma: caccioppoli inequality green} and, again, \eqref{eq:aux_const} to bound the norm of the gradient:
\begin{multline}
	\label{eq: norm L2L2 gradient green}
	\norm{\nabla G(\cdot,\cdot;\bx)}_{L^2\left((0,t)\times\Delta\right)} 
	\le \frac{2M\beta}{\alpha}\norm{G(\cdot,\cdot;\bx)}_{L^2\left((0,t)\times\Delta_2\right)}\\
	\le \frac{2M\beta}{\alpha} \frac{  C \abs{\Delta_2}^{1/2}}{ \sqrt{2\nacon}\abs{2\tilde{R}-R-L} } \frac{1}{t^{d/2-1}} e^{-\nacon\frac{\abs{2\tilde{R}-R-L}^2}{  t}}\\
	\le \frac{2M\beta}{\alpha} \frac{  C \abs{\Delta_2}^{1/2} b }{ \sqrt{2\nacon}\abs{R-L} } \frac{1}{t^{d/2-1}} e^{-c\frac{\abs{R-L}^2}{  t}}.
\end{multline}
Moreover, we know that $\abs{\Delta}\le\abs{\Delta_K}$ and $\abs{\Delta_2}\le2\abs{\Delta_K}$. Thus, by plugging estimates \cref{eq:estimates norms v,eq: norm L2 green,eq: norm L2L2 green,eq: norm L2L2 gradient green} into \cref{eq: initial estimate theta}, and by assuming that the periodic cell is the unit square, $\abs{K}=1$, one gets
%
\begin{multline*}
	\abs{\theta^i(\bx,t)} \le 
	\frac{C \abs{\Delta}^{1/2}}{t^{d/2}} e^{-c\frac{\abs{R-L}^2}{  t}} 
	\left(\frac{\abs{\Delta_K}}{\abs{K}}\right)^{1/2} \norm{\nabla\cdot\left(a(\cdot)\mathbf{e}_i\right)}_{L^2\left(K\right)} \\
	+ C_{\rho}\beta 
	\frac{  C \abs{\Delta}^{1/2}b}{\sqrt{2\nacon}\abs{R-L} } \frac{1}{t^{d/2-1}} e^{-c\frac{\abs{R-L}^2}{ t}} 
	\left(\frac{\abs{\Delta_K}}{\abs{K}}\right)^{1/2} \frac{1}{\sqrt{2\alpha}} \norm{\nabla\cdot\left(a(\cdot)\mathbf{e}_i\right)}_{L^2\left(K\right)} \\
	+ C_{\rho}\beta 
	\frac{2M\beta}{\alpha} \frac{  C \abs{\Delta_2}^{1/2} b }{ \sqrt{2\nacon}\abs{R-L} } \frac{1}{t^{d/2-1}} e^{-c\frac{\abs{R-L}^2}{  t}}
	\left(\frac{\abs{\Delta_K}}{\abs{K}}\right)^{1/2} \frac{C_P}{\sqrt{2\alpha}} \norm{\nabla\cdot\left(a(\cdot)\mathbf{e}_i\right)}_{L^2\left(K\right)}\\
	\le \left[ 1 + C_{\rho}\beta \frac{  b t}{\sqrt{2\nacon}\abs{R-L} }  \frac{1}{\sqrt{2\alpha}} +  C_{\rho}\beta 
	\frac{2M\beta}{\alpha} \frac{ \sqrt{2} b t}{ \sqrt{2\nacon}\abs{R-L} }  \frac{C_P}{\sqrt{2\alpha}} \right] \times \\ 
	\times \frac{ C \abs{\Delta_K} \norm{\nabla\cdot\left(a(\cdot)\mathbf{e}_i\right)}_{L^2\left(K\right)}}{t^{d/2}}  e^{-c\frac{\abs{R-L}^2}{  t}}\\
	\le 
	\left(1+C_2 \frac{t}{\abs{R-L}}\right) 
	\frac{C_1 R^{d-1} \norm{\nabla \cdot\left(a(\cdot)\mathbf{e}_i\right)}_{L^2(K)}}{t^{d/2}}
	e^{-c\frac{\abs{R-L}^2}{t}},
\end{multline*}
with constants $C_1>0$ and $\displaystyle C_2:=  \frac{C_{\rho}\beta b}{2\sqrt{\alpha\nacon}}
\left(1 + 2\sqrt{2}MC_P\frac{\beta}{\alpha}\right)$.

\end{proof}%
}

\section*{Acknowledgements}
This work is partially supported by the Swiss National Science Foundation, grant No. $200020\_172710$.

	
	\bibliographystyle{plain}
	\bibliography{anmc}	

\begin{thebibliography}{10}

\bibitem{Abd02}
Assyr Abdulle.
\newblock Fourth order {C}hebyshev methods with recurrence relation.
\newblock {\em SIAM J. Sci. Comput.}, 23(6):2041--2054, 2002.

\bibitem{Abd05b}
Assyr Abdulle.
\newblock On a priori error analysis of fully discrete heterogeneous multiscale
  {FEM}.
\newblock {\em Multiscale Model. Simul.}, 4(2):447--459, 2005.

\bibitem{Abd09a}
Assyr Abdulle.
\newblock The finite element heterogeneous multiscale method: a computational
  strategy for multiscale {PDE}s.
\newblock In {\em Multiple scales problems in biomathematics, mechanics,
  physics and numerics}, volume~31 of {\em GAKUTO Internat. Ser. Math. Sci.
  Appl.}, pages 133--181. Gakk{\=o}tosho, Tokyo, 2009.

\bibitem{Abd11b}
Assyr Abdulle.
\newblock A priori and a posteriori error analysis for numerical
  homogenization: a unified framework.
\newblock {\em Ser. Contemp. Appl. Math. CAM}, 16:280--305, 2011.

\bibitem{AAP19a}
Assyr Abdulle, Doghonay Arjmand, and Edoardo Paganoni.
\newblock Exponential decay of the resonance error in numerical homogenization
  via parabolic and elliptic cell problems.
\newblock {\em C. R. Math. Acad. Sci. Paris}, 357:545--551, 2019.

\bibitem{AEE12}
Assyr Abdulle, Weinan E, Bj{\"o}rn Engquist, and Eric Vanden-Eijnden.
\newblock The heterogeneous multiscale method.
\newblock {\em Acta Numer.}, 21:1--87, 2012.

\bibitem{AbM01}
Assyr Abdulle and Alexei~A. Medovikov.
\newblock Second order {C}hebyshev methods based on orthogonal polynomials.
\newblock {\em Numer. Math.}, 90(1):1--18, 2001.

\bibitem{AbV11b}
Assyr Abdulle and Gilles Vilmart.
\newblock Analysis of the finite element heterogeneous multiscale method for
  quasilinear elliptic homogenization problems.
\newblock {\em Math. Comp.}, 83(286):513--536, 2014.

\bibitem{ArR16b}
Doghonay Arjmand and Olof Runborg.
\newblock {A time dependent approach for removing the cell boundary error in
  elliptic homogenization problems}.
\newblock {\em J. Comput. Phys.}, 314(Supplement C):206--227, 2016.

\bibitem{ArR16}
Doghonay Arjmand and Olof Runborg.
\newblock Estimates for the upscaling error in heterogeneous multiscale methods
  for wave propagation problems in locally-periodic media.
\newblock ArXiv e-print 1605.02386, 2016.

\bibitem{ArR17}
Doghonay Arjmand and Olof Runborg.
\newblock Estimates for the upscaling error in heterogeneous multiscale methods
  for wave propagation problems in locally periodic media.
\newblock {\em Multiscale Model. Simul.}, 15(2):948--976, 2017.

\bibitem{ArS16}
Doghonay Arjmand and Christian Stohrer.
\newblock A finite element heterogeneous multiscale method with improved
  control over the modeling error.
\newblock {\em Communications in Mathematical Sciences}, 14(2):463--487, 2016.

\bibitem{AKM19}
Scott Armstrong, Tuomo Kuusi, and Jean-Christophe Mourrat.
\newblock {\em Quantitative Stochastic Homogenization and Large-Scale
  Regularity}.
\newblock Springer, 2019.

\bibitem{Aro68}
D.~G. Aronson.
\newblock Non-negative solutions of linear parabolic equations.
\newblock {\em Ann. Sc. Norm. Super. Pisa Cl. Sci.}, Ser. 3, 22(4):607--694,
  1968.

\bibitem{BLP78}
Alain Bensoussan, Jacques-Louis Lions, and George Papanicolaou.
\newblock {\em Asymptotic analysis for periodic structures}.
\newblock North-Holland Publishing Co., Amsterdam, 1978.

\bibitem{BlL10}
Xavier Blanc and Claude Le~Bris.
\newblock {Improving on computation of homogenized coefficients in the periodic
  and quasi-periodic settings}.
\newblock {\em Netw. Heterog. Media}, 5(1):1--29, 2010.

\bibitem{Bre10}
Haim Brezis.
\newblock {\em Functional analysis, Sobolev spaces and partial differential
  equations}.
\newblock Springer Science \& Business Media, 2010.

\bibitem{CiD99}
Doina Cioranescu and Patrizia Donato.
\newblock {\em An introduction to homogenization}, volume~17 of {\em Oxford
  Lecture Series in Mathematics and its Applications}.
\newblock Oxford University Press, New York, 1999.

\bibitem{E11}
Weinan E.
\newblock {\em Principles of multiscale modeling}.
\newblock Cambridge University Press, Cambridge, 2011.

\bibitem{EE03}
Weinan E and Bj{\"o}rn Engquist.
\newblock The heterogeneous multiscale methods.
\newblock {\em Commun. Math. Sci.}, 1(1):87--132, 2003.

\bibitem{EMZ05}
Weinan E, Pingbing Ming, and Pingwen Zhang.
\newblock Analysis of the heterogeneous multiscale method for elliptic
  homogenization problems.
\newblock {\em J. Amer. Math. Soc.}, 18(1):121--156, 2005.

\bibitem{Eva10}
Lawrence~C. Evans.
\newblock {\em Partial differential equations}, volume~19 of {\em Graduate
  Studies in Mathematics}.
\newblock American Mathematical Society, Providence, RI, second edition, 2010.

\bibitem{GeN88}
A.~George and E.~Ng.
\newblock On the complexity of sparse {$QR$} and {$LU$} factorization of
  finite-element matrices.
\newblock {\em SIAM J. Sci. and Stat. Comput.}, 9(5):849--861, 1988.

\bibitem{Glo11}
Antoine Gloria.
\newblock Reduction of the resonance error. {P}art 1: {A}pproximation of
  homogenized coefficients.
\newblock {\em Math. Models Methods Appl. Sci.}, 21(8):1601--1630, 2011.

\bibitem{GlH16}
Antoine Gloria and Zakaria Habibi.
\newblock Reduction in the resonance error in numerical homogenization ii:
  Correctors and extrapolation.
\newblock {\em Found. Comput. Math.}, 16(1):217--296, Feb 2016.

\bibitem{GNO15}
Antoine Gloria, Stefan Neukamm, and Felix Otto.
\newblock {Quantification of ergodicity in stochastic homogenization: optimal
  bounds via spectral gap on Glauber dynamics}.
\newblock {\em Inventiones mathematicae}, 199(2):455--515, 2015.

\bibitem{GlO15}
Antoine Gloria and Felix Otto.
\newblock The corrector in stochastic homogenization: optimal rates, stochastic
  integrability, and fluctuations.
\newblock arXiv:1510.08290.

\bibitem{GlO11a}
Antoine Gloria and Felix Otto.
\newblock An optimal variance estimate in stochastic homogenization of discrete
  elliptic equations.
\newblock {\em Ann. Probab.}, 39(3):779--856, 05 2011.

\bibitem{GlO12a}
Antoine Gloria and Felix Otto.
\newblock An optimal error estimate in stochastic homogenization of discrete
  elliptic equations.
\newblock {\em Ann. Appl. Probab.}, 22(1):1--28, 02 2012.

\bibitem{HeM14}
Patrick Henning and Axel M{\aa}lqvist.
\newblock Localized orthogonal decomposition techniques for boundary value
  problems.
\newblock {\em SIAM J. Sci. Comput.}, 36(4):A1609--A1634, 2014.

\bibitem{HoW97}
Thomas~Y. Hou and Xiao-Hui Wu.
\newblock A multiscale finite element method for elliptic problems in composite
  materials and porous media.
\newblock {\em J. Comput. Phys.}, 134(1):169--189, 1997.

\bibitem{HWC99}
Thomas~Y. Hou, Xiao-Hui Wu, and Zhiqiang Cai.
\newblock Convergence of a multiscale finite element method for elliptic
  problems with rapidly oscillating coefficients.
\newblock {\em Math. Comp.}, 68(227):913--943, 1999.

\bibitem{HFM98}
Thomas~J.R. Hughes, Gonzalo~R. Feij{\'o}o, Luca Mazzei, and Jean-Baptiste
  Quincy.
\newblock The variational multiscale method -- a paradigm for computational
  mechanics.
\newblock {\em Comput. Methods Appl. Mech. Engrg.}, 166(1):3 -- 24, 1998.
\newblock Advances in Stabilized Methods in Computational Mechanics.

\bibitem{JKO94}
Vasilii~V. Jikov, Sergei~M. Kozlov, and Olga~A. Oleinik.
\newblock {\em Homogenization of differential operators and integral
  functionals}.
\newblock Springer-Verlag, Berlin, Heidelberg, 1994.

\bibitem{KGH03}
Ioannis~G. Kevrekidis, C.~William Gear, James~M. Hyman, Panagiotis~G.
  Kevrekidis, Olof Runborg, and Constantinos Theodoropoulos.
\newblock Equation-free, coarse-grained multiscale computation: enabling
  microscopic simulators to perform system-level analysis.
\newblock {\em Commun. Math. Sci.}, 1(4):715--762, 2003.

\bibitem{LiM68}
Jacques-Louis Lions and Enrico Magenes.
\newblock {\em {Probl\`{e}mes aux limites non homog\`{e}nes et applications}},
  volume~1 of {\em Travaux et recherches math\'{e}matiques}.
\newblock Dunod, Paris, 1968.

\bibitem{MaP14}
Axel M{\aa}lqvist and Daniel Peterseim.
\newblock Localization of elliptic multiscale problems.
\newblock {\em Math. Comp.}, 83(290):2583--2603, 2014.

\bibitem{Mou19}
Jean-Christophe Mourrat.
\newblock Efficient methods for the estimation of homogenized coefficients.
\newblock {\em Found. Comput. Math.}, 19(2):435--483, Apr 2019.

\bibitem{PaW60}
L.~E. Payne and H.~F. Weinberger.
\newblock An optimal {P}oincar\'e inequality for convex domains.
\newblock {\em Arch. Rational Mech. Anal.}, 5:286--292, 1960.

\bibitem{Saa03}
Yousef Saad.
\newblock {\em Iterative methods for sparse linear systems}, volume~82.
\newblock SIAM, 2003.

\bibitem{Spa68}
S~Spagnolo.
\newblock {Sulla convergenza di soluzioni di equazioni paraboliche ed
  ellittiche}.
\newblock {\em Ann. Sc. Norm. Super. Pisa Cl. Sci.}, 22(4):571--597, 1968.

\bibitem{HoS80}
Pieter van~der Houwen and Ben~P. Sommeijer.
\newblock On the internal stage {R}unge-{K}utta methods for large m-values.
\newblock {\em Z. Angew. Math. Mech.}, 60:479--485, 1980.

\bibitem{YuE07}
Xingye Yue and Weinan E.
\newblock The local microscale problem in the multiscale modeling of strongly
  heterogeneous media: effects of boundary conditions and cell size.
\newblock {\em J. Comput. Phys.}, 222(2):556--572, 2007.

\end{thebibliography}
\end{document}